\documentclass[a4paper,11pt]{amsart}

\pdfoutput=1 
\usepackage{ifpdf}

\usepackage[hidelinks,pdfpagelabels]{hyperref}

\usepackage[british]{babel} 
\usepackage[latin1]{inputenc} 
\usepackage[T1]{fontenc} 

\usepackage{amsmath,amssymb,amsfonts,amsthm}
\usepackage{bbm}
\usepackage[svgnames]{xcolor}
\usepackage[all,2cell,cmtip]{xy} 
\UseAllTwocells

\allowdisplaybreaks

\numberwithin{equation}{section}

\usepackage{enumitem}

\theoremstyle{plain}
\newtheorem{teo}{Theorem}[section]
\newtheorem{prop}[teo]{Proposition}
\newtheorem{lem}[teo]{Lemma}
\newtheorem{cor}[teo]{Corollary}
\newtheorem*{claim}{Claim}
\newtheorem*{teo*}{Theorem}
\theoremstyle{remark}

\newtheorem{rem}[teo]{Remark}

\theoremstyle{definition}
\newtheorem{defi}[teo]{Definition}
\newtheorem*{notat}{Notation}

\DeclareMathOperator{\Set}{\mathsf{Set}}

\DeclareMathOperator{\Opmon}{\mathsf{OpMon}}

\DeclareMathOperator{\OplaxAct}{\mathsf{OplaxAct}}
\DeclareMathOperator{\Cat}{\mathsf{Cat}}
\DeclareMathOperator{\Mod}{\mathsf{Mod}}

\DeclareMathOperator{\Span}{\mathsf{Span}}
\DeclareMathOperator{\id}{\mathrm{id}}
\DeclareMathOperator{\Ob}{\mathsf{Ob}}
\DeclareMathOperator{\nLax}{\mathsf{nLax}}
\DeclareMathOperator{\tr}{\mathsf{tr}}
\DeclareMathOperator{\sk}{\mathsf{sk}}
\DeclareMathOperator{\cosk}{\mathsf{cosk}}

\newcommand{\ot}{^{\circ}\!}
\newcommand{\ob}{_{\circ}\!}
\renewcommand{\qed}{\hfill$\blacksquare$}

\newenvironment{frcseries}{\fontfamily{frc}\selectfont}{}
\newcommand{\textfrc}[1]{{\frcseries#1}}
\newcommand{\mathfrc}[1]{\textnormal{\textfrc{#1}}}
\newcommand{\cs}{\mathfrc{s}}

\usepackage{etoolbox}
\makeatletter 
\mathchardef\latex@simeq\simeq
\let\simeq\relax
\DeclareRobustCommand{\simeq}{\mathrel{\mathpalette\new@simeq\relax}}
\newcommand{\new@simeq}[2]{%
  \raisebox{\simeq@raise{#1}}{$\m@th#1\latex@simeq$}%
}
\newcommand{\simeq@raise}[1]{%
  \ifx#1\displaystyle .425\fontdimen22\textfont2\fi
  \ifx#1\textstyle .425\fontdimen22\textfont2\fi
  \ifx#1\scriptstyle .425\fontdimen22\scriptfont2\fi
  \ifx#1\scriptscriptstyle .425\fontdimen22\scriptscriptfont2\fi
}
\patchcmd{\@vereq}{.5}{0}{}{}
\makeatother 

\title{Monads of oplax actions are skew monoidales}
\author{Ramón Abud Alcalá}
\date{\today}
\thanks{The results in this paper are included in the second chapter of my PhD thesis \emph{Oplax actions and enriched icons with applications to coalgebroids and quantum categories} which was written under the supervision of Steve Lack. I want to thank Steve Lack for all the support and guidance he gave me to finalise this paper.
}
\begin{document}
\begin{abstract}
Szlachányi showed that bialgebroids can be characterised using skew monoidal categories. The characterisation reduces the amount of data, structure, and properties required to define them. Lack and Street provide a bicategorical account of that same fact; they characterise quantum categories in terms of skew monoidal structures internal to a monoidal bicategory. A quantum category is an opmonoidal monad on an enveloping monoidale $R\ot\otimes R$ in a monoidal bicategory. In a previous paper, we characterised opmonoidal arrows on enveloping monoidales as a simpler structure called oplax action.

This is the second paper based on the author's PhD thesis. Here, motivated by the fact that opmonoidal monads are monads in the bicategory of monoidales, opmonoidal arrows, and opmonoidal cells; we prove that right skew monoidales are ``monads of oplax actions''. To do so, we arrange oplax actions as the 1-simplices of a simplicial object in $\Cat$. In nice cases this simplicial object is ought to be thought as a bicategory whose arrows are oplax actions, that is to say, it is weakly equivalent to a nerve of a bicategory. We define monads of oplax actions as simplicial maps out of the Catalan simplicial set and prove that these are in bijective correspondence with right skew monoidales whose unit has a right adjoint, no assumptions required on the ambient monoidal bicategory.
\end{abstract}
\maketitle

\section{Introduction}
Oplax actions arise from the study of coalgebroids in the generalised context of a monoidal bicateogry \cite{Abud2018}. The original definition of a coalgebroid was stated by Takeuchi in \cite[Definition 3.5]{Takeuchi1987}, roughly; for $R$ and $S$ two algebras over a commutative ring $k$, an $R|S$-coalgebroid $C$ is a module in $(R\ot\otimes R)$-$\Mod$-$(S\ot\otimes S)$ together with some additional structure including a comonoid structure on $C$ when it is seen as a module in $S$-$\Mod$-$S$. In \cite{Szlachanyi2004}, Szlachányi showed that $R|S$-coalgebroids are the arrows of a bicategory whose monads are $R$-bialgebroids. He also characterized $R$-bialgebroids in \cite{Szlachanyi2003} as cocontinuous opmonoidal monads on the monoidal category $R$-$\Mod$-$R$. Lack and Street realised in \cite{Lack2012} that bialgebroids may also be viewed as opmonoidal monads on the enveloping monoidale $R\ot\otimes R$ in the monoidal bicategory $\Mod_k$ of $k$-algebras, modules, and module morphisms. In \cite{Szlachanyi2012} there is a further characterization of bialgebroids as closed skew monoidal structures on the category $\Mod$-$R$ with $R$ as the unit object. This skew monoidal categories are categories together with product and unit functors and, witnessing asssociativity and unitality, three not necessarily invertible natural transformations that satisfy five axioms similar to those for monoidal categories. Following the same idea, Lack and Street provide in \cite[Theorem 5.2]{Lack2012} an equivalence between opmonoidal monads on enveloping monoidales and skew monoidal structures whose unit has a right adjoint in the framework of a monoidal bicategory. These monoidal bicategories must satisfy some mild conditions that we gather under the name \emph{opmonadic-friendly monoidal bicategories} in Definition~\ref{def:OpmonadicFriendly}. In fact, the existence of an opmonadic adjunction $\xymatrix@1@C=5mm{i\ob\dashv i\ot:I\ar[r]&R\ot}$ is a key ingredient of Lack and Street's equivalence. This equivalence may be described in two steps. First, one sends an opmonoidal monad $\xymatrix@1@C=5mm{R\ot\otimes R\ar[r]&R\ot\otimes R}$ (and all its structure cells) to its transpose $\xymatrix@1@C=5mm{R\otimes R\ot\otimes R\ar[r]&R}$ under the biduality $R\dashv R\ot$. And second, one precomposes such an arrow (and the corresponding structure cells) with the opmonadic left adjoint $\xymatrix@1@C=5mm{1\otimes i\ob\otimes 1:R\otimes R\ar[r]&R\otimes R\ot\otimes  R}$ to obtain a right skew monoidal structure on $R$ with product $\xymatrix@1@C=5mm{R\otimes R\ar[r]&R}$ given by the composition one just obtained and with unit the transpose of $i\ot$. The inverse process requires to use first the universal property of the same opmonadic left adjoint and then the transposition along the aforementioned biduality. Now, since opmonoidal arrows between enveloping monoidales in the bicategory $\Mod_k$ are the coalgebroids of Takeuchi; when we mention coalgebroids in the generalised context of a monoidal bicategory what we really mean are opmonoidal arrows on an enveloping monoidale. In \cite[Corollary 6.11]{Abud2018}, we characterize opmonoidal arrows between enveloping monoidales using a similar process than that of Lack and Street, what we obtain are oplax actions. Oplax actions in a monoidal bicategory are a notion of action with respect to a right skew monoidale where the associative and unit laws are witnessed by cells that are not necessarily invertible, and satisfy further coherence conditions.

This paper builds up on \cite{Abud2018} and it is based on the second chapter of the author's PhD thesis \cite{Abud2017}. Here, we analyse how oplax actions in the equivalence of \cite[Corollary 6.11]{Abud2018} take the place of right skew monoidales in the equivalence of \cite[Theorem 5.2]{Lack2012}. Now, since opmonoidal monads in a monoidal bicategory $\mathcal{M}$ are monads in the bicategory $\Opmon(\mathcal{M})$ of monoidales, opmonoidal arrows, and opmonoidal cells; we expect that ``right skew monoidales are monads of oplax actions whose unit has a right adjoint''. As the title suggests, this motto motivates the whole paper. Our objective is to make the motto into a precise mathematical statement and show that it is true in any monoidal bicateogry. In Section~\ref{sec:Preliminaries} we recall some definitions from formal monoidal category theory. We also show that if the ambient monoidal bicategory $\mathcal{M}$ is opmonadic-friendly, we can form a bicategory of oplax actions by simply copying the bicategory structure from $\Opmon(\mathcal{M})$. In this case, the motto becomes true by construction but we do not have an explicit description of a monad of oplax actions since we are relying on the universal property of opmonadicity. Furthermore, there is no easy way to define a bicategory whose arrows $\xymatrix@1@C=5mm{R\ar[r]|-@{|}&S}$ are oplax actions $\xymatrix@1@C=5mm{SR\ar[r]&S}$ for skew monoidales $R$ and $S$ when $\mathcal{M}$ is not opmonadic-friendly. We solve this problem in Section~\ref{sec:OplaxActions} by defining a simplicial object $\OplaxAct(\mathcal{M})$ in $\Cat$ whose 1-simplices are oplax actions in $\mathcal{M}$. The 2-simplices may then be thought of as encoding generalised horizontal composites of oplax actions. The fact that this simplicial object is 2-coskeletal means that simplicial maps $\xymatrix@1@C=5mm{\mathbb{C}\ar[r]&\OplaxAct(\mathcal{M})}$ out of the Catalan simplicial set may be interpreted as monads in $\OplaxAct(\mathcal{M})$. That is because monads in a bicategory $\mathcal{B}$ are in bijection with simplicial maps $\xymatrix@1@C=5mm{\mathbb{C}\ar[r]&N\mathcal{B}}$ into the 1-nerve of the bicategory $\mathcal{B}$ \cite{Buckley2014}. In Section~\ref{sec:MonadsOfOplaxActions} we prove that these simplicial-style monads of oplax actions are in bijection with right skew monoidales in $\mathcal{M}$ whose unit has a right adjoint, no additional assumptions required on the monoidal bicategory $\mathcal{M}$. So far, we have only used the underlying simplicial set of the simplicial object in $\Cat$ of oplax actions. The 2-dimensional part of the structure is used in the last section as we shall explain. Now, simplicial objects in $\Cat$ are organised in a 2-category $[\Delta^{\mathrm{op}},\Cat]$, which apart from the usual notions of equality, isomorphism, and equivalence that exist in any 2-category, there is a notion of \emph{weak equivalence}. 
In the same way that there is a nerve construction which assigns to each category a simplicial set, there are many different nerve-like constructions that assign to each bicategory a simplicial object in $\Cat$ or in $\Set$ \cite{Carrasco2010}. We are interested in what we call the \emph{lax-2-nerve}. We conclude this paper with Section~\ref{sec:OplaxActionsOpmonoidalArrows}, where we prove that if $\mathcal{M}$ is an opmonadic-friendly monoidal bicategory, our simplicial object in $\Cat$ of oplax actions is weakly equivalent to the lax-2-nerve of the full subbicategory of $\Opmon(\mathcal{M})$ on the enveloping monoidales in $\mathcal{M}$.
\section{Preliminaries}\label{sec:Preliminaries}
We follow the same notation and conventions as in \cite{Abud2018}. Our universe of discourse is a monoidal bicategory $\mathcal{M}$, we denote its tensor product by juxtaposition and its unit object by $I$. Opmonadic adjunctions in $\mathcal{M}$ play an important role, especially when these behave well with respect to the monoidal structure of $\mathcal{M}$.
\begin{defi}
An adjunction $\xymatrix@1@C=5mm{f\dashv g:S\ar[r]&R}$ (or a left adjoint) is called \emph{opmonadic} (or of \emph{Kleisli type}), if  $R$ is a Kleisli object for the monad $t$ associated to the adjunction $f\dashv g$. That is to say, if for every object $X$ the adjunction obtained by applying the representable functor $\mathcal{M}(\_,X)$ is monadic in $\Cat$ in the up to equivalence sense.
\[
\vcenter{\hbox{\xymatrix{
R\dtwocell_{f}^{g}{'\dashv}\\
S\ar@(ru,rd)^-{t}
}}}
\qquad\qquad\qquad
\vcenter{\hbox{\xymatrix{ \mathcal{M}(R,X)\dtwocell_{\mathcal{M}(f,X)\hspace{9mm}}^{\hspace{9mm}\mathcal{M}(g,X)}{'\dashv}\\
\mathcal{M}(S,X)\ar@(ru,rd)[]!<6mm,0mm>;[]!<6mm,0mm>^-{\mathcal{M}(t,X)}
}}}
\]
\end{defi}
\begin{defi}\label{def:OpmonadicFriendly}
An \emph{opmonadic-friendly monoidal bicategory} is a monoidal bicategory such that
\begin{itemize}
\item Tensoring with objects on either side preserves opmonadicity.
\item Composing with arrows on either side preserves any existing reflexive coequaliser in the hom categories.
\end{itemize}
\end{defi}

An object $R$ has a \emph{right bidual} $R\ot$ if there are unit $\xymatrix@1@C=5mm{n:I\ar[r]&R\ot R}$ and counit $\xymatrix@1@C=5mm{e:RR\ot\ar[r]&I}$ arrows in $\mathcal{M}$ satisfying the triangle equations up to coherent isomorphism; this situation is called \emph{biaduality} and it is denoted by $R\dashv R\ot$.

\begin{defi}\label{def:SkewMonoidale}
A \emph{right skew monoidale} in $\mathcal{M}$ consists of an object $M$, a product arrow $\xymatrix@1@C=5mm{m:MM\ar[r]&M}$, a unit arrow $\xymatrix@1@C=5mm{u:I\ar[r]&M}$, an associator cell $\alpha$, a left unitor cell $\lambda$, and a right unitor cell $\rho$ (not necessarily invertible),
\[
\vcenter{\hbox{\xymatrix@!0@=15mm{
MMM\ar[r]^-{m1}\ar[d]_-{1m}\xtwocell[rd]{}<>{^\alpha}&MM\ar[d]^-{m}\\
MM\ar[r]_-{m}&M
}}}
\qquad
\vcenter{\hbox{\xymatrix@!0@=15mm{
M\ar[r]^-{u1}\ar[rd]_-{1}\xtwocell[rd]{}<>{^<-2>\lambda}&MM\ar[d]|-{m}&M\ar[l]_-{1u}\ar[ld]^-{1}\xtwocell[ld]{}<>{^<2>\rho}\\
&M&
}}}
\]
satisfying five axioms: in the same order as \cite[Section 4]{Lack2012} we will refer to them as, the pentagon \eqref{ax:SKM1}, the triangle \eqref{ax:SKM2}, \eqref{ax:SKM3}, \eqref{ax:SKM4}, and \eqref{ax:SKM5}.
\end{defi}

If $\alpha$, $\lambda$, and $\rho$ are isomorphisms we speak of a \emph{monoidale} and using the same argument as in \cite{Kelly1964} the pentagon \eqref{ax:SKM1} and the triangle \eqref{ax:SKM2} axioms imply the other three axioms. A biduality $R\dashv R$ induces a monoidal structure on the object $R\ot R$ with product $\xymatrix@1@C=5mm{1e1:R\ot RR\ot R\ar[r]&R\ot R}$, unit $n$, and whose associative and unit laws are satisfied up to coherent isomorphism. Monoidales induced by bidualities are referred to as \emph{enveloping monoidales}. We recall \cite[Lemma 4.12]{Abud2018}, this is an example of a skew monoidale which in general is not a monoidale.

\begin{lem}[Abud Alcalá]\label{lem:OneRightSkewMonoidale}
For every adjunction $\xymatrix@1@C=5mm{i\dashv i^*:I\ar[r]&R}$ there is a right skew monoidal structure on $R$ given as follows:
\begin{center}
\begin{tabular}{rlrl}
\centering
Product&
$
\xymatrix{
RR\ar[r]^-{i^*1}&R
}
$&
Unit&
$
\xymatrix{
I\ar[r]^-{i}&R
}$
\\
Associator&
$
\vcenter{\hbox{\xymatrix@!0@=15mm{
RRR\ar[r]^-{i^*11}\ar[d]_-{1i^*1}&RR\ar[d]^-{i^*1}\\
RR\ar[r]_-{i^*1}\ar@{}[ru]|-*[@ru]{\cong}&R
}}}
$&
Unitors&
$
\vcenter{\hbox{\xymatrix@!0{
R\ar[rr]^-{i1}\ar[rrdd]_-{1}\xtwocell[rrdd]{}<>{^<-2>\eta1}&&RR\ar[dd]|-{i^*1}\ar@{}[rd]|-<<<<*[@rd]{\cong}&&R\ar[ll]_-{1i}\ar[ld]^-{i^*}\ar@/^1cm/[lldd]^-{1}\xtwocell[lldd]{}<>{^<-3>\varepsilon}\\
&&&I\ar[ld]^-{i}&\\
&&R&&
}}}
$
\end{tabular}
\end{center}
\end{lem}

\begin{defi}
An \emph{opmonoidal arrow} $\xymatrix@1@C=5mm{C:M\ar[r]&N}$ between right skew monoidales $M$ and $N$ in $\mathcal{M}$ consists of an arrow $\xymatrix@1@C=5mm{C:M\ar[r]&N}$ in $\mathcal{M}$ equipped with an opmonoidal composition constraint cell $C^2$ and an opmonoidal unit constraint cell $C^0$ as shown below,
\[
\vcenter{\hbox{\xymatrix@!0@=15mm{
MM\ar[r]^-{CC}\ar[d]_-{m}\xtwocell[rd]{}<>{^C^2}&NN\ar[d]^-{m}\\
M\ar[r]_-{C}&N
}}}
\qquad\qquad
\vcenter{\hbox{\xymatrix@!0@R=15mm@C=10mm{
&I\ar[ld]_-{u}\ar[rd]^-{u}\xtwocell[rd]{}<>{^<3>C^0}&\\
M\ar[rr]_-{C}&&N
}}}
\]
satisfying three axioms, in the same order as \cite[Definition 4.8]{Abud2018}, they are referred to as \eqref{ax:OM1}, \eqref{ax:OM2}, and \eqref{ax:OM3}.
\end{defi}
\begin{defi}
An \emph{opmonoidal cell} between a parallel pair of opmonoidal arrows $C$ and $\xymatrix@1@C=5mm{C':M\ar[r]&N}$ in $\mathcal{M}$ consists of a cell $\xi$ as shown,
\[
\vcenter{\hbox{\xymatrix{
M\ar@/_3mm/[r]_-{C}\ar@/^3mm/[r]^-{C'}\xtwocell[r]{}<>{^\xi}&N
}}}
\]
satisfying two axioms, in the same order as \cite[Definition 4.10]{Abud2018}, they are referred to as \eqref{ax:OM4} and \eqref{ax:OM5}.
\end{defi}
There is a bicategory $\Opmon(\mathcal{M})$ whose objects are monoidales in $\mathcal{M}$, arrows are opmonoidal arrows between them, and cells are opmonoidal cells between them. When $\mathcal{M}$ is right autonomous, i.e. right biduals exist for every object of $\mathcal{M}$, we may consider the full subbicategory $\Opmon^{\mathrm{e}}(\mathcal{M})$ of $\Opmon(\mathcal{M})$ consisting of the enveloping monoidales $R\ot R$ induced by objects $R$ and a chosen right bidual $R\ot$.

\begin{defi}
\emph{An oplax $M$-action on $A$} in a monoidal bicategory $\mathcal{M}$ consists of an arrow $\xymatrix@1@C=5mm{a:AM\ar[r]&A}$, an associator cell $a^2$, and a right unitor cell $a^0$,
\[
\vcenter{\hbox{\xymatrix@!0@=15mm{
AMM\ar[r]^-{a1}\ar[d]_-{1m}\xtwocell[rd]{}<>{^a^2}&AM\ar[d]^-{a}\\
AM\ar[r]_-{a}&A
}
\qquad
\xymatrix@!0@=15mm{
AM\ar[d]_-{a}&A\ar[l]_-{1u}\ar[dl]^-{1}\xtwocell[ld]{}<>{^<2>\ \ a^0}\\
A
}}}
\]
satisfying three axioms, in the same order as \cite[Definition 6.1]{Abud2018}, they are referred to as \eqref{ax:OLA1}, \eqref{ax:OLA2}, and \eqref{ax:OLA3}.
\end{defi}
\begin{defi} Let $a$ and $a'$ be oplax right $M$-actions on $A$. A \emph{cell of oplax right $M$-actions on $A$} from $a$ to $a'$ consists of a cell $\varphi$ in $\mathcal{M}$
\[
\vcenter{\hbox{\xymatrix{
**[l]AM\ar@/^3mm/[r]^-{a'}\ar@/_3mm/[r]_-{a}\xtwocell[r]{}<>{^\varphi}&A
}}}
\]
satisfying two axioms, in the same order as \cite[Definition 6.5]{Abud2018}, they are referred to as \eqref{ax:OLA4} and \eqref{ax:OLA5}.
\end{defi}
Oplax $M$-actions on an object $A$ form a category $\OplaxAct(M;A)$ and there is a forgetful functor $\xymatrix@1@C=5mm{\OplaxAct(M;A)\ar[r]&\mathcal{M}(AM,A)}$ which takes an oplax action to its underlying arrow.

In the introduction we recalled a sketch of the proof of \cite[Corollary 6.11]{Abud2018} (see Theorem~\ref{teo:Opmon_is_OplaxAct_Local} below). This theorem establishes an equivalence of categories between certain categories of oplax actions and hom categories of opmonoidal arrows between enveloping monoidales; it holds true provided that $\mathcal{M}$ is an opmonadic-friendly monoidal bicategory in the sense of Definition~\ref{def:OpmonadicFriendly}. The oplax actions involved in this equivalence are on an object $S$ that has a right bidual $S\ot$ and with respect to a right skew monoidale $R$ whose unit $\xymatrix@1@C=5mm{i:I\ar[r]&R}$ has a right adjoint $i^*$. In fact, the right skew monoidal structure on $R$ has product $\xymatrix@1@C=5mm{i^*1:RR\ar[r]&R}$ and unit $i$, this structure is precisely the one induced by the adjunction $i\dashv i^*$ as in Lemma~\ref{lem:OneRightSkewMonoidale}. Furthermore, $R$ must have a right bidual $R\ot$ as an object of $\mathcal{M}$, and the opposite adjunction $i\ob\dashv i\ot$ of the adjunction $i\dashv i^*$ must be opmonadic.  

\begin{teo}[Abud Alcalá]\label{teo:Opmon_is_OplaxAct_Local}
Let $\mathcal{M}$ be an opmonadic-friendly autonomous monoidal bicategory and let $i\ob\dashv i\ot$ be an opmonadic adjunction as shown,
\[
\vcenter{\hbox{\xymatrix{
R\ot\xtwocell[d]{}_{i\ob}^{i\ot}{'\dashv}\\
I
}}}
\]
the following are equivalent:
\begin{enumerate}
\item An opmonoidal arrow $\xymatrix@1@C=5mm{R\ot R\ar[r]&S\ot S}$ between enveloping monoidales.
\item An oplax right action $\xymatrix@1@C=5mm{SR\ar[r]&S}$, with respect to the skew monoidal structure on $R$ induced as in Lemma~\ref{lem:OneRightSkewMonoidale} by the adjunction $i\dashv i^*$ opposite to $i\ob\dashv i\ot$.
\end{enumerate}
In other words, for every two bidualities $R\dashv R\ot$ and $S\dashv S\ot$ there is an equivalence of categories
\[
\Opmon(R\ot R,S\ot S)\simeq\OplaxAct(R;S)
\]
obtained by taking the transpose along the biduality $S\dashv S\ot$ and then precomposing with the opmonadic left adjoint $\xymatrix@1@C=5mm{1i\ob 1:SR\ar[r]&SR\ot R}$.
\end{teo}\qed

This equivalence exposes an ``arrow-like'' essence for oplax actions, which justifies the notation for the categories $\OplaxAct(M;A)$. Following this idea, we depict oplax actions and cells of oplax actions with arrows and double arrows with a dash in the middle. This also distinguishes them from those of $\mathcal{M}$.
\[
\vcenter{\hbox{\xymatrix@!0@=8mm{
M\ar@/^4mm/[rr]|@{|}\ar@/_4mm/[rr]|@{|}\xtwocell[rr]{}<>{^}\ar@{}@<-.5mm>[rr]|{-}&&A
}}}
\]

Hence, it is also worthwhile to figure out when the categories of oplax actions are the hom categories of a bicategory. A simple answer is to globally require the hypothesis of Theorem~\ref{teo:Opmon_is_OplaxAct_Local}. In which case we call the resulting bicategory $\underline{\OplaxAct}(\mathcal{M})$; the underline is used to distinguish between the bicategory of oplax actions and the simplicial object in $\Cat$ of oplax actions that we shall define in the next section.

\begin{teo}\label{teo:Opmon_is_OplaxAct_Global}
Let $\mathcal{M}$ be a right autonomous opmonadic-friendly monoidal bicategory such that every object $R$ has a chosen adjunction $i\dashv i^*$ whose opposite adjunction $i\ob\dashv i\ot$ is opmonadic.
\[
\vcenter{\hbox{\xymatrix{
R\ot\dtwocell_{i\ob}^{i\ot}{'\dashv}\\
I
}}}
\]
There exists a \emph{bicategory of oplax actions} $\underline{\OplaxAct}(\mathcal{M})$ whose objects are those of $\mathcal{M}$ equipped with the skew monoidal structure induced by their chosen adjunction, and with $\OplaxAct(R;S)$ as its hom categories. Moreover, there is a biequivalence of bicategories
\[
\Opmon^{\mathrm{e}}(\mathcal{M})\simeq\underline{\OplaxAct}(\mathcal{M})
\]
\end{teo}
\begin{proof}

Composition and identities are calculated by going back and forth along instances of the equivalence of Theorem~\ref{teo:Opmon_is_OplaxAct_Local}. Thus, for an object $R$ in $\mathcal{M}$ the identity functor on $R$ in $\underline{\OplaxAct}(\mathcal{M})$ is given as follows,
\[
\vcenter{\hbox{\xymatrix{
\mathbbm{1}\ar[r]^-{\id}&\Opmon(R\ot R,R\ot R)\simeq\OplaxAct(R;R),
}}}
\]
its image is the regular oplax action $\xymatrix@1@C=5mm{i^*1:RR\ar[r]&R}$ of the right skew monoidale structure on $R$ induced by the chosen adjunction $i\dashv i^*$, see Lemma~\ref{lem:OneRightSkewMonoidale}. For three objects $R$, $S$, and $T$ of $\mathcal{M}$, the composition functor is given as follows.
\begin{multline*}
\OplaxAct(S;T)\times\OplaxAct(R;S)\simeq\Opmon(S\ot S,T\ot T)\times\Opmon(R\ot R,S\ot S)\\
\vcenter{\hbox{\xymatrix{
\ar[r]&\Opmon(R\ot R,T\ot T)\simeq\OplaxAct(R;T)
}}}
\end{multline*}
Unfortunately, there is no explicit description of the composition functor since the first equivalence relies on the existential part of the universal property of opmonadic adjunctions.
\end{proof}

At this point, we want to draw our attention to some facts regarding monads in $\Opmon^{\mathrm{e}}(\mathcal{M})$ and, consequently, in $\underline{\OplaxAct}(\mathcal{M})$. First, we know that monads in $\Opmon(\mathcal{M})$ are by definition opmonoidal monads in $\mathcal{M}$. Second, opmonoidal monads on enveloping monoidales are precisely the quantum categories of Day and Street in the case $\mathcal{M}=\Mod(\mathcal{V})$, see \cite[Proposition 3.3 and Section 12]{Day2003a}. And third, in Theorem~\ref{teo:Opmon_is_OplaxAct_Global} we used the same technique that Lack and Street used to characterise opmonoidal monads on enveloping monoidales in \cite[Theorem~5.2]{Lack2012}, see below.

\begin{teo}[Lack-Street]\label{teo:OpmonMnd_is_SkewMon}
Let $\mathcal{M}$ be an opmonadic-friendly monoidal bicategory. For every biduality $R\dashv R\ot$ and every opmonadic adjunction
\[
\vcenter{\hbox{\xymatrix{
R\ot\dtwocell_{i\ob}^{i\ot}{'\dashv}\\
I
}}}
\]
the following are equivalent:
\begin{itemize}
\item An opmonoidal monad on an enveloping monoidale $\xymatrix@1@C=5mm{R\ot R\ar[r]&R\ot R}$.
\item A right skew monoidal structure on $R$ with skew unit $\xymatrix@1@C=5mm{i:I\ar[r]&R}$ the opposite of $i\ot$.
\end{itemize}
\end{teo}

Hence, under the hypotheses of Theorem~\ref{teo:Opmon_is_OplaxAct_Global}, we automatically have a description of the monads in $\underline{\OplaxAct}(\mathcal{M})$.

\begin{cor}\label{cor:MonadsOfOplaxActions}
For every right autonomous opmonadic-friendly monoidal bicategory $\mathcal{M}$ such that every object $R$ has a chosen adjunction $i\dashv i^*$ whose opposite adjunction $i\ob\dashv i\ot$ is opmonadic;
\[
\vcenter{\hbox{\xymatrix{
R\ot\dtwocell_{i\ob}^{i\ot}{'\dashv}\\
I
}}}
\]
monads in $\underline{\OplaxAct}(\mathcal{M})$ are right skew monoidales whose unit has a right adjoint.
\end{cor}\qed

An explicit horizontal composition that does not require to go back and forth along the equivalence of Theorem~\ref{teo:Opmon_is_OplaxAct_Global} would allow us to give an elementary description of monads in $\underline{\OplaxAct}(\mathcal{M})$. In consequence, it would help us to understand which of the various parts of a right skew monoidale correspond to an ``underlying'' oplax action and which other parts play the roles of the multiplication and the unit of a monad of oplax actions. Hence, in what follows we avoid using that $\mathcal{M}$ is opmonadic-friendly. In which case, we do not have a bicategory but a simplicial object in $\Cat$ of oplax actions.

\section{A simplicial object in \texorpdfstring{$\Cat$}{Cat} of Oplax Actions}\label{sec:OplaxActions}
\subsection{Simplicial objects in \texorpdfstring{$\Cat$}{Cat}} Let us recall some of the basic definitions regarding simplicial objects in $\Cat$ and settle the notation that we shall be using. The 2-category $[\Delta^{\mathrm{op}},\Cat]$, where the square brackets denote \mbox{(2-)functors}, \mbox{(2-)natural} transformations, and modifications, is the 2-category of \emph{simplicial objects in $\Cat$}, simplicial morphisms, and simplicial transformations between them. A simplicial object $\mathbb{X}$ in $\Cat$ consists of categories $\mathbb{X}_n$ for each natural number, face functors $\partial$, and degeneracy functors $\cs$ that satisfy the usual simplicial equations \cite[pp. 179]{MacLane1997}.
\[
\cdots
\vcenter{\hbox{\xymatrix@!0@=30mm{
\mathbb{X}_3\ar@<9mm>[r]|-{\partial_0}\ar@<3mm>[r]|-{\partial_1}\ar@<-3mm>[r]|-{\partial_2}\ar@<-9mm>[r]|-{\partial_3}&\mathbb{X}_2\ar@<6mm>[r]|-{\partial_0}\ar[r]|-{\partial_1}\ar@<-6mm>[r]|-{\partial_2}\ar@<-6mm>[l]|-{\cs_0}\ar[l]|-{\cs_1}\ar@<6mm>[l]|-{\cs_2}&\mathbb{X}_1\ar@<3mm>[r]|-{\partial_0}\ar@<-3mm>[r]|-{\partial_1}\ar@<-3mm>[l]|-{\cs_0}\ar@<3mm>[l]|-{\cs_1}&\mathbb{X}_0\ar[l]|-{\cs_0}
}}}
\]
The category of sets may be thought as a locally discrete 2-category, and as such, there is a 2-functor $\xymatrix@1@C=5mm{\Set\ar[r]&\Cat}$ that sends each set to its corresponding discrete category. It induces a forgetful 2-functor which takes a simplicial object in $\Cat$ to the simplicial set that keeps only the set of objects of each category of simplices in each dimension. We refer to it as the \emph{underlying simplicial set} of a simplicial object in $\Cat$.
\[
\vcenter{\hbox{\xymatrix@!0@R=6mm@C=30mm{
[\Delta^{\mathrm{op}},\Cat]\ar[r]&[\Delta^{\mathrm{op}},\Set]\\
\mathbb{X}\ar@{|->}[r]&\mathbb{X}^{(0)}
}}}
\]
The term \emph{simplicial category} is commonly used in the literature to refer to a category enriched in simplicial sets, in general these are different to simplicial objects in $\Cat$. In fact, categories enriched in simplicial sets are in bijection with those simplicial objects in $\Cat$ whose underlying simplicial set is a constant functor $\xymatrix@1@C=5mm{\Delta^{\mathrm{op}}\ar[r]&\Set}$.

Let $\Delta_{\leq n}$ be the standard $n$-simplex viewed as a category, the inclusion $\xymatrix@1@C=5mm{\Delta_{\leq n}\ar[r]&\Delta}$ induces the $n$-truncation functor
\[
\vcenter{\hbox{\xymatrix{
\tr_n:[\Delta^{\mathrm{op}},\Cat]\ar[r]&[\Delta_{\leq n}^{\mathrm{op}},\Cat]
}}}
\]
which forgets about the $m$-simplices for all $m>n$ in a simplicial object in $\Cat$. It has a left adjoint and a right adjoint which are called the \emph{$n$-skeleton} and \emph{$n$-coskeleton}, these may be calculated via right and left Kan extensions.
\[
\sk_n\dashv\tr_n\dashv\cosk_n
\]
One may also describe them inductively; for all $m\leq n$ the categories of $m$-simplices of both the $n$-skeleton and $n$-coskeleton are the equal to $\mathbb{X}_m$. And for $m>n$ the $n$-skeleton freely adds all degenerate $m$-simplices for all $(m-1)$-simplices in $\sk_n(\mathbb{X})_{(m-1)}$, and the $n$-coskeleton freely adds as its $m$-simplices unique fillers for all $(m-1)$-spheres in $\cosk_n(\mathbb{X})_{(m-1)}$; thus $(\sk_n\mathbb{X})_m\subset(\cosk_n\mathbb{X})_m$. We say that a simplicial object in $\Cat$ is \emph{$n$-(co)skeletal} if it is isomorphic to its own $n$-(co)skeleton in  $[\Delta^{\mathrm{op}},\Cat]$.

The 2-category of simplicial objects in $\Cat$, apart from having the notions of equality, isomorphism, and equivalence, has a concept of weak equivalence.
\begin{defi}
Let be $\mathbb{X}$ and $\mathbb{Y}$ two simplicial objects in $\Cat$. A simplicial morphism $\xymatrix@1@C=5mm{F:\mathbb{X}\ar[r]&\mathbb{Y}}$ is a \emph{weak equivalence} if for every $n$ the component of $F$ at $n$
\[
\xymatrix{\mathbb{X}_n\ar[r]_-{\simeq}^-{F_n}&\mathbb{Y}_n}
\]
is an equivalence of categories.
\end{defi}

To have a weak equivalence $\xymatrix@1@C=5mm{F:\mathbb{X}\ar[r]&\mathbb{Y}}$ of simplicial objects in $\Cat$ is not the same as having an equivalence in the 2-category $[\Delta^{\mathrm{op}},\Cat]$, because although it is true that for every $n$ we have pseudoinverses $\xymatrix@1@C=5mm{G_n:\mathbb{Y}_n\ar[r]&\mathbb{X}_n}$, in general these do not constitute a simplicial morphism, but only a pseudonatural transformation. So in other words, a simplicial morphism $F$ is a weak equivalence if it is an equivalence when viewed in the bicategory $[\Delta^{\mathrm{op}},\Cat]_{\mathrm{ps}}$ of 2-functors, pseudonatural transformations, and modifications.

Let $\nLax$ be the 2-category of bicategories, normal lax functors, and icons. The following definition may be found in \cite[Definition 5.2]{Carrasco2010}.
\begin{defi}
The \emph{lax-2-nerve} of a bicategory $\mathcal{B}$ is the simplicial object in $\Cat$ defined by
\[
\vcenter{\hbox{\xymatrix@!0@R=6mm@C=20mm{
\Delta^{\mathrm{op}}\ar[r]&\Cat\\
[n]\ar@{|->}[r]&\nLax([n],\mathcal{B})
}}}
\]
\end{defi}
The underlying simplicial set of the lax-2-nerve of a bicategory is the nerve of a bicategory (also called the \emph{Street-nerve} or \emph{1-nerve}), and it is 3-coskeletal. One may see that the lax-2-nerve is also 3-coskeletal by noticing that the category of 3-simplices is the full subcategory of the 3-coskeleton on the 3-simplices of the 1-nerve.

\subsection{Simplices of Oplax Actions}\label{subsec:OplaxActions}
In an attempt to give a more explicit description of the bicategory $\underline{\OplaxAct}(\mathcal{M})$ found in Corollary~\ref{cor:MonadsOfOplaxActions} without relying on the structure of $\Opmon(\mathcal{M})$, the first thing one tries is to define a composition of oplax actions, which is a functor
\[
\vcenter{\hbox{\xymatrix{
\OplaxAct(S;T)\times\OplaxAct(R;S)\ar[rr]^-{?}&&\OplaxAct(R;T)
}}}
\]
whose argument takes two oplax actions $\xymatrix@1@C=5mm{SR\ar[r]^-{s}&S}$ and $\xymatrix@1@C=5mm{TS\ar[r]^-{t}&T}$. Unfortunately these oplax actions do not seem to compose in any straightforward way. This leaves us unable to access a horizontal composition of oplax actions directly, which exists in the case where $\mathcal{M}$ satisfies the hypotheses of Theorem~\ref{teo:Opmon_is_OplaxAct_Global}. To get around this problem, instead of a bicategory of oplax actions, we describe a simplicial object in $\Cat$ of oplax actions. We denote it by $\OplaxAct$ (without the underline), or $\OplaxAct(\mathcal{M})$ if one wishes to specify the ambient bicategory. This simplicial object in $\Cat$ has two important properties, which solve our initial problem of relating Theorems~\ref{teo:Opmon_is_OplaxAct_Global} and \ref{teo:OpmonMnd_is_SkewMon}:
\begin{enumerate}[label={(\arabic*)}]
\item There is an appropriate notion of monad in a simplicial set ---which for the 1-nerve of a bicategory $\mathcal{B}$ is precisely a monad in $\mathcal{B}$--- and a monad in the underlying simplicial set of the simplicial object in $\Cat$ of oplax actions $\OplaxAct(\mathcal{M})^{(0)}$ is precisely a right skew monoidale whose unit has a right adjoint.
\item It is weakly equivalent to the lax-2-nerve of a bicategory of opmonoidal arrows, assuming that $\mathcal{M}$ is a right autonomous opmonadic-friendly monoidal bicategory with chosen right biduals $R\ot$ and chosen left adjoints $\xymatrix@1@C=5mm{i:I\ar[r]&R}$ for every object $R$.
\end{enumerate}
We know that point (1) is true in the case where $\mathcal{M}$ satisfies the hypothesis of Corollary~\ref{cor:MonadsOfOplaxActions}, but point (1) is true in a far more general context: for every monoidal bicategory $\mathcal{M}$ with no extra assumptions whatsoever.

An entirely valid question at this point is: What is the benefit of using simplicial objects in $\Cat$ over mere simplicial sets? And while it is true that for point (1) one only needs simplicial sets and the 1-nerve of a bicategory, for point (2) this is not the case. The extra dimension is crucial to express the equivalences (and not mere isomorphisms) such as the one from Theorem~\ref{teo:Opmon_is_OplaxAct_Local}, which is basically the case of 1-simplices. To avoid working with this extra categorical dimension on our simplicial objects, one might use a strict monoidal 2-category instead of a monoidal bicategory (or a Gray monoid) and with opmonadicity up-to-isomorphism instead of up-to-equivalence. But we do not intend to restrict in this way, as this will rule out our primary examples $\mathcal{M}=\Mod_k$, $\mathcal{M}=\Span^{\mathrm{co}}$ or more generally $\mathcal{M}=\Mod(\mathcal{V})$ for a symmetric monoidal category $\mathcal{V}$.

The rest of this section is dedicated to the definition of $\OplaxAct(\mathcal{M})$. It is a 3-coskeletal simplicial object in $\Cat$, so it is enough to define the first four categories of simplices and their respective face and degeneracy functors.
\[
\vcenter{\hbox{\xymatrix@!0@=30mm{
\OplaxAct_3\ar@<9mm>[r]|-{\partial_0}\ar@<3mm>[r]|-{\partial_1}\ar@<-3mm>[r]|-{\partial_2}\ar@<-9mm>[r]|-{\partial_3}&\OplaxAct_2\ar@<6mm>[r]|-{\partial_0}\ar[r]|-{\partial_1}\ar@<-6mm>[r]|-{\partial_2}\ar@<-6mm>[l]|-{\cs_0}\ar[l]|-{\cs_1}\ar@<6mm>[l]|-{\cs_2}&\OplaxAct_1\ar@<3mm>[r]|-{\partial_0}\ar@<-3mm>[r]|-{\partial_1}\ar@<-3mm>[l]|-{\cs_0}\ar@<3mm>[l]|-{\cs_1}&\OplaxAct_0\ar[l]|-{\cs_0}
}}}
\]

\subsubsection{0-simplices}
\begin{defi}
The category of 0-simplices $\OplaxAct_0$ is the discrete category whose objects are pairs $(R,i\dashv i^*)$ of an object $R$ and an adjunction $i\dashv i^*$ in $\mathcal{M}$.
\[
\vcenter{\hbox{\xymatrix{
R\dtwocell_{i}^{i^*}{'\dashv}\\
I
}}}
\]
\end{defi}

\subsubsection{1-simplices}
\begin{defi}
The category of 1-simplices $\OplaxAct_1$ has as objects all oplax actions on the 0-simplices, and as arrows all cells of oplax actions. Explicitly, $\OplaxAct_1$ is the coproduct in $\Cat$ over the set of pairs of 0-simplices $((R,i\dashv i^*),(S,j\dashv j^*))$ of the categories of oplax actions $\OplaxAct(R;S)$.
\[
\OplaxAct_1:=\coprod_{\substack{(R,i\dashv i^*)\\(S,j\dashv j^*)}}\OplaxAct(R;S)
\]
\end{defi}

The picture of oplax actions with a dashed arrow $\xymatrix@1@C=5mm{r:R\ar[r]|-@{|}&S}$ was done on purpose to resemble the geometrical shape of a standard 1-simplex, this shows easily what the \emph{face functors} are: $\partial_0(r)=(S,j\dashv j^*)$ and $\partial_1(r)=(R,i\dashv i^*)$. The \emph{degeneracy functor} is defined for a 0-simplex $(R,i\dashv i^*)$ as the regular oplax $R$-action on $R$ of the right skew monoidal structure on $R$ induced by the adjunction $i\dashv i^*$, see Lemma~\ref{lem:OneRightSkewMonoidale}. In short, $\cs_0(R)$ is  $\xymatrix@1@C=5mm{i^*1:RR\ar[r]&R}$ with structure cells as shown below.
\[
\vcenter{\hbox{\xymatrix@!0@=15mm{
RRR\ar[r]^-{i^*11}\ar[d]_-{1i^*1}&RR\ar[d]^-{i^*1}\\
RR\ar[r]_-{i^*1}\ar@{}[ru]|-*[@ru]{\cong}&R
}}}
\qquad
\vcenter{\hbox{\xymatrix@!0{
RR\ar[dd]_-{i^*1}\ar@{}[rd]|-<<<<*[@rd]{\cong}&&R\ar[ll]_-{1i}\ar[dl]^-{i^*}\ar@/^1cm/[ddll]^-{1}\xtwocell[ddll]{}<>{^<-3>\varepsilon}\\
&I\ar[dl]^-i&\\
R&&
}}}
\]

\subsubsection{2-simplices}
For pedagogical purposes relevant in Section~\ref{sec:OplaxActionsOpmonoidalArrows} we first define categories of 2-simplices with fixed 0-faces.

\begin{defi}
Given three 0-simplices $(R,i\dashv i^*)$, $(S,j\dashv j^*)$, and $(T,k\dashv k^*)$, denote by $\OplaxAct(R;S;T)$ the category whose objects are quadruples $(s,t,v,\alpha)$, where $\xymatrix@1@C=5mm{s:SR\ar[r]&S}$, $\xymatrix@1@C=5mm{t:TS\ar[r]&T}$, and $\xymatrix@1@C=5mm{v:TR\ar[r]&T}$ are oplax actions; and $\alpha$ is a square in $\mathcal{M}$ that we depict as a double arrow with a dash inside a triangle of dashed arrows,
\[
\vcenter{\hbox{\xymatrix@!0@C=12mm@R=5mm{
R\ar[rd]_-{s}|-@{|}\ar@/^4mm/[rr]^-{v}|-@{|}\xtwocell[rr]{}<>{^\alpha}\ar@{}@<-.5mm>[rr]|{-}&&T\\
&S\ar[ru]_-{t}|-@{|}&
}}}
\qquad:=\qquad
\vcenter{\hbox{\xymatrix@!0@=15mm{
TSR\ar[r]^-{t1}\ar[d]_-{1s}\xtwocell[rd]{}<>{^\alpha}&TR\ar[d]^-{v}\\
TS\ar[r]_-{t}&T
}}}
\]
satisfying the following three axioms.
\begin{align}
\tag{2SIM1}\label{ax:2SIM1}
\vcenter{\hbox{\xymatrix@!0@C=15mm{
&TSR\ar[rd]^-{t1}&\\
TSSR\ar[ru]^-{t11}\ar[dd]_-{11s}\ar[rd]_-{1j^*11}\xtwocell[rr]{}<\omit>{^t^21\ }&&TR\ar[dd]^-{v}\\
&TSR\ar[dd]^-{1s}\ar[ru]^-{t1}\xtwocell[rd]{}<\omit>{^\alpha}&\\
TSS\ar[rd]_-{1j^*1}\ar@{}[ru]|*[@]{\cong}&&T\\
&TS\ar[ru]_-{t}&
}}}
\quad&=\quad
\vcenter{\hbox{\xymatrix@!0@C=15mm{
&TSR\ar[rd]^-{t1}\ar[dd]_-{1s}\xtwocell[rddd]{}<\omit>{^\alpha}&\\
TSSR\ar[ru]^-{t11}\ar[dd]_-{11s}&&TR\ar[dd]^-{v}\\
&TS\ar[rd]^-{t}&\\
TSS\ar[rd]_-{1j^*1}\ar[ru]_-{t1}\ar@{}[ruuu]|*[@]{\cong}\xtwocell[rr]{}<\omit>{^t^2}&&T\\
&TS\ar[ru]_-{t}&
}}}
\\
\tag{2SIM2}\label{ax:2SIM2}
\vcenter{\hbox{\xymatrix@!0@C=15mm{
&TRR\ar[rd]^-{v1}&\\
TSRR\ar[ru]^-{t11}\ar[dd]_-{11i^*1}\ar[rd]_-{1s1}\xtwocell[rddd]{}<\omit>{^1s^2}\xtwocell[rr]{}<\omit>{^\alpha 1\ }&&TR\ar[dd]^-{v}\\
&TSR\ar[dd]^-{1s}\ar[ru]^-{t1}\xtwocell[rd]{}<\omit>{^\alpha}&\\
TSR\ar[rd]_-{1s}&&T\\
&TS\ar[ru]_-{t}&
}}}
\quad&=\quad
\vcenter{\hbox{\xymatrix@!0@C=15mm{
&TRR\ar[rd]^-{v1}\ar[dd]_-{1i^*1}\xtwocell[rddd]{}<\omit>{^v^2}&\\
TSRR\ar[ru]^-{t11}\ar[dd]_-{11i^*1}&&TR\ar[dd]^-{v}\\
&TR\ar[rd]^-{v}&\\
TSR\ar[rd]_-{1s}\ar[ru]_-{t1}\ar@{}[ruuu]|*[@]{\cong}\xtwocell[rr]{}<\omit>{^\alpha}&&T\\
&TS\ar[ru]_-{t}&
}}}
\\
\tag{2SIM3}\label{ax:2SIM3}
\vcenter{\hbox{\xymatrix@!0@C=15mm{
&T\ar@/^7mm/[rddd]^-{1}&\\
TS\ar[ru]^-{t}\ar[dd]_-{11i}\ar@/^7mm/[rddd]^-{1}\xtwocell[rddd]{}<\omit>{^1s^0}&&\\
&&\\
TSR\ar[rd]_-{1s}&&T\\
&TS\ar[ru]_-{t}&
}}}
\quad&=\quad
\vcenter{\hbox{\xymatrix@!0@C=15mm{
&T\ar[dd]_-{1i}\ar@/^7mm/[rddd]^-{1}\xtwocell[rddd]{}<\omit>{^v^0}&\\
TS\ar[ru]^-{t}\ar[dd]_-{11i}&&\\
&TR\ar[rd]^-{v}&\\
TSR\ar[rd]_-{1s}\ar[ru]_-{t1}\ar@{}[ruuu]|*[@]{\cong}\xtwocell[rr]{}<\omit>{^\alpha}&&T\\
&TS\ar[ru]_-{t}&
}}}
\end{align}
Arrows in $\OplaxAct(R;S;T)$ are triples $\xymatrix@1@C=5mm{(\sigma,\tau,\nu):(s,t,v,\alpha)\ar[r]&(s',t',v',\alpha')}$ of cells of oplax actions
\[
\vcenter{\hbox{\xymatrix{
R\ar@/^3mm/[r]|-@{|}^-{s'}\ar@/_3mm/[r]|-@{|}_-{s}\xtwocell[r]{}<>{^\sigma}\ar@{}@<-.5mm>[r]|{-}&S
}}}
\qquad
\vcenter{\hbox{\xymatrix{
S\ar@/^3mm/[r]|-@{|}^-{t'}\ar@/_3mm/[r]|-@{|}_-{t}\xtwocell[r]{}<>{^\tau}\ar@{}@<-.5mm>[r]|{-}&T
}}}
\qquad
\vcenter{\hbox{\xymatrix{
R\ar@/^3mm/[r]|-@{|}^-{v'}\ar@/_3mm/[r]|-@{|}_-{v}\xtwocell[r]{}<>{^\nu}\ar@{}@<-.5mm>[r]|{-}&T
}}}
\]
satisfying the following axiom.
\begin{equation}
\tag{2SIM4}\label{ax:2SIM4}
\vcenter{\hbox{\xymatrix@!0@=19mm{
TSR\ar@/^3mm/[r]^-{t'1}\ar@/^3mm/[d]^-{1s'}\ar@/_3mm/[d]_-{1s}\xtwocell[d]{}<>{^1\sigma}\xtwocell[rd]{}<>{^<-2>\alpha'}&TR\ar@/^3mm/[d]^-{v'}\\
TS\ar@/^3mm/[r]^-{t'}\ar@/_3mm/[r]_-{t}\xtwocell[r]{}<>{^\tau}&T
}}}
\quad=\quad
\vcenter{\hbox{\xymatrix@!0@=19mm{
TSR\ar@/^3mm/[r]^-{t'1}\ar@/_3mm/[r]_-{t1}\xtwocell[r]{}<>{^\tau 1\ }\ar@/_3mm/[d]_-{1s}\xtwocell[rd]{}<>{^<2>\alpha}&TR\ar@/^3mm/[d]^-{v'}\ar@/_3mm/[d]_-{v}\xtwocell[d]{}<>{^\nu}\\
TS\ar@/_3mm/[r]_-{t}&T
}}}
\end{equation}
Composition and identities are calculated as in the category below.
\[
\OplaxAct(R;S)\times\OplaxAct(S;T)\times\OplaxAct(R;T)
\]
\end{defi}

\begin{defi}
The category of 2-simplices $\OplaxAct_2$ is the coproduct in $\Cat$ over the set of triples of 0-simplices $((R,i\dashv i^*),(S,j\dashv j^*),(T,k\dashv k^*))$ of the categories $\OplaxAct(R;S;T)$.
\[
\OplaxAct_2:=\coprod_{\substack{(R,i\dashv i^*)\\(S,j\dashv j^*)\\(T,k\dashv k^*)}}\OplaxAct(R;S;T)
\]
\end{defi}

Again, the picture with dashes of a 2-simplex is set to resemble the geometrical shape of a standard 2-simplex and remind us that the \emph{face functors} are given by $\partial_0(\alpha)=t$, $\partial_1(\alpha)=v$, and $\partial_2(\alpha)=s$. For an oplax action $\xymatrix@1@C=5mm{t:S\ar[r]|-@{|}&T}$ the \emph{degeneracy functors} are defined as follows,
\begin{align*}
\vcenter{\hbox{\xymatrix@!0@C=12mm@R=5mm{
S\ar[rd]_-{\cs_0(S)}|-@{|}\ar@/^4mm/[rr]^-{t}|-@{|}\xtwocell[rr]{!<3mm,0mm>}<>{^\cs_0(t)\quad\ }\ar@{}@<-.5mm>[rr]!<3mm,0mm>|{-}&&T\\
&S\ar[ru]_-{t}|-@{|}&
}}}
\;&=\;
\vcenter{\hbox{\xymatrix@!0@=15mm{
TSS\ar[r]^-{t1}\ar[d]_-{1j^*1}\xtwocell[rd]{}<>{^t^2}&TS\ar[d]^-{t}\\
TS\ar[r]_-{t}&T
}}}\\
\vcenter{\hbox{\xymatrix@!0@C=12mm@R=5mm{
S\ar[rd]_-{t}|-@{|}\ar@/^4mm/[rr]^-{t}|-@{|}\xtwocell[rr]{!<3mm,0mm>}<>{^\cs_1(t)\quad\ }\ar@{}@<-.5mm>[rr]!<3mm,0mm>|{-}&&T\\
&T\ar[ru]_-{\cs_0(T)}|-@{|}&
}}}
\;&=\;
\vcenter{\hbox{\xymatrix@!0@=15mm{
TTS\ar[r]^-{k^*11}\ar[d]_-{1t}\ar@{}[d]_-{\phantom{1j^*1}}&TS\ar[d]^-{t}\\
TT\ar[r]_-{k^*1}\ar@{}[ru]|*[@]{\cong}&T
}}}
\end{align*}
and the axioms for a 2-simplex are trivially satisfied. Similarly for a cell of oplax actions $\xymatrix@1@C=5mm{\tau:t\ar@2[r]&t'}$ the degeneracies are $\cs_0(\tau)=(\id_{\cs_0(S)},\tau,\tau)$ and $\cs_1(\tau)=(\tau,\id_{\cs_0(T)},\tau)$.

\subsubsection{3-simplices} In the same way as for 2-simplices, we first define categories of 3-simplices with fixed 0-faces.
\begin{defi}
Given four 0-simplices $(R,i\dashv i^*)$, $(S,j\dashv j^*)$, $(T,k\dashv k^*)$, and $(U,l\dashv l^*)$, define the category $\OplaxAct(R;S;T;U)$ as the full subcategory of
\begin{multline*}
\OplaxAct(R;S;T)\times\OplaxAct(R;S;U)\\\times\OplaxAct(R;T;U)\times\OplaxAct(S;T;U)
\end{multline*}
consisting of the quadruples $\Gamma=(\alpha,\beta,\gamma,\zeta)$ of 2-simplices of oplax actions whose boundaries match appropriately to be arranged in a tetrahedral configuration,
\[
\vcenter{\hbox{\xymatrix@!0@C=8mm@R=10mm{
R\ar[rd]_-{s}|-@{|}\ar@/^4mm/[rrrd]^-{v}|-@{|}\ar@/^5mm/[rrrr]^-{w}|-@{|}\xtwocell[rrrd]{}<>{^\alpha}\ar@{}@<-.7mm>[rrrd]|*[@]{-}&&{\xtwocell[rr]{}<>{^\gamma}}\ar@{}@<-.5mm>[rr]|{-}&&U\\
&S\ar[rr]_-{t}|-@{|}&&T\ar[ru]_-{u}|-@{|}&
}}}
\vcenter{\hbox{\xymatrix@!0{
\ar@3[r]^-{\Gamma}|-@{|}&
}}}
\vcenter{\hbox{\xymatrix@!0@C=8mm@R=10mm{
R\ar[dr]_-{s}|-@{|}\ar@/^5mm/[rrrr]^-{w}|-@{|}\xtwocell[rrr]{}<>{^\beta}\ar@{}@<-.5mm>[rrr]|{-}&&&&U\\
&S\ar[rr]_-{t}|-@{|}\ar@/^5mm/[rrru]^-{x}|-@{|}\xtwocell[rrru]{}<>{^\zeta}\ar@{}@<-.5mm>[rrru]|*[@]{-}&&T\ar[ru]_-{u}|-@{|}&
}}}
\]
and that satisfy the following ``composing'' condition.
\begin{equation}
\tag{3SIM}\label{ax:3SIM}
\vcenter{\hbox{\xymatrix@!0@C=15mm{
&USR\ar[rd]^-{x1}&\\
UTSR\ar[ru]^-{u11}\ar[dd]_-{11s}\ar[rd]_-{1t1}\xtwocell[rddd]{}<\omit>{^1\alpha\ }\xtwocell[rr]{}<\omit>{^\zeta 1}&&UR\ar[dd]^-{w}\\
&UTR\ar[dd]^-{1v}\ar[ru]^-{u1}\xtwocell[rd]{}<\omit>{^\gamma}&\\
UTS\ar[rd]_-{1t}&&U\\
&UT\ar[ru]_-{u}&
}}}
\quad=\quad
\vcenter{\hbox{\xymatrix@!0@C=15mm{
&USR\ar[rd]^-{x1}\ar[dd]_-{1s}\xtwocell[rddd]{}<\omit>{^\beta}&\\
UTSR\ar[ru]^-{u11}\ar[dd]_-{11s}&&UR\ar[dd]^-{w}\\
&US\ar[rd]^-{x}&\\
UTS\ar[rd]_-{1t}\ar[ru]_-{u1}\ar@{}[ruuu]|*[@]{\cong}\xtwocell[rr]{}<\omit>{^\zeta}&&U\\
&UT\ar[ru]_-{u}&
}}}
\end{equation}
\end{defi}

\begin{defi}
The category of 3-simplices $\OplaxAct_3$ is the coproduct in $\Cat$ over the set of quadruples of 0-simplices $((R,i\dashv i^*),(S,j\dashv j^*),(T,k\dashv k^*),(U,l\dashv l^*))$ of the categories $\OplaxAct(R;S;T;U)$.
\[
\OplaxAct_3:=\coprod_{\substack{(R,i\dashv i^*)\\(S,j\dashv j^*)\\(T,k\dashv k^*)\\(U,l\dashv l^*)}}\OplaxAct(R;S;T;U)
\]
\end{defi}
\begin{rem}
The condition ``to match in a tetrahedral configuration'' in the definition of the categories $\OplaxAct(R;S;T;U)$ means that $\OplaxAct_3$ is not only a full subcategory of quadruples of objects in $\OplaxAct_2$ but a full subcategory of the 3-coskeleton of the 2-truncated simplicial object in $\Cat$ of the simplices of oplax actions defined so far.
\[
\vcenter{\hbox{\xymatrix@!0@=30mm{
\OplaxAct_2\ar@<6mm>[r]|-{\partial_0}\ar[r]|-{\partial_1}\ar@<-6mm>[r]|-{\partial_2}&\OplaxAct_1\ar@<3mm>[r]|-{\partial_0}\ar@<-3mm>[r]|-{\partial_1}\ar@<-3mm>[l]|-{\cs_0}\ar@<3mm>[l]|-{\cs_1}&\OplaxAct_0\ar[l]|-{\cs_0}
}}}
\]
Hence, a quadruple of 2-simplices in this 2-coskeleton is in $\OplaxAct(\mathcal{M})$ if and only if it satisfies condition~\eqref{ax:3SIM}, so in practice we say that a 3-simplex is a quadruple $(\alpha,\beta,\gamma,\zeta)$ satisfying the property \eqref{ax:3SIM}.
\end{rem}
Once more, for a 3-simplex $\Gamma$ as above, the tetrahedral picture with dashes is set up to resemble a standard 3-simplex and to remind us of the fact that the \emph{face functors} are defined on objects as $\partial_0(\Gamma)=\alpha$, $\partial_1(\Gamma)=\beta$, $\partial_2(\Gamma)=\gamma$, and $\partial_3(\Gamma)=\zeta$, and in a similar way on arrows. For a 2-simplex $(s,t,v,\alpha)$ the \emph{degeneracy functors} are defined on objects as: $\cs_0(\alpha)$ is axiom~\eqref{ax:2SIM2} for $\alpha$,
\[
\vcenter{\hbox{\xymatrix@!0@C=8mm@R=10mm{
R\ar[rd]_-{\cs_0(R)}|-@{|}\ar@/^4mm/[rrrd]^-{s}|-@{|}\ar@/^5mm/[rrrr]^-{v}|-@{|}\xtwocell[rrrd]{!<2mm,0mm>}<>{^\cs_0(s)\quad\ }\ar@{}@<-.7mm>[rrrd]!<2mm,0mm>|*[@]{-}&&{\xtwocell[rr]{}<>{^\alpha}}\ar@{}@<-.5mm>[rr]|{-}&&T\\
&R\ar[rr]_-{s}|-@{|}&&S\ar[ru]_-{t}|-@{|}\ar@{}[ru]_-{\phantom{\cs_0(T)}}&
}}}
\vcenter{\hbox{\xymatrix@!0{
\ar@3[r]|-@{|}^-{\cs_0(\alpha)}&
}}}
\vcenter{\hbox{\xymatrix@!0@C=8mm@R=10mm{
R\ar[rd]_-{\cs_0(R)}|-@{|}\ar@/^5mm/[rrrr]^-{v}|-@{|}\xtwocell[rrr]{}<>{^\cs_0(v)\quad}\ar@{}@<-.5mm>[rrr]|{-}&&&&T\\
&R\ar[rr]_-{s}|-@{|}\ar@/^5mm/[rrru]^-{v}|-@{|}\xtwocell[rrru]{}<>{^\alpha}\ar@{}@<-.5mm>[rrru]|*[@]{-}&&S\ar[ru]_-{t}|-@{|}\ar@{}[ru]_-{\phantom{\cs_0(T)}}&
}}}
\]
$\cs_1(\alpha)$ is axiom~\eqref{ax:2SIM1} for $\alpha$,
\[
\vcenter{\hbox{\xymatrix@!0@C=8mm@R=10mm{
R\ar[rd]_-{\phantom{\cs_0(R)}}\ar[rd]_-{s}|-@{|}\ar@/^4mm/[rrrd]^-{s}|-@{|}\ar@/^5mm/[rrrr]^-{v}|-@{|}\xtwocell[rrrd]{!<2mm,0mm>}<>{^\cs_1(s)\quad\ }\ar@{}@<-.7mm>[rrrd]!<2mm,0mm>|*[@]{-}&&{\xtwocell[rr]{}<>{^\alpha}}\ar@{}@<-.5mm>[rr]|{-}&&T\\
&S\ar[rr]_-{\cs_0(s)}|-@{|}&&S\ar[ru]_-{t}|-@{|}\ar@{}[ru]_-{\phantom{\cs_0(R)}}&
}}}
\vcenter{\hbox{\xymatrix@!0{
\ar@3[r]|-@{|}^{\cs_1(\alpha)}&
}}}
\vcenter{\hbox{\xymatrix@!0@C=8mm@R=10mm{
R\ar[rd]_-{\phantom{\cs_0(R)}}\ar[rd]_-{s}|-@{|}\ar@/^5mm/[rrrr]^-{v}|-@{|}\xtwocell[rrr]{}<>{^\alpha}\ar@{}@<-.5mm>[rrr]|{-}&&&&T\\
&S\ar[rr]_-{\cs_0(S)}|-@{|}\ar@/^5mm/[rrru]^-{t}|-@{|}\xtwocell[rrru]{}<>{^\cs_0(t)\quad\ }\ar@{}@<-.5mm>[rrru]|*[@]{-}&&S\ar[ru]_-{t}|-@{|}\ar@{}[ru]_-{\phantom{\cs_0(T)}}&
}}}
\]
$\cs_2(\alpha)$ is an instance of the coherence law for the interchange isomorphisms.
\[
\vcenter{\hbox{\xymatrix@!0@C=8mm@R=10mm{
R\ar[rd]_-{\phantom{\cs_0(R)}}\ar[rd]_-{s}|-@{|}\ar@/^4mm/[rrrd]^-{v}|-@{|}\ar@/^5mm/[rrrr]^-{v}|-@{|}\xtwocell[rrrd]{}<>{^\alpha}\ar@{}@<-.7mm>[rrrd]|*[@]{-}&&{\xtwocell[rr]{}<>{^\cs_1(v)\quad\ }}\ar@{}@<-.5mm>[rr]|{-}&&T\\
&S\ar[rr]_-{t}|-@{|}&&T\ar[ru]_-{\cs_0(T)}|-@{|}&
}}}
\vcenter{\hbox{\xymatrix@!0{
\ar@3[r]|-@{|}^-{\cs_2(\alpha)}&
}}}
\vcenter{\hbox{\xymatrix@!0@C=8mm@R=10mm{
R\ar[rd]_-{\phantom{\cs_0(R)}}\ar[rd]_-{s}|-@{|}\ar@/^5mm/[rrrr]^-{v}|-@{|}\xtwocell[rrr]{}<>{^\alpha}\ar@{}@<-.5mm>[rrr]|{-}&&&&T\\
&S\ar[rr]_-{t}|-@{|}\ar@/^5mm/[rrru]^-{t}|-@{|}\xtwocell[rrru]{}<>{^\cs_1(t)\quad\ }\ar@{}@<-.5mm>[rrru]|*[@]{-}&&T\ar[ru]_-{\cs_0(T)}|-@{|}&
}}}
\]

This finishes the definition of $\OplaxAct(\mathcal{M})$, now we may address points (1) and (2) made at the beginning of this section. This is done in Sections~\ref{sec:MonadsOfOplaxActions} and \ref{sec:OplaxActionsOpmonoidalArrows}.
\section{Monads of Oplax Actions and Skew Monoidales}\label{sec:MonadsOfOplaxActions}
In \cite{Buckley2014}, \cite{Buckley2016}, \cite{Greenspan2015}, and \cite{Bohm2017} the authors classify various monoidal-like notions as simplicial morphisms out of a particular simplicial set $\mathbb{C}$, whose name is \emph{the Catalan simplicial set}. One might think of the Catalan simplicial set as the ``free living monoidal-like structure'' in the sense that one decides which kind of monoidal-like structure to get by choosing different kinds of nerves. For example, simplicial morphisms from $\mathbb{C}$ into the \emph{1-nerve} of a bicategory are in bijection with monads in the bicategory. By using other kinds of nerves as the target of a simplicial morphism one may get monoids in a monoidal category, or monoidal categories in $\Cat$, or skew monoidales in a monoidal bicategory, and so on. The Catalan simplicial set owes its name is to the fact that the number of $n$-simplices is the $n$th Catalan number.
\begin{defi}
The \emph{Catalan simplicial set} $\mathbb{C}$ is the simplicial set that has a unique 0-simplex $\star$, two 1-simplices; one degenerate $\xymatrix@1@C=5mm{s_0(\star)=e:\star\ar[r]&\star}$, and a unique non-degenerate $\xymatrix@1@C=5mm{c:\star\ar[r]&\star}$; and the rest is built by coskeletality.
\end{defi}

The 1-nerve of a bicategory is 3-coskeletal, and so it is reasonable to say:

\begin{defi}
A \emph{monad} in a 3-coskeletal simplicial set $\mathbb{X}$ is a simplicial morphism as shown.
\[
\vcenter{\hbox{\xymatrix{
\mathbb{C}\ar[r]&\mathbb{X}
}}}
\]
\end{defi}

\subsection{Monads of oplax actions}
In the previous section we built a simplicial object in $\Cat$ using oplax actions in $\mathcal{M}$. In this section we are interested in its underlying simplicial set $\OplaxAct(\mathcal{M})^{(0)}$ which is obtained by taking the set of objects in each dimension. Particularly, we are interested in monads in $\OplaxAct(\mathcal{M})^{(0)}$ because, as we shall see in Theorem~\ref{teo:MonadsAreSkewMonoidales} below, these are right skew monoidales in $\mathcal{M}$ whose unit has a right adjoint.
\begin{defi}
A \emph{monad of oplax actions} $(R,i\dashv i^*,r,\mu_0,\mu_2)$ in a monoidal bicategory $\mathcal{M}$ is a monad in the simplicial set $\OplaxAct(\mathcal{M})^{(0)}$. More explicitly, it consists of the following items.
\begin{enumerate}
\item One 0-simplex $(R,i\dashv i^*)$
\item One 1-simplex $\xymatrix@1@C=5mm{r:R\ar[r]|-@{|}&R}$
\item Two 2-simplices
\[
\vcenter{\hbox{\xymatrix@R=1mm{
R\ar[rd]_-{i^*\!1}|-@{|}\ar@/^4mm/[rr]^-{r}|-@{|}\xtwocell[rr]{}<>{^\mu_0\ }\ar@{}@<-.5mm>[rr]|{-}&&R\\
&R\ar[ru]_-{i^*\!1}|-@{|}&
}}}
\qquad
\vcenter{\hbox{\xymatrix@R=1mm{
R\ar[rd]_-{r}|-@{|}\ar@/^4mm/[rr]^-{r}|-@{|}\xtwocell[rr]{}<>{^\mu_2\ }\ar@{}@<-.5mm>[rr]|{-}&&R\\
&R\ar[ru]_-{r}|-@{|}&
}}}\]
\item Three 3-simplices
\begin{center}
\begin{tabular}{r@{}c@{}l}
$
\vcenter{\hbox{\xymatrix@!0@C=8mm@R=10mm{
R\ar[rd]_-{r}|-@{|}\ar@{}[rd]_-{\phantom{s_0(R)}}\ar@/^4mm/[rrrd]^-{r}|-@{|}\ar@/^5mm/[rrrr]^-{r}|-@{|}\xtwocell[rrrd]{}<>{^\mu_2\ }\ar@{}@<-.7mm>[rrrd]|*[@]{-}&&{\xtwocell[rr]{}<>{^\mu_2\ }}\ar@{}@<-.5mm>[rr]|{-}&&R\\
&R\ar[rr]_-{r}|-@{|}&&R\ar[ru]_-{r}|-@{|}\ar@{}[ru]_-{\phantom{s_0(R)}}&
}}}
$
&
$
\vcenter{\hbox{\xymatrix@!0{
\ar@3[r]^{\ref{ax:M1}}|-@{|}&
}}}
$
&
$
\vcenter{\hbox{\xymatrix@!0@C=8mm@R=10mm{
R\ar[rd]_-{r}|-@{|}\ar@{}[rd]_-{\phantom{s_0(R)}}\ar@/^5mm/[rrrr]^-{r}|-@{|}\xtwocell[rrr]{}<>{^\mu_2\ }\ar@{}@<-.5mm>[rrr]|{-}&&&&R\\
&R\ar[rr]_-{r}|-@{|}\ar@/^5mm/[rrru]^-{r}|-@{|}\xtwocell[rrru]{}<>{^\mu_2\ }\ar@{}@<-.5mm>[rrru]|*[@]{-}&&R\ar[ru]_-{r}|-@{|}\ar@{}[ru]_-{\phantom{s_0(R)}}&
}}}
$
\\
$
\vcenter{\hbox{\xymatrix@!0@C=8mm@R=10mm{
R\ar[rd]_-{s_0(R)}|-@{|}\ar@/^4mm/[rrrd]^-{r}|-@{|}\ar@/^5mm/[rrrr]^-{r}|-@{|}\xtwocell[rrrd]{}<>{^\mu_0\ }\ar@{}@<-.7mm>[rrrd]|*[@]{-}&&{\xtwocell[rr]{}<>{^\mu_2\ }}\ar@{}@<-.5mm>[rr]|{-}&&R\\
&R\ar[rr]_-{s_0(R)}|-@{|}&&R\ar[ru]_-{r}|-@{|}\ar@{}[ru]_-{\phantom{s_0(R)}}&
}}}
$
&
$
\vcenter{\hbox{\xymatrix@!0{
\ar@3[r]^{\ref{ax:M2}}|-@{|}&
}}}
$
&
$
\vcenter{\hbox{\xymatrix@!0@C=8mm@R=10mm{
R\ar[rd]_-{s_0(R)}|-@{|}\ar@/^5mm/[rrrr]^-{r}|-@{|}\xtwocell[rrr]{}<>{^s_0(r)\quad\ }\ar@{}@<-.5mm>[rrr]|{-}&&&&R\\
&R\ar[rr]_-{s_0(R)}|-@{|}\ar@/^5mm/[rrru]^-{r}|-@{|}\xtwocell[rrru]{}<>{^s_0(r)\quad\ }\ar@{}@<-.5mm>[rrru]|*[@]{-}&&R\ar[ru]_-{r}|-@{|}&
}}}
$
\\
$
\vcenter{\hbox{\xymatrix@!0@C=8mm@R=10mm{
R\ar[rd]_-{r}|-@{|}\ar@{}[rd]_-{\phantom{s_0(R)}}\ar@/^4mm/[rrrd]^-{r}|-@{|}\ar@/^5mm/[rrrr]^-{r}|-@{|}\xtwocell[rrrd]{!<2mm,0mm>}<>{^s_1(r)\quad\ }\ar@{}@<-.7mm>[rrrd]!<2mm,0mm>|*[@]{-}&&{\xtwocell[rr]{}<>{^s_1(r)\quad\ }}\ar@{}@<-.5mm>[rr]|{-}&&R\\
&R\ar[rr]_-{s_0(R)}|-@{|}&&R\ar[ru]_-{s_0(R)}|-@{|}&
}}}
$
&
$
\vcenter{\hbox{\xymatrix@!0{
\ar@3[r]^{\ref{ax:M3}}|-@{|}&
}}}
$
&
$
\vcenter{\hbox{\xymatrix@!0@C=8mm@R=10mm{
R\ar[rd]_-{r}|-@{|}\ar@{}[rd]_-{\phantom{s_0(R)}}\ar@/^5mm/[rrrr]^-{r}|-@{|}\xtwocell[rrr]{}<>{^\mu_2\ }\ar@{}@<-.5mm>[rrr]|{-}&&&&R\\
&R\ar[rr]_-{s_0(R)}|-@{|}\ar@/^5mm/[rrru]^-{r}|-@{|}\xtwocell[rrru]{}<>{^\mu_0\ }\ar@{}@<-.5mm>[rrru]|*[@]{-}&&R\ar[ru]_-{s_0(R)}|-@{|}&
}}}
$
\end{tabular}
\end{center}
\end{enumerate}
If we fully unpack each item and enumerate in the same order, a monad of oplax actions amounts to the items below.
\begin{enumerate}
\item An object $R$, together with the right skew monoidal structure induced by the adjunction $i\dashv i^*$ as in Lemma~\ref{lem:OneRightSkewMonoidale}.
\item An oplax right $R$-action on $R$ with respect to the right skew monoidal structure in (i); that is, an arrow $\xymatrix@1@C=5mm{r:RR\ar[r]&R}$ with structure cells $r^2$ and $r^0$ that satisfy axioms \eqref{ax:OLA1}, \eqref{ax:OLA2}, and \eqref{ax:OLA3}.
\[
\vcenter{\hbox{\xymatrix@!0@=15mm{
RRR\ar[r]^{r1}\ar[d]_{1i^*\!1}\xtwocell[rd]{}<>{^r^2}&RR\ar[d]^{r}\\
RR\ar[r]_{r}&R
}
\qquad\qquad
\xymatrix@!0@=15mm{
RR\ar[d]_-{r}&S\ar[l]_{1i}\ar[dl]^1\xtwocell[ld]{}<>{^<2>\ r^0}\\R
}}}
\]
\item Two quadrangular cells $\mu_2$ and $\mu_0$ each satisfying instances of the three axioms \eqref{ax:2SIM1}, \eqref{ax:2SIM2}, and \eqref{ax:2SIM3}.
\[
\vcenter{\hbox{\xymatrix@!0@=15mm{
RRR\ar[r]^{r1}\ar[d]_{1r}\xtwocell[rd]{}<>{^\mu_2\ }&RR\ar[d]^{r}\\
RR\ar[r]_{r}&R
}\qquad\qquad\xymatrix@!0@=15mm{
RRR\ar[r]^{i^*\!11}\ar[d]_{1i^*\!1}\xtwocell[rd]{}<>{^\mu_0\ }&RR\ar[d]^{r}\\
RR\ar[r]_{i^*\!1}&R
}}}
\]
\item Three instances of the axiom \eqref{ax:3SIM}.
\begin{align}
\tag{M1}\label{ax:M1}
\vcenter{\hbox{\xymatrix@!0@C=15mm{
&RRR\ar[rd]^-{r1}&\\
RRRR\ar[ru]^-{r11}\ar[dd]_-{11r}\ar[rd]_-{1r1}\xtwocell[rddd]{}<>{^1\mu_2\ \ }\xtwocell[rr]{}<>{^\mu_21\quad}&&RR\ar[dd]^-{r}\\
&RRR\ar[dd]^-{1r}\ar[ru]^-{r1}\xtwocell[rd]{}<>{^\mu_2\ }&\\
RRR\ar[rd]_-{1r}&&R\\
&RR\ar[ru]_-{r}&
}}}
\quad&=\quad
\vcenter{\hbox{\xymatrix@!0@C=15mm{
&RRR\ar[rd]^-{r1}\ar[dd]_-{1r}\xtwocell[rddd]{}<>{^\mu_2\ }&\\
RRRR\ar[ru]^-{r11}\ar[dd]_-{11r}&&RR\ar[dd]^-{r}\\
&RR\ar[rd]^-{r}&\\
RRR\ar[rd]_-{1r}\ar[ru]_-{r1}\ar@{}[ruuu]|*[@]{\cong}\xtwocell[rr]{}<>{^\mu_2\ }&&R\\
&RR\ar[ru]_-{r}&
}}}
\\
\tag{M2}\label{ax:M2}
\vcenter{\hbox{\xymatrix@!0@C=15mm{
&RRR\ar[rd]^-{r1}&\\
RRRR\ar[ru]^-{r11}\ar[dd]_-{11i^*\!1}\ar[rd]_-{1i^*\!11}\xtwocell[rddd]{}<>{^1\mu_0\quad}\xtwocell[rr]{}<>{^r^21\quad}&&RR\ar[dd]^-{r}\\
&RRR\ar[dd]^-{1r}\ar[ru]^-{r1}\xtwocell[rd]{}<>{^\mu_2\ }&\\
RRR\ar[rd]_-{1i^*\!1}&&R\\
&RR\ar[ru]_-{r}&
}}}
\quad&=\quad
\vcenter{\hbox{\xymatrix@!0@C=15mm{
&RRR\ar[rd]^-{r1}\ar[dd]_-{1i^*\!1}\xtwocell[rddd]{}<>{^r^2\ }&\\
RRRR\ar[ru]^-{r11}\ar[dd]_-{11i^*\!1}&&RR\ar[dd]^-{r}\\
&RR\ar[rd]^-{r}&\\
RRR\ar[rd]_-{1i^*\!1}\ar[ru]_-{r1}\ar@{}[ruuu]|*[@]{\cong}\xtwocell[rr]{}<>{^r^2\ }&&R\\
&RR\ar[ru]_-{r}&
}}}\\
\tag{M3}\label{ax:M3}
\vcenter{\hbox{\xymatrix@!0@C=15mm{
&RRR\ar[rd]^-{r1}&\\
RRRR\ar[ru]^-{i^*\!111}\ar[dd]_-{11r}\ar[rd]_-{1i^*\!11}\xtwocell[rr]{}<>{^\mu_01\quad}&&RR\ar[dd]^-{r}\\
&RRR\ar[dd]^-{1r}\ar[ru]^-{i^*\!11}&\\
RRR\ar[rd]_-{1i^*\!1}\ar@{}[ru]|*[@]{\cong}&&R\\
&RR\ar[ru]_-{i^*\!1}\ar@{}[ruuu]|*[@]{\cong}&
}}}
\quad&=\quad
\vcenter{\hbox{\xymatrix@!0@C=15mm{
&RRR\ar[rd]^-{r1}\ar[dd]_-{1r}\xtwocell[rddd]{}<>{^\mu_2}&\\
RRRR\ar[ru]^-{i^*\!111}\ar[dd]_-{11r}&&RR\ar[dd]^-{r}\\
&RR\ar[rd]^-{r}&\\
RRR\ar[rd]_-{1i^*\!1}\ar[ru]_-{i^*\!11}\ar@{}[ruuu]|*[@]{\cong}\xtwocell[rr]{}<>{^\mu_0\ }&&R\\
&RR\ar[ru]_-{i^*\!1}&
}}}
\end{align}
\end{enumerate}
This gives a total of twelve axioms for the data $(R,i\dashv i^*,r,r^0,r^2,\mu_0,\mu_2)$.
\end{defi}

As one can see, a monad of oplax actions amounts to a lot of information. Fortunately, some of it is redundant. We shall see below that the cell $r^2$ may be written in terms of the rest of the structure.

\begin{lem}\label{lem:RedundantOplaxActionAssociator}
For every monad of oplax actions $(R,i\dashv i^*,r,\mu_0,\mu_2)$ in a monoidal bicategory $\mathcal{M}$ the following equality holds.
\[
\vcenter{\hbox{\xymatrix@!0@=15mm{
RRR\ar[r]^-{r1}\ar[d]_-{1i^*\!1}\xtwocell[rd]{}<>{^r^2}&RR\ar[d]^-{r}\\
RR\ar[r]_-{r}&R
}}}
\quad=
\vcenter{\hbox{\xymatrix@!0@C=15mm{
RRR\ar@/^8mm/[rr]^-{1}\ar[rd]_-{11i1}\ar[r]^-{1i^*\!1}\ar@/_7mm/[dddr]_-{1}\xtwocell[rddd]{}<>{^11\eta 1\quad}\xtwocell[rr]{}<>{^<-3>1\varepsilon 1\ \ }&RR\ar[r]^-{1i1}&RRR\ar[dd]^-{1r}\ar[r]^-{r1}\xtwocell[rdd]{!<2mm,0mm>}<>{^\mu_2}&**{!<-2mm>}RR\ar@<2mm>[dd]^-{r}&\\
&RRRR\ar[ru]_-{1i^*\!11}\ar[dd]|-{11i^*\!1}\ar@{}[u]|*[@]{\cong}\xtwocell[rd]{}<>{^1\mu_0\ \ }&\\
&&RR\ar[r]_-{r}&**{!<-2mm>}R\\
&RRR\ar[ru]_-{1i^*\!1}&
}}}
\]
\end{lem}
\begin{proof}

Starting with the right hand side, one uses the following calculation.
\begin{align*}
\vcenter{\hbox{\xymatrix@!0@C=15mm{
&RR\ar@/^9mm/[rrdd]^-{1}\ar@{{}{ }{}}@/^2mm/[rddd]|<<<<<<<{1\varepsilon 1}&&\\
RRR\ar@/^15mm/[rrdd]^-{1}\ar[rd]_-{11i1}\ar[r]^-{1i^*\!1}\ar@/_7mm/[dddr]_-{1}\ar@<1mm>[ru]^-{r1}\xtwocell[rddd]{}<>{^11\eta 1\quad}\xtwocell[rrdd]{!<0mm,2mm>}<>{^<-6>\ \ }&RR\ar[rdd]^-{1i1}&&\\
&RRRR\ar[rd]_-{1i^*\!11}\ar[dd]_-{11i^*\!1}\ar@{}[u]|*[@]{\cong}\xtwocell[rddd]{}<>{^1\mu_0\ \ }&&RR\ar[dd]^-{r}\\
&&RRR\ar[dd]^-{1r}\ar[ru]^-{r1}\xtwocell[rd]{}<>{^\mu_2\ }&\\
&RRR\ar[rd]_-{1i^*\!1}&&R\\
&&RR\ar[ru]_-{r}&
}}}
&\stackrel{\eqref{ax:OLA2}}{=}\!\!\!\!
\vcenter{\hbox{\xymatrix@!0@C=15mm{
&RR\ar[rd]|-{1i1}\ar@/^9mm/[rrdd]^-{1}\xtwocell[rrdd]{}<>{^<-3>r^01\ \ }&&\\
RRR\ar[rd]_-{11i1}\ar@/_7mm/[dddr]_-{1}\ar@<1mm>[ru]^-{r1}\xtwocell[rddd]{}<>{^11\eta 1\quad}&&RRR\ar[rd]^-{r1}&\\
&RRRR\ar[rd]_-{1i^*\!11}\ar[dd]_-{11i^*\!1}\ar[ru]^-{r11}\ar@{}[uu]|*[@]{\cong}\xtwocell[rddd]{}<>{^1\mu_0\ \ }\xtwocell[rr]{}<>{^r^21\ \ }&&RR\ar[dd]^-{r}\\
&&RRR\ar[dd]^-{1r}\ar[ru]^-{r1}\xtwocell[rd]{}<>{^\mu_2\ }&\\
&RRR\ar[rd]_-{1i^*\!1}&&R\\
&&RR\ar[ru]_-{r}&
}}}
\\
\stackrel{\eqref{ax:M2}}{=}
\vcenter{\hbox{\xymatrix@!0@C=15mm{
&RR\ar[rd]|-{1i1}\ar@/^9mm/[rrdd]^-{1}\xtwocell[rrdd]{}<>{^<-3>r^01\ \ }&&\\
RRR\ar[rd]_-{11i1}\ar@/_7mm/[dddr]_-{1}\ar@<1mm>[ru]^-{r1}\xtwocell[rddd]{}<>{^11\eta 1\quad}&&RRR\ar[rd]^-{r1}\ar[dd]_-{1i^*\!1}\xtwocell[rddd]{}<>{^r^2\ }&\\
&RRRR\ar[dd]_-{11i^*\!1}\ar[ru]^-{r11}\ar@{}[uu]|*[@]{\cong}&&RR\ar[dd]^-{r}\\
&&RR\ar[rd]^-{r}&\\
&RRR\ar[rd]_-{1i^*\!1}\ar[ru]^-{r1}\ar@{}[ruuu]|*[@]{\cong}\xtwocell[rr]{}<>{^r^2\ }&&R\\
&&RR\ar[ru]_-{r}&
}}}
&\stackrel{\phantom{\eqref{ax:OLA2}}}{=}\!\!\!\!
\vcenter{\hbox{\xymatrix@!0@C=15mm{
&RR\ar[rd]|-{1i1}\ar@/^9mm/[rrdd]^-{1}\ar@/_7mm/[dddr]_-{1}\xtwocell[rrdd]{}<>{^<-3>r^01\ \ }\xtwocell[rddd]{}<>{^1\eta 1\quad}&&\\
RRR\ar@/_7mm/[dddr]_-{1}\ar@<1mm>[ru]^-{r1}&&RRR\ar[dd]_-{1i^*\!1}\ar[rd]^-{r1}\xtwocell[rddd]{}<>{^r^2\ }&\\
&&&RR\ar[dd]^-{r}\\
&&RR\ar[rd]^-{r}&\\
&RRR\ar[rd]_-{1i^*\!1}\ar[ru]\xtwocell[rr]{}<>{^r^2\ }&&R\\
&&RR\ar[ru]_-{r}&
}}}
\end{align*}
\[
\stackrel{\eqref{ax:OLA3}}{=}
\vcenter{\hbox{\xymatrix@!0@=15mm{
RRR\ar[r]^-{r1}\ar[d]_-{1i^*\!1}\xtwocell[rd]{}<>{^r^2}&RR\ar[d]^-{r}\\
RR\ar[r]_-{r}&R
}}}
\]
\end{proof}

At this point we would like to point out a couple of technical facts involving the calculus of mates.

\begin{rem}\label{rem:AltSkewMonoidale}
A right skew monoidale $(M,i,m,\alpha,\lambda,\rho)$ whose the unit has a right adjoint $i\dashv i^*$ may be expressed in simpler terms using the mate $\kappa$ of the cell $\lambda$ under the adjunction. This involves the data $(M,i,m,\alpha,\kappa,\rho)$ satisfying five axioms; two of which remain unchanged since they do not involve the left unitor cell $\lambda$: the pentagon~\eqref{ax:SKM1} and the $\alpha$-$\rho$ compatibility~\eqref{ax:SKM4}. But the remaining three have modified versions as follows.
\begin{align}
\tag{SKM2'}\label{ax:SKM2'}
\vcenter{\hbox{\xymatrix@!0@C=15mm{
MM\ar[rd]_-{1i1}\ar@/_7mm/[rddd]_-{1}\ar@/^7mm/[rr]^-{1}\xtwocell[rddd]{}<>{^1\eta 1\ }\xtwocell[rr]{}<>{^\rho 1\ }&&MM\ar[dd]^-{m}\\
&MMM\ar@/_3mm/[dd]_-{1i^*\!1}\ar@/^3mm/[dd]^-{1m}\xtwocell[dd]{}<>{^1\kappa}\ar[ru]^-{m1}\xtwocell[rd]{}<>{^\alpha}&\\
&&M\\
&MM\ar[ru]_-{m}&
}}}
\quad&=\!\!\!\!\!
\vcenter{\hbox{\xymatrix@!0@C=15mm{
MM\ar@/_7mm/[rrdd]_-{m}\ar@/^7mm/[rrdd]^-{m}&&\\
&&\\
&&M
}}}
\\
\tag{SKM3'}\label{ax:SKM3'}
\vcenter{\hbox{\xymatrix@!0@=18mm{
RRR\ar@/_3mm/[r]_-{i^*\!11}\ar@/^3mm/[r]^-{m1}\xtwocell[r]{!<2mm,0mm>}<>{^\kappa 1\ }\ar[d]_-{1m}&RR\ar[d]^-{m}\\
RR\ar[r]_-{i^*\!1}\ar@{{}{ }{}}@/_3mm/[r]_-{\phantom{i^*\!1}}\ar@{}[ru]|*=0[@]{\cong}&R
}}}
\quad&=\quad
\vcenter{\hbox{\xymatrix@!0@=18mm{
RRR\ar[r]^-{m1}\ar[d]_-{1m}\ar@{{}{ }{}}@/^3mm/[r]^-{\phantom{m1}}\xtwocell[rd]{}<>{^\alpha}&RR\ar[d]^-{m}\\
RR\ar@/_3mm/[r]_-{i^*\!1}\ar@/^3mm/[r]^-{m}\xtwocell[r]{}<>{^\kappa}&R
}}}
\\
\tag{SKM5'}\label{ax:SKM5'}
\vcenter{\hbox{\xymatrix@!0@=18mm{
RR\ar@/_3mm/[d]_-{i^*\!1}\ar@/^3mm/[d]^-{m}\xtwocell[d]{}<>{^\kappa}&R\ar[l]_-{1i}\ar@/^3mm/[ld]^-{1}\xtwocell[ld]{}<>{^<1>\rho}\\
R
}}}
\quad&=\quad
\vcenter{\hbox{\xymatrix@!0@=9mm{
RR\ar[dd]_-{i^*\!1}\ar@{}[rd]|*=0[@]{\cong}&&R\ar[ll]_-{1i}\ar[ld]_-{i^*}\ar@/^7mm/[lldd]^-{1}\xtwocell[lldd]{}<>{^<-2>\varepsilon}\\
&I\ar[ld]_-{i}&\\
R
}}}
\end{align}
\end{rem}

The following lemma involves a pair of mates $\lambda$ and $\kappa$ as above.

\begin{lem}\label{lem:BijectionLeftUnitVsMonadUnit}
For every object $R$, every arrow $\xymatrix@1@C=5mm{m:RR\ar[r]&R}$, and every adjunction $i\dashv i^*$ in $\mathcal{M}$
\[
\vcenter{\hbox{\xymatrix{
R\xtwocell[d]{}_{i}^{i^*}{'\dashv}\\
I
}}}
\]
there are bijections between cells $\mu_0$ that satisfy the equation \eqref{ax:2SIM1mu0} below, cells $\kappa$, and cells $\lambda$ with sources and targets as shown.
\[
\left\lbrace
\vcenter{\hbox{\xymatrix@!0@=15mm{
RRR\ar[r]^{i^*\!11}\ar[d]_{1i^*\!1}\xtwocell[rd]{}<>{^\mu_0}&RR\ar[d]^{m}\\
RR\ar[r]_{i^*\!1}&R
}}}
\vrule
\text{ \eqref{ax:2SIM1mu0}}
\right\rbrace
\cong
\left\lbrace
\vcenter{\hbox{\xymatrix{
RR\ar@/_3mm/[d]_-{i^*\!1}\ar@/^3mm/[d]^-{m}\xtwocell[d]{}<>{^\kappa}\\
R
}}}
\right\rbrace
\cong
\left\lbrace
\vcenter{\hbox{\xymatrix@!0@=15mm{
R\ar[r]^-{i1}\ar[rd]_-{1}\xtwocell[rd]{}<>{^<-2>\lambda}&RR\ar[d]^-{m}\\
&R
}}}
\right\rbrace
\]
\begin{equation}
\tag{2SIM1$\mu_0$}\label{ax:2SIM1mu0}
\vcenter{\hbox{\xymatrix@!0@C=15mm{
&RRR\ar[rd]^-{i^*\!11}&\\
RRRR\ar[ru]^-{i^*\!111}\ar[dd]_-{11i^*\!1}\ar[rd]_-{1i^*\!11}&&RR\ar[dd]^-{m}\\
&RRR\ar[dd]^-{1i^*\!1}\ar[ru]^-{i^*\!11}\ar@{}[uu]|*[@]{\cong}\xtwocell[rd]{}<>{^\mu_0\ }&\\
RRR\ar[rd]_-{1i^*\!1}\ar@{}[ru]|*[@]{\cong}&&R\\
&RR\ar[ru]_-{i^*\!1}&
}}}
\quad=\quad
\vcenter{\hbox{\xymatrix@!0@C=15mm{
&RRR\ar[rd]^-{i^*\!11}\ar[dd]_-{1i^*\!1}\xtwocell[rddd]{}<>{^\mu_0\ }&\\
RRRR\ar[ru]^-{i^*\!111}\ar[dd]_-{11i^*\!1}&&RR\ar[dd]^-{m}\\
&RR\ar[rd]^-{i^*\!1}&\\
RRR\ar[rd]_-{1i^*\!1}\ar[ru]_-{i^*\!11}\ar@{}[ruuu]|*[@]{\cong}&&R\\
&RR\ar[ru]_-{i^*\!1}\ar@{}[uu]|*[@]{\cong}&
}}}
\end{equation}
\end{lem}
\begin{proof}

The bijection between cells $\lambda$ and cells $\kappa$ is done by the calculus of mates of the adjunction $i1\dashv i^*1$. The interesting part is to prove the bijection between cells $\kappa$ and cells $\mu_0$ satisfying \eqref{ax:2SIM1mu0}. For this purpose, one defines the functions below.
\begin{align*}
\vcenter{\hbox{\xymatrix{
RR\ar@/_3mm/[d]_-{i^*\!1}\ar@/^3mm/[d]^-{m}\xtwocell[d]{}<>{^\kappa}\\
R
}}}
\vcenter{\hbox{\xymatrix@C=15mm{
\ar@{|->}[r]&
}}}
\widetilde{\kappa}:=
&
\vcenter{\hbox{\xymatrix@!0@=15mm{
RRR\ar[r]^-{i^*\!11}\ar[d]_-{1i^*\!1}&RR\ar@/_3mm/[d]_-{i^*\!1}\ar@/^3mm/[d]^-{m}\xtwocell[d]{}<>{^\kappa}\\
RR\ar[r]_-{i^*\!1}\ar@{}[]!<-3mm,-0mm>;[ru]!<-3mm,-0mm>|*=0[@]{\cong}&R\\
}}}
\\
\overline{\mu}_0:=
\vcenter{\hbox{\xymatrix@!0@C=15mm{
RR\ar@/^8mm/[rr]^-{1}\ar@/_7mm/[rddd]_-{1}\ar[rd]|-{1i1}\ar[r]|-{i^*\!1}\xtwocell[rr]{}<>{^<-3>\varepsilon 1}\xtwocell[rddd]{}<>{^1\eta 1\ }&R\ar[r]|-{i1}&RR\ar[dd]^-{m}\\
&RRR\ar[dd]|-{1i^*\!1}\ar[ru]|-{i^*\!11}\xtwocell[rd]{}<>{^\mu_0\ }\ar@{}[u]|*=0[@]{\cong}&\\
&&R\\
&RR\ar[ru]_-{i^*\!1}&
}}}
\vcenter{\hbox{\xymatrix@C=10mm{
&\ar@{|->}[l]
}}}
&
\vcenter{\hbox{\xymatrix@!0@=15mm{
RRR\ar[r]^{i^*\!11}\ar[d]_{1i^*\!1}\xtwocell[rd]{}<>{^\mu_0}&RR\ar[d]^{m}\\
RR\ar[r]_{i^*\!1}&R
}}}
\end{align*}
It is easy to see that given a cell $\kappa$ the cell $\widetilde{\kappa}$ satisfies axiom \eqref{ax:2SIM1mu0} as a consequence of the coherence for the interchange law. Now, the following calculations show that these functions are inverse to each other. First, let $\kappa$ as in the statement and argue as follows,
\[
\overline{\widetilde{\kappa}}=
\vcenter{\hbox{\xymatrix@!0@C=15mm{
RR\ar@/^8mm/[rr]^-{1}\ar@/_7mm/[rddd]_-{1}\ar[rd]|-{1i1}\ar[r]|-{i^*\!1}\xtwocell[rr]{}<>{^<-3>\varepsilon 1}\xtwocell[rddd]{}<>{^1\eta 1\ }&R\ar[r]|-{i1}&RR\ar@/_3mm/[dd]_-{i^*\!1}\ar@/^3mm/[dd]^-{m}\xtwocell[dd]{}<>{^\kappa}\\
&RRR\ar[dd]|-{1i^*\!1}\ar[ru]|-{i^*\!11}\ar@{}[u]|*=0[@]{\cong}&\\
&&R\\
&RR\ar[ru]_-{i^*\!1}\ar@{}[ruuu]|*=0[@]{\cong}&
}}}
=
\vcenter{\hbox{\xymatrix@!0@C=15mm{
RR\ar@/^8mm/[rr]^-{1}\ar@/_7mm/[rddd]_-{1}\ar[r]|-{i^*\!1}\xtwocell[rr]{}<>{^<-3>\varepsilon 1}&R\ar[r]|-{i1}\ar@/_7mm/[rdd]_-{1}\xtwocell[rdd]{}<>{^<2>1\eta 1\ }&RR\ar@/_3mm/[dd]_-{i^*\!1}\ar@/^3mm/[dd]^-{m}\xtwocell[dd]{}<>{^\kappa}\\
&&\\
&&R\\
&RR\ar[ru]_-{i^*\!1}&
}}}
=\kappa
\]
where the first equality holds by the definition of $\overline{\widetilde{\kappa}}$, the second by Gray monoid axioms, and the third by the snake equation of the adjunction. Now let $\mu_0$ be a cell satisfying axiom~\eqref{ax:2SIM1mu0}, then the other composite is the identity.
\[
\widetilde{\overline{\mu}}_0
=\!
\vcenter{\hbox{\xymatrix@!0@C=9mm@R=5.3mm{
&&RR\ar[rddd]|-{1i1}\ar[rrdd]^-{i^*\!1}\ar@/_9mm/[rddddddd]_-{1}\ar@/^14mm/[rrrddddd]^-{1}\xtwocell[rrrddddd]{}<>{^<-6>\varepsilon 1\ }\xtwocell[rddddddd]{}<>{^<2>1\eta 1}&&&\\
&&&&&\\
RRR\ar[rruu]^-{i^*\!11}\ar@/_9mm/[rddddddd]_<(.6){1}&&&&R\ar[rddd]^-{i1}&\\
&&&RRR\ar[rrdd]|-{i^*\!11}\ar[dddd]|-{1i^*\!1}\ar@{}[ru]|*=0[@]{\cong}\xtwocell[rrdddddd]{}<>{^\mu_0\ }&&\\
&&&&&\\
&&&&&RR\ar[dddd]^-{m}\\
&&&&&\\
&&&RR\ar[rrdd]^-{i^*\!1}&&\\
&&&&&\\
&RRR\ar[rrdd]_-{1i^*\!1}\ar[rruu]^-{i^*\!11}&&&&R\\
&&&&&\\
&&&RR\ar[rruu]_-{1i^*\!1}\ar@{}[uuuu]|*=0[@]{\cong}&&
}}}
\!\!\!=\!
\vcenter{\hbox{\xymatrix@!0@C=9mm@R=5.3mm{
&&RR\ar[rddd]|-{1i1}\ar[rrdd]^-{i^*\!1}\ar@/^14mm/[rrrddddd]^-{1}\xtwocell[rrrddddd]{}<>{^<-6>\varepsilon 1\ }&&&\\
&&&&&\\
RRR\ar[rruu]^-{i^*\!11}\ar[rddd]|-{11i1}\ar@/_9mm/[rddddddd]_<(.6){1}\xtwocell[rddddddd]{}<>{^<2>11\eta 1}&&&&R\ar[rddd]^-{i1}&\\
&&&RRR\ar[rrdd]|-{i^*\!11}\ar[dddd]|-{1i^*\!1}\ar@{}[ru]|*=0[@]{\cong}\xtwocell[rrdddddd]{}<>{^\mu_0\ }&&\\
&&&&&\\
&RRRR\ar[dddd]|-{11i^*\!1}\ar[rruu]^-{i^*\!111}\ar@{}[ruuuuu]|*=0[@]{\cong}&&&&RR\ar[dddd]^-{m}\\
&&&&&\\
&&&RR\ar[rrdd]^-{i^*\!1}&&\\
&&&&&\\
&RRR\ar[rrdd]_-{1i^*\!1}\ar[rruu]^-{i^*\!11}\ar@{}[rruuuuuu]|*=0[@]{\cong}&&&&R\\
&&&&&\\
&&&RR\ar[rruu]_-{i^*\!1}\ar@{}[uuuu]|*=0[@]{\cong}&&
}}}
\]
\[
\stackrel{\eqref{ax:2SIM1mu0}}{=}
\vcenter{\hbox{\xymatrix@!0@C=9mm@R=5.3mm{
&&RR\ar[rddd]|-{1i1}\ar[rrdd]^-{i^*\!1}\ar@/^14mm/[rrrddddd]^-{1}\xtwocell[rrrddddd]{}<>{^<-6>\varepsilon 1\ }&&&\\
&&&&&\\
RRR\ar[rruu]^-{i^*\!11}\ar[rddd]|-{11i1}\ar@/_9mm/[rddddddd]_<(.6){1}\xtwocell[rddddddd]{}<>{^<2>11\eta 1}&&&&R\ar[rddd]^-{i1}&\\
&&&RRR\ar[rrdd]|-{i^*\!11}\ar@{}[ru]|*=0[@]{\cong}&&\\
&&&&&\\
&RRRR\ar[dddd]|-{11i^*\!1}\ar[rruu]^-{i^*\!111}\ar[rrdd]|-{1i^*\!11}\ar@{}[ruuuuu]|*=0[@]{\cong}&&&&RR\ar[dddd]^-{m}\\
&&&&&\\
&&&RRR\ar[rruu]^-{i^*\!11}\ar[dddd]^-{1i^*\!1}\ar@{}[uuuu]|*=0[@]{\cong}\xtwocell[rrdd]{}<>{^\mu_0\ }&&\\
&&&&&\\
&RRR\ar[rrdd]_-{1i^*\!1}\ar@{}[rruu]|*=0[@]{\cong}&&&&R\\
&&&&&\\
&&&RR\ar[rruu]_-{i^*\!1}&&
}}}
\]
\[
=\!
\vcenter{\hbox{\xymatrix@!0@C=9mm@R=5.3mm{
&&RR\ar@/^14mm/[rrrddddd]^-{1}&&&\\
&&&&&\\
RRR\ar[rruu]^-{i^*\!11}\ar[rddd]|-{11i1}\ar[rrdd]^-{1i^*\!1}\ar@/^14mm/[rrrddddd]^-{1}\ar@/_9mm/[rddddddd]_<(.6){1}\xtwocell[rddddddd]{}<>{^<2>11\eta 1}\xtwocell[rrrddddd]{}<>{^<-6>1\varepsilon 1\ \ }&&&&&\\
&&&&&\\
&&RR\ar[rddd]^-{1i1}&&&\\
&RRRR\ar[dddd]|-{11i^*\!1}\ar[rrdd]|-{1i^*\!11}\ar@{}[ru]|*=0[@]{\cong}&&&&RR\ar[dddd]^-{m}\\
&&&&&\\
&&&RRR\ar[rruu]^-{i^*\!11}\ar[dddd]^-{1i^*\!1}\xtwocell[rrdd]{}<>{^\mu_0\ }&&\\
&&&&&\\
&RRR\ar[rrdd]_-{1i^*\!1}\ar@{}[rruu]|*=0[@]{\cong}&&&&R\\
&&&&&\\
&&&RR\ar[rruu]_-{i^*\!1}&&
}}}
\!\!\!=\!
\vcenter{\hbox{\xymatrix@!0@C=9mm@R=5.3mm{
&&RR\ar@/^14mm/[rrrddddd]^-{1}&&&\\
&&&&&\\
RRR\ar[rruu]^-{i^*\!11}\ar[rrdd]^-{1i^*\!1}\ar@/^14mm/[rrrddddd]^-{1}\ar@/_9mm/[rddddddd]_<(.6){1}\xtwocell[rrrddddd]{}<>{^<-6>1\varepsilon 1\ \ }&&&&&\\
&&&&&\\
&&RR\ar[rddd]^-{1i1}\ar@/_9mm/[rddddddd]_-{1}\xtwocell[rddddddd]{}<>{^<2>11\eta 1}&&&\\
&&&&&RR\ar[dddd]^-{m}\\
&&&&&\\
&&&RRR\ar[rruu]^-{i^*\!11}\ar[dddd]^-{1i^*\!1}\xtwocell[rrdd]{}<>{^\mu_0\ }&&\\
&&&&&\\
&RRR\ar[rrdd]_-{1i^*\!1}&&&&R\\
&&&&&\\
&&&RR\ar[rruu]_-{i^*\!1}&&
}}}
\!\!\!=
\mu_0
\]
\end{proof}

Note that for a monad of oplax actions $(R,i\dashv i^*,r,\mu_0,\mu_2)$ axiom~\eqref{ax:2SIM1} for $\mu_0$ is precisely the axiom~\eqref{ax:2SIM1mu0} from Lemma~\ref{lem:BijectionLeftUnitVsMonadUnit}. Hence, we get a pair of mates $\lambda$ and $\kappa$ each of which amounts to the same information as the cell $\mu_0$ satisfying \eqref{ax:2SIM1}. This simplifies the formula from Lemma~\ref{lem:RedundantOplaxActionAssociator}.
\begin{equation}\label{eq:associator}
\vcenter{\hbox{\xymatrix@!0@=15mm{
RRR\ar[r]^-{r1}\ar[d]_-{1i^*\!1}\xtwocell[rd]{}<>{^r^2}&RR\ar[d]^-{r}\\
RR\ar[r]_-{r}&R
}}}
\quad:=
\vcenter{\hbox{\xymatrix@!0@=15mm{
RRR\ar@/_3mm/[d]_-{1i^*\!1}\ar@/^3mm/[d]^-{1r}\xtwocell[d]{}<>{^1\kappa}\ar[r]^-{r1}\xtwocell[rd]{!<4mm,-3mm>}<>{^<-1>\mu_2}&RR\ar[d]^-{r}\\
RR\ar[r]_-{r}&R\\
}}}
\end{equation}
With this information we describe a right skew monoidale induced by a monad of oplax actions.

\begin{prop}\label{prop:MonadInducesSkewMonoidale}
For every monad of oplax actions $(R,i\dashv i^*,r,\mu_0,\mu_2)$ in a monoidal bicategory $\mathcal{M}$ there is a right skew monoidale on the object $R$ with structure given as follows,
\begin{center}
\begin{tabular}{rlrl}
Product&$\xymatrix@1@C=5mm{RR\ar[r]^-{r}&R}$&Unit&$\xymatrix@1@C=5mm{I\ar[r]^-{i}&R}$\\
Associator&
$
\vcenter{\hbox{\xymatrix@!0@=15mm{
RRR\ar[r]^-{r1}\ar[d]_-{1r}\xtwocell[rd]{}<>{^\mu_2\ }&RR\ar[d]^-{r}\\
RR\ar[r]_-{r}&R
}}}
$
&
Unitors
&
$
\vcenter{\hbox{\xymatrix@!0@=15mm{
R\ar[r]^-{i1}\ar[rd]_-{1}\xtwocell[rd]{}<>{^<-2>\lambda}&RR\ar[d]|-{r}&R\ar[l]_-{1i}\ar[dl]^-{1}\xtwocell[ld]{}<>{^<2>\ r^0}\\
&R
}}}
$
\end{tabular}
\end{center}
where the left unitor $\lambda$ is obtained from $\mu_0$ as in Lemma~\ref{lem:BijectionLeftUnitVsMonadUnit}.
\end{prop}
\begin{proof}

Let $\lambda$ and $\kappa$ as in Lemma~\ref{lem:BijectionLeftUnitVsMonadUnit}, then we have the following formulas.
\[
\mu_0=
\vcenter{\hbox{\xymatrix@!0@=15mm{
RRR\ar[r]^-{i^*\!11}\ar[d]_-{1i^*\!1}&RR\ar@/_3mm/[d]_-{i^*\!1}\ar@/^3mm/[d]^-{r}\xtwocell[d]{}<>{^\kappa}\\
RR\ar[r]_-{i^*\!1}\ar@{}[]!<-3mm,-0mm>;[ru]!<-3mm,-0mm>|*=0[@]{\cong}&R\\
}}}
\qquad
\kappa=
\vcenter{\hbox{\xymatrix@!0@C=15mm{
RR\ar@/^8mm/[rr]^-{1}\ar@/_7mm/[rddd]_-{1}\ar[rd]|-{1i1}\ar[r]|-{i^*\!1}\xtwocell[rr]{}<>{^<-3>\varepsilon 1}\xtwocell[rddd]{}<>{^1\eta 1\ }&R\ar[r]|-{i1}&RR\ar[dd]^-{r}\\
&RRR\ar[dd]|-{1i^*\!1}\ar[ru]|-{i^*\!11}\xtwocell[rd]{}<>{^\mu_0\ }\ar@{}[u]|*=0[@]{\cong}&\\
&&R\\
&RR\ar[ru]_-{i^*\!1}&
}}}
\]
Now, axiom~\eqref{ax:SKM1} is literally axiom~\eqref{ax:M1} for the monad of oplax actions and axiom~\eqref{ax:SKM4} is literally axiom~\eqref{ax:2SIM3} for $\mu_2$. Since the proposed unit for the right skew monoidale has a right adjoint, for the rest of the axioms one may prove their alternative versions written in terms of $\kappa$ as in Remark~\ref{rem:AltSkewMonoidale}. The following calculation proves axiom~\eqref{ax:SKM2'}.
\[
\vcenter{\hbox{\xymatrix@!0@C=15mm{
RR\ar[rd]_-{1i1}\ar@/_7mm/[rddd]_-{1}\ar@/^7mm/[rr]^-{1}\xtwocell[rddd]{}<>{^1\eta 1\ }\xtwocell[rr]{}<>{^r^01\ \ }&&RR\ar[dd]^-{r}\\
&RRR\ar@/_3mm/[dd]_-{1i^*\!1}\ar@/^3mm/[dd]^-{1r}\xtwocell[dd]{}<>{^1\kappa}\ar[ru]^-{r1}\xtwocell[rd]{}<>{^\mu_2\ }&\\
&&R\\
&RR\ar[ru]_-{r}&
}}}
\!\!\stackrel{\eqref{eq:associator}}{=}\!\!
\vcenter{\hbox{\xymatrix@!0@C=15mm{
RR\ar[rd]_-{1i1}\ar@/_7mm/[rddd]_-{1}\ar@/^7mm/[rr]^-{1}\xtwocell[rddd]{}<>{^1\eta 1\ }\xtwocell[rr]{}<>{^r^01\ \ }&&RR\ar[dd]^-{r}\\
&RRR\ar[dd]_-{1i^*\!1}\ar[ru]^-{r1}\xtwocell[rd]{}<>{^r^2\ }&\\
&&R\\
&RR\ar[ru]_-{r}&
}}}
\!\stackrel{\eqref{ax:OLA3}}{=}\!\!\!\!\!\!
\vcenter{\hbox{\xymatrix@!0@C=15mm{
RR\ar@/_7mm/[rrdd]_-{r}\ar@/^7mm/[rrdd]^-{r}&&\\
&&\\
&&R
}}}
\]
Axiom \eqref{ax:SKM3'} follows from axiom \eqref{ax:M3} for the monad of oplax actions.
\[
\vcenter{\hbox{\xymatrix@!0@=18mm{
RRR\ar@/_3mm/[r]_-{i^*\!11}\ar@/^3mm/[r]^-{r1}\xtwocell[r]{!<2mm,0mm>}<>{^\kappa 1\ }\ar[d]_-{1r}&RR\ar[d]^-{r}\\
RR\ar[r]_-{i^*\!1}\ar@{{}{ }{}}@/_3mm/[r]_-{\phantom{i^*\!1}}\ar@{}[ru]|*=0[@]{\cong}&R
}}}
=
\vcenter{\hbox{\xymatrix@!0@C=9mm@R=5mm{
&&&&RRR\ar[rrdd]^-{r1}&&\\
&&RR\ar[rru]^-{i11}&&&&\\
RRR\ar[dddd]_-{1r}\ar[rr]_-{1i11}\ar[rru]^-{i^*\!11}\ar@/^9mm/[rrrruu]^-{1}\ar@/_7mm/[rrrrdd]_-{1}\xtwocell[rrrrdd]{}<>{^<1>1\eta 11\quad}\xtwocell[rrrruu]{}<>{^<-3>\varepsilon 11\ \ }&&RRRR\ar[rruu]|-{i^*\!111}\ar[rrdd]_-{1i^*\!11}\xtwocell[rrrr]{}<>{^\mu_01\quad}&&&&RR\ar[dddd]^-{r}\\
&&&&&&\\
&&&&RRR\ar[dddd]^-{1r}\ar[rruu]^-{i^*\!11}&&\\
&&&&&&\\
RR\ar@/_7mm/[rrrrdd]_-{1}&&&&&&R\\
&&&&&&\\
&&&&RR\ar[rruu]_-{i^*\!1}\ar@{}[rruuuuuu]|*[@]{\cong}&&
}}}
\]
\[
=
\vcenter{\hbox{\xymatrix@!0@C=9mm@R=5mm{
&&&&RRR\ar[rrdd]^-{r1}&&\\
&&RR\ar[rru]^-{i11}&&&&\\
RRR\ar[dddd]_-{1r}\ar[rr]_-{1i11}\ar[rru]^-{i^*\!11}\ar@/^9mm/[rrrruu]^-{1}\xtwocell[rrrruu]{}<>{^<-3>\varepsilon 11\ \ }&&RRRR\ar[rruu]|-{i^*\!111}\ar[rrdd]_-{1i^*\!11}\ar[dddd]_-{11r}\xtwocell[rrrr]{}<>{^\mu_01\quad}&&&&RR\ar[dddd]^-{r}\\
&&&&&&\\
&&&&RRR\ar[dddd]^-{1r}\ar[rruu]^-{i^*\!11}&&\\
&&&&&&\\
RR\ar[rr]_-{1i1}\ar@{}[rruuuu]|*=0[@]{\cong}\ar@/_7mm/[rrrrdd]_-{1}\xtwocell[rrrrdd]{}<>{^<1>1\eta 1\ \ }&&RRR\ar[rrdd]_-{1i^*\!1}\ar@{}[rruu]|*[@]{\cong}&&&&R\\
&&&&&&\\
&&&&RR\ar[rruu]_-{i^*\!1}\ar@{}[rruuuuuu]|*[@]{\cong}&&
}}}
\]
\[\stackrel{\eqref{ax:M3}}{=}
\vcenter{\hbox{\xymatrix@!0@C=9mm@R=5mm{
&&&&RRR\ar[rrdd]^-{r1}\ar[dddd]_-{1r}\xtwocell[rrdddddd]{}<>{^\mu_2}&&\\
&&RR\ar[rru]^-{i11}&&&&\\
RRR\ar[dddd]_-{1r}\ar[rr]_-{1i11}\ar[rru]^-{i^*\!11}\ar@/^9mm/[rrrruu]^-{1}\xtwocell[rrrruu]{}<>{^<-3>\varepsilon 11\ \ }&&RRRR\ar[rruu]|-{i^*\!111}\ar[dddd]_-{11r}&&&&RR\ar[dddd]^-{r}\\
&&&&&&\\
&&&&RR\ar[rrdd]^-{r}&&\\
&&&&&&\\
RR\ar[rr]_-{1i1}\ar@{}[rruuuu]|*=0[@]{\cong}\ar@/_7mm/[rrrrdd]_-{1}\xtwocell[rrrrdd]{}<>{^<1>1\eta 1\ \ }&&RRR\ar[rrdd]_-{1i^*\!1}\ar[rruu]_-{i^*\!11}\ar@{}[rruuuuuu]|*[@]{\cong}\xtwocell[rrrr]{}<>{^\mu_0\ }&&&&R\\
&&&&&&\\
&&&&RR\ar[rruu]_-{i^*\!1}&&
}}}
\]
\[
=
\vcenter{\hbox{\xymatrix@!0@C=9mm@R=5mm{
&&&&RRR\ar[rrdd]^-{r1}\ar[dddd]_-{1r}\xtwocell[rrdddddd]{}<>{^\mu_2}&&\\
&&&&&&\\
RRR\ar[dddd]_-{1r}\ar@/^9mm/[rrrruu]^-{1}&&&&&&RR\ar[dddd]^-{r}\\
&&&&&&\\
&&&&RR\ar[rrdd]^-{r}&&\\
&&R\ar[rru]^-{i1}&&&&\\
RR\ar[rr]_-{1i1}\ar[rru]^-{i^*\!1}\ar@/_7mm/[rrrrdd]_-{1}\ar@/^9mm/[rrrruu]^-{1}\xtwocell[rrrruu]{}<>{^<-3>\varepsilon 1\ \ }\xtwocell[rrrrdd]{}<>{^<1>1\eta 1\ \ }&&RRR\ar[rrdd]_-{1i^*\!1}\ar[rruu]_-{i^*\!11}\xtwocell[rrrr]{}<>{^\mu_0\ }&&&&R\\
&&&&&&\\
&&&&RR\ar[rruu]_-{i^*\!1}&&
}}}
=
\vcenter{\hbox{\xymatrix@!0@=18mm{
RRR\ar[r]^-{r1}\ar[d]_-{1r}\ar@{{}{ }{}}@/^3mm/[r]^-{\phantom{r1}}\xtwocell[rd]{}<>{^\mu_2\ }&RR\ar[d]^-{r}\\
RR\ar@/_3mm/[r]_-{i^*\!1}\ar@/^3mm/[r]^-{r}\xtwocell[r]{}<>{^\kappa}&R
}}}
\]
And axiom \eqref{ax:SKM5'} follows from axiom \eqref{ax:2SIM3} for $\mu_0$.
\[
\vcenter{\hbox{\xymatrix@!0@=18mm{
RR\ar@/_3mm/[d]_-{i^*\!1}\ar@/^3mm/[d]^-{r}\xtwocell[d]{}<>{^\kappa}&R\ar[l]_-{1i}\ar@/^3mm/[ld]^-{1}\xtwocell[ld]{}<>{^<1>r^0}\\
R
}}}
=
\vcenter{\hbox{\xymatrix@!0@C=15mm{
R\ar[dd]_-{1i}\ar@/^8mm/[rr]^-{1}&&R\ar[dd]^-{1i}\ar@/^8mm/[dddd]^-{1}\xtwocell[dddd]{}<>{^<-3>r^0}\\
&&\\
RR\ar@/^8mm/[rr]^-{1}\ar@/_7mm/[rddd]_-{1}\ar[rd]|-{1i1}\ar[r]|-{i^*\!1}\xtwocell[rr]{}<>{^<-3>\varepsilon 1}\xtwocell[rddd]{}<>{^1\eta 1\ }&R\ar[r]|-{i1}&RR\ar[dd]^-{r}\\
&RRR\ar[dd]_-{1i^*\!1}\ar[ru]|-{i^*\!11}\ar@{}[u]|*=0[@]{\cong}\xtwocell[rd]{}<>{^\mu_0\ }&\\
&&R\\
&RR\ar[ru]_-{i^*\!1}&
}}}
\]
\[
=
\vcenter{\hbox{\xymatrix@!0@C=15mm{
R\ar[dd]_-{1i}\ar[rd]|-{1i}\ar[r]|-{i^*}\ar@/^8mm/[rr]^-{1}\xtwocell[rr]{}<>{^<-3>\varepsilon}&R\ar[r]|-{i}&R\ar[dd]^-{1i}\ar@/^8mm/[dddd]^-{1}\xtwocell[dddd]{}<>{^<-3>r^0}\\
&RR\ar[dd]_-{11i}\ar[ru]|-{i^*\!1}\ar@{}[u]|*=0[@]{\cong}&\\
RR\ar@/_7mm/[rddd]_-{1}\ar[rd]|-{1i1}\ar@{}[ru]|*=0[@]{\cong}\xtwocell[rddd]{}<>{^1\eta 1\ }&&RR\ar[dd]^-{r}\\
&RRR\ar[dd]_-{1i^*\!1}\ar[ru]|-{i^*\!11}\ar@{}[ruuu]|*=0[@]{\cong}\xtwocell[rd]{}<>{^\mu_0\ }&\\
&&R\\
&RR\ar[ru]_-{i^*\!1}&
}}}
\stackrel{\eqref{ax:2SIM3}}{=}
\vcenter{\hbox{\xymatrix@!0@C=7.5mm@R=4.3mm{
R\ar[dddd]_-{1i}\ar[rrdd]|-{1i}\ar[rr]|-{i^*}\ar@/^8mm/[rrrr]^-{1}\xtwocell[rrrr]{}<>{^<-3>\varepsilon}&&R\ar[rr]|-{i}&&R\ar@/^6mm/[dddddddd]^-{1}\\
&&&&\\
&&RR\ar[dddd]_-{11i}\ar[rruu]|-{i^*\!1}\ar[rddd]|-{1i^*}\ar@/^12mm/[dddddddd]^-{1}\ar@{}[uu]|*=0[@]{\cong}&&\\
&&&&\\
RR\ar@/_7mm/[rrdddddd]_-{1}\ar[rrdd]_-{1i1}\ar@{}[rruu]|*=0[@]{\cong}\xtwocell[rrdddddd]{}<>{^1\eta 1\ }&&&&\\
&&&R\ar[lddddd]|-{1i}&\\
&&RRR\ar[dddd]_-{1i^*\!1}\ar@{}[ruu]|*=0[@]{\cong}\xtwocell[d]{}<>{^<-5>1\varepsilon}&&\\
&&&&\\
&&&&R\\
&&&&\\
&&RR\ar[rruu]_-{i^*\!1}&&
}}}
\]
\[
=
\vcenter{\hbox{\xymatrix@!0@C=7.5mm@R=4.3mm{
R\ar[dddd]_-{1i}\ar[rrdd]|-{1i}\ar[rr]|-{i^*}\ar@/^8mm/[rrrr]^-{1}\ar@/_7mm/[rrrddddd]_-{1}\xtwocell[rrrr]{}<>{^<-3>\varepsilon}\xtwocell[rrrddddd]{}<>{^<1>1\eta\ }&&R\ar[rr]|-{i}&&R\ar@/^6mm/[dddddddd]^-{1}\\
&&&&\\
&&RR\ar[rruu]|-{i^*\!1}\ar[rddd]|-{1i^*}\ar@/^12mm/[dddddddd]^-{1}\ar@{}[uu]|*=0[@]{\cong}&&\\
&&&&\\
RR\ar@/_7mm/[rrdddddd]_-{1}&&&&\\
&&&R\ar[lddddd]|-{1i}&\\
&&{\xtwocell[d]{}<>{^<-5>1\varepsilon}}&&\\
&&&&\\
&&&&R\\
&&&&\\
&&RR\ar[rruu]_-{i^*\!1}&&
}}}
=
\vcenter{\hbox{\xymatrix@!0@=9mm{
RR\ar[dd]_-{i^*\!1}\ar@{}[rd]|*=0[@]{\cong}&&R\ar[ll]_-{1i}\ar[ld]_-{i^*}\ar@/^7mm/[lldd]^-{1}\xtwocell[lldd]{}<>{^<-2>\varepsilon}\\
&I\ar[ld]_-{i}&\\
R
}}}
\]
\end{proof}

\subsection{Skew monoidales whose unit has a right adjoint}

We shall now investigate skew monoidales whose unit has a right adjoint. An example of such right skew monoidales is given in Lemma~\ref{lem:OneRightSkewMonoidale}. It says that adjunctions $i\dashv i^*$ in $\mathcal{M}$
\[
\vcenter{\hbox{\xymatrix{
R\xtwocell[d]{}_{i}^{i^*}{'\dashv}\\
I
}}}
\]
induce right skew monoidal structures on $R$, with product $\xymatrix@1@C=5mm{i^*1:RR\ar[r]&R}$ and unit $i$. We also know that if $\mathcal{M}$ satisfies the hypothesis of Corollary~\ref{cor:MonadsOfOplaxActions} all monads of oplax actions are skew monoidales whose unit has a right adjoint. But, as we mentioned above, this holds true for an arbitrary monoidal bicategory $\mathcal{M}$.

Now, analogous to how a monad in a bicategory has an underlying arrow, we show that a right skew monoidale whose unit has a right adjoint has an ``underlying oplax action''. More precisely, there is a monad of oplax actions that carries the same information as the right skew monoidale, and the ``underlying oplax action'' is the 1-simplex part of such a monad of oplax actions. The existence of such monad of oplax actions is shown in Theorem~\ref{teo:MonadsAreSkewMonoidales} below. First, we construct the ``underlying oplax action'' of a right skew monoidale. To do so, we look at what we know about oplax actions which are the 1-simplex part of a monad of oplax actions; formula~\eqref{eq:associator} tells us that the associator $r^2$ is written in terms of $\mu_2$ and a cell $\kappa$ induced by $\mu_0$. Now, note that the associator $\alpha$ of a right skew monoidale on an object $R$ is of the same type as the cell $\mu_2$ for a monad of oplax actions on a 0-simplex $(R,i\dashv i^*)$. Thus, one may use formula~\eqref{eq:associator} as a guide to define the associator $r^2$ for the ``underlying oplax action'' of a right skew monoidale by replacing the instance of $\mu_2$ for $\alpha$ and using the mate $\kappa$ of $\lambda$.

\begin{prop}\label{prop:SkewMonoidaleInducesOplaxAction}
Every right skew monoidale $(R,i,m,\alpha,\lambda,\rho)$ for which the unit has a right adjoint $i\dashv i^*$ induces an oplax right $R$-action $(R,r)$ with respect to the right skew monoidal structure on $R$ induced by $i\dashv i^*$ as in Lemma~\ref{lem:OneRightSkewMonoidale}. The new oplax action structure is given by $r:=m$, $r^0:=\rho$, and
\[
\vcenter{\hbox{\xymatrix@!0@=15mm{
RRR\ar[r]^-{r1}\ar[d]_-{1i^*\!1}\xtwocell[rd]{}<>{^r^2}&RR\ar[d]^-{r}\\
RR\ar[r]_-{r}&R
}}}
\quad:=\ 
\vcenter{\hbox{\xymatrix@!0@=15mm{
RRR\ar@/_3mm/[d]_-{1i^*\!1}\ar@/^3mm/[d]^-{1r}\xtwocell[d]{}<>{^1\kappa}\ar[r]^-{r1}\xtwocell[rd]{!<4mm,-3mm>}<>{^<-1>\alpha}&RR\ar[d]^-{r}\\
RR\ar[r]_-{r}&R\\
}}}
\]
where $\kappa$ is the mate of $\lambda$ under the adjunction $i\dashv i^*$.
\end{prop}
\begin{proof}

Using the alternative versions of the axioms for the skew monoidale $(R,i,m,\alpha,\kappa,\rho)$ as described in Remark~\ref{rem:AltSkewMonoidale}, axiom~$\eqref{ax:OLA1}$ for $(R,r)$ follows by using axiom~\eqref{ax:SKM3'}.
\[
\vcenter{\hbox{\xymatrix@!0@C=19mm@R=10.5mm{
&RRR\ar[rd]^-{r1}&\\
RRRR\ar[ru]^-{r11}\ar[dd]_-{11i^*\!1}\ar@/_3mm/[rd]_>>>>{1i^*\!11}\ar@/^3mm/[rd]^>>>>>{1r1}\xtwocell[rd]{}<>{^1\kappa 1\ \ }\xtwocell[rr]{!<2mm,0mm>}<>{^\alpha 1\ }&&RR\ar[dd]^-{r}\\
&RRR\ar@/_3mm/[dd]_-{1i^*\!1}\ar@/^3mm/[dd]^-{1r}\xtwocell[dd]{}<>{^1\kappa}\ar[ru]^-{r1}\xtwocell[rd]{}<>{^\alpha}&\\
RRR\ar[rd]_-{1i^*\!1}\ar@{}[ru]|*=0[@]{\cong}&&R\\
&RR\ar[ru]_-{r}&
}}}
\]
\begin{align*}
\stackrel{\phantom{\eqref{ax:SKM1}}}{=}\!\!\!\!
\vcenter{\hbox{\xymatrix@!0@C=19mm@R=10.5mm{
&RRR\ar[rd]^-{r1}&\\
RRRR\ar[ru]^-{r11}\ar@/_3mm/[dd]_-{11i^*\!1}\ar@/^3mm/[dd]^>>>>>{11r}\xtwocell[dd]{}<>{^11\kappa}\ar@/_3mm/[rd]_>>>{1i^*\!11}\ar@/^3mm/[rd]^>>>>>{1r1}\xtwocell[rd]{}<>{^1\kappa 1\ \ }\xtwocell[rr]{!<2mm,0mm>}<>{^\alpha 1\ }&&RR\ar[dd]^-{r}\\
&RRR\ar[dd]^-{1r}\ar[ru]^-{r1}\xtwocell[rd]{}<>{^\alpha}&\\
RRR\ar[rd]_-{1i^*\!1}\ar@{}[ru]|*=0[@]{\cong}&&R\\
&RR\ar[ru]_-{r}&
}}}
&\!\!\stackrel{\eqref{ax:SKM3}}{=}\!\!\!\!
\vcenter{\hbox{\xymatrix@!0@C=19mm@R=10.5mm{
&RRR\ar[rd]^-{r1}&\\
RRRR\ar[ru]^-{r11}\ar@/_3mm/[dd]_-{11i^*\!1}\ar@/^3mm/[dd]^-{11r}\ar[rd]^-{1r1}\xtwocell[dd]{}<>{^11\kappa}\xtwocell[rddd]{}<>{^1\alpha\ }\xtwocell[rr]{}<>{^\alpha 1\ }&&RR\ar[dd]^-{r}\\
&RRR\ar[ru]^-{r1}\ar[dd]^-{1r}\xtwocell[rd]{}<>{^\alpha}&\\
RRR\ar@/_3mm/[rd]_-{1i^*\!1}\ar@/^3mm/[rd]^-{1r}\xtwocell[rd]{}<>{^1\kappa\ }&&R\\
&RR\ar[ru]_-{r}&
}}}
\\
\stackrel{\eqref{ax:SKM1}}{=}\!\!\!\!
\vcenter{\hbox{\xymatrix@!0@C=19mm@R=10.5mm{
&RRR\ar[rd]^-{r1}\ar[dd]^-{1r}\xtwocell[rddd]{}<>{^\alpha}&\\
RRRR\ar[ru]^-{r11}\ar@/_3mm/[dd]_-{11i^*\!1}\ar@/^3mm/[dd]^-{11r}\xtwocell[dd]{}<>{^11\kappa}&&RR\ar[dd]^-{r}\\
&RR\ar[rd]^-{r}&\\
RRR\ar@/_3mm/[rd]_-{1i^*\!1}\ar@/^3mm/[rd]^-{1r}\ar@{}[ruuu]|*=0[@]{\cong}\xtwocell[rd]{}<>{^1\kappa\ }\ar[ru]_-{r1}\xtwocell[rr]{}<>{^\alpha}&&R\\
&RR\ar[ru]_-{r}&
}}}
&\!\!\stackrel{\phantom{\eqref{ax:SKM3}}}{=}\!\!\!\!
\vcenter{\hbox{\xymatrix@!0@C=19mm@R=10.5mm{
&RRR\ar[rd]^-{r1}\ar@/_3mm/[dd]_-{1i^*\!1}\ar@/^3mm/[dd]^-{1r}\xtwocell[dd]{}<>{^1\kappa}\xtwocell[rddd]{}<>{^\alpha}&\\
RRRR\ar[ru]^-{r11}\ar[dd]_-{11i^*\!1}\ar@{{}{ }{}}@/_3mm/[dd]_-{\phantom{11i^*\!1}}&&RR\ar[dd]^-{r}\\
&RR\ar[rd]^-{r}&\\
RRR\ar@/_3mm/[rd]_-{1i^*\!1}\ar@/^3mm/[rd]^-{1r}\ar@{}[ruuu]|*=0[@]{\cong}\xtwocell[rd]{}<>{^1\kappa\ }\ar[ru]_-{r1}\xtwocell[rr]{}<>{^\alpha}&&R\\
&RR\ar[ru]_-{r}&
}}}
\end{align*}
To prove axiom~\eqref{ax:OLA2} one makes use of axiom \eqref{ax:SKM5'}.
\[
\vcenter{\hbox{\xymatrix@!0@C=7.5mm@R=4.3mm{
&&R\ar@/^7mm/[rrdddddd]^-{1}&&\\
&&&&\\
RR\ar[rruu]^-{r}\ar[dddd]_-{11i}\ar[rddd]|-{1i^*}\ar@/^7mm/[rrdddddd]^-{1}\xtwocell[rrdddddd]{}<>{^<-2.5>1\varepsilon\ }&&&&\\
&&&&\\
&&&&\\
&R\ar[rddd]|-{1i}&&&\\
RRR\ar[rrdd]_-{1i^*\!1}\ar@{}[ruu]|*=0[@]{\cong}&&&&R\\
&&&&\\
&&RR\ar[rruu]_-{r}&&
}}}
\stackrel{\eqref{ax:SKM5'}}{=}
\vcenter{\hbox{\xymatrix@!0@C=15mm{
&R\ar@/^7mm/[rddd]^-{1}&\\
RR\ar[ru]^-{r}\ar[dd]_-{11i}\ar@/^7mm/[rddd]^-{1}\xtwocell[rddd]{}<>{^1\rho\ }&&\\
&&\\
RRR\ar@/_3mm/[rd]_-{1i^*\!1}\ar@/^3mm/[rd]^<<<{1r}\xtwocell[rd]{}<>{^1\kappa\ }&&R\\
&RR\ar[ru]_-{r}&
}}}
\]\[
\stackrel{\eqref{ax:SKM4}}{=}
\vcenter{\hbox{\xymatrix@!0@C=15mm{
&R\ar@/^7mm/[rddd]^-{1}\ar[dd]_-{1i}\xtwocell[rddd]{}<>{^\rho}&\\
RR\ar[ru]^-{r}\ar[dd]_-{11i}&&\\
&RR\ar[rd]^-{r}&\\
RRR\ar@/_3mm/[rd]_-{1i^*\!1}\ar@/^3mm/[rd]^<<<{1r}\xtwocell[rd]{}<>{^1\kappa\ }\ar[ru]^-{r1}\ar@{}[ruuu]|*=0[@]{\cong}\xtwocell[rr]{}<>{^\alpha}&&R\\
&RR\ar[ru]_-{r}&
}}}
\]
And axiom~\eqref{ax:OLA3} is precisely \eqref{ax:SKM2'}.
\end{proof}

It is worth pointing out that in the presence of a right skew monoidale on an object $R$ whose unit has a right adjoint $i\dashv i^*$ there are automatically three oplax actions defined on $R$.
\begin{itemize}
\item The regular oplax right action.
\item The oplax action induced by the adjunction $i\dashv i^*$.
\item The ``underlying oplax action'' from the previous proposition.
\end{itemize}

Now we are ready to prove the one of the main theorems of this paper.

\begin{teo}\label{teo:MonadsAreSkewMonoidales}
In a monoidal bicategory $\mathcal{M}$, monads of oplax actions are in bijection with right skew monoidales whose unit has a right adjoint.
\end{teo}
\begin{proof}

In Proposition~\ref{prop:MonadInducesSkewMonoidale} we built a right skew monoidale out of a monad of oplax actions. Conversely, for a right skew monoidale $(R,i,m,\alpha,\lambda,\rho)$ whose unit has a right adjoint $i^*$, and where $\kappa$ is the mate of $\lambda$ under the adjunction; the following items (numbered as in the definition) constitute a monad of oplax actions.
\begin{enumerate}
\item Take the underlying object $R$ and the unit $i$ with its right adjoint $(R,i\dashv i^*)$.
\item Take the ``underlying oplax action'' $(r,r^0,r^2)$ of $(R,i,m,\alpha,\lambda,\rho)$ as constructed in Proposition~\ref{prop:SkewMonoidaleInducesOplaxAction}. It is given by $r:=m$, $r^0:=\rho$, and associator $r^2$ as shown.
\[
\vcenter{\hbox{\xymatrix@!0@=15mm{
RRR\ar[r]^-{r1}\ar[d]_-{1i^*\!1}\xtwocell[rd]{}<>{^r^2}&RR\ar[d]^-{r}\\
RR\ar[r]_-{r}&R
}}}
\quad:=\ 
\vcenter{\hbox{\xymatrix@!0@=15mm{
RRR\ar@/_3mm/[d]_-{1i^*\!1}\ar@/^3mm/[d]^-{1r}\xtwocell[d]{}<>{^1\kappa}\ar[r]^-{r1}\xtwocell[rd]{!<4mm,-3mm>}<>{^<-1>\alpha}&RR\ar[d]^-{r}\\
RR\ar[r]_-{r}&R\\
}}}
\]
\item Take $\mu_2$ to be the associator cell $\alpha$ and $\mu_0$ the cell corresponding to the left unitor as in Lemma~\ref{lem:BijectionLeftUnitVsMonadUnit}.
\[
\mu_2:=\alpha
\quad\qquad
\mu_0:=
\vcenter{\hbox{\xymatrix@!0@=15mm{
RRR\ar[r]^-{i^*\!11}\ar[d]_-{1i^*\!1}&RR\ar@/_3mm/[d]_-{i^*\!1}\ar@/^3mm/[d]^-{r}\xtwocell[d]{}<>{^\kappa}\\
RR\ar[r]_-{i^*\!1}\ar@{}[]!<-3mm,-0mm>;[ru]!<-3mm,-0mm>|*=0[@]{\cong}&R\\
}}}
\]
\end{enumerate}
Now, one needs to prove that the cells $\mu_0$ and $\mu_2$ satisfy axioms~\eqref{ax:2SIM1}, \eqref{ax:2SIM2}, and \eqref{ax:2SIM3}. The calculation below verifies axiom~\eqref{ax:2SIM1} for $\mu_2$,
\[
\vcenter{\hbox{\xymatrix@!0@C=15mm{
&RRR\ar[rd]^-{r1}&\\
RRRR\ar[ru]^-{r11}\ar[dd]_-{11r}\ar@/_3mm/[rd]_>>>{1i^*\!11}\ar@/^3mm/[rd]^<<{1r1}\xtwocell[rd]{!<3mm,-2mm>}<>{^1\kappa 1\ \ }\xtwocell[rr]{!<3mm,0mm>}<>{^\alpha 1\ }&&RR\ar[dd]^-{r}\\
&RRR\ar[dd]^-{1r}\ar[ru]^-{r1}\xtwocell[rd]{}<>{^\alpha\ }&\\
RRR\ar[rd]_-{1i^*\!1}\ar@{}[ru]|*=0[@]{\cong}&&R\\
&RR\ar[ru]_-{r}&
}}}
\stackrel{\eqref{ax:OLA2}}{=}
\vcenter{\hbox{\xymatrix@!0@C=15mm{
&RRR\ar[rd]^-{r1}&\\
RRRR\ar[ru]^-{r11}\ar[dd]_-{11r}\ar[rd]_-{1r1}\xtwocell[rddd]{}<>{^1\alpha\ \ }\xtwocell[rr]{}<>{^\alpha 1\ }&&RR\ar[dd]^-{r}\\
&RRR\ar[dd]^-{1r}\ar[ru]^-{r1}\xtwocell[rd]{}<>{^\alpha\ }&\\
RRR\ar@/_3mm/[rd]_-{1i^*\!1}\ar@/^3mm/[rd]^-{1r}\xtwocell[rd]{!<2mm,-1mm>}<>{^1\kappa\ }&&R\\
&RR\ar[ru]_-{r}&
}}}
\]\[
\stackrel{\eqref{ax:SKM1}}{=}
\vcenter{\hbox{\xymatrix@!0@C=15mm{
&RRR\ar[rd]^-{r1}\ar[dd]_-{1r}\xtwocell[rddd]{}<>{^\alpha\ }&\\
RRRR\ar[ru]^-{r11}\ar[dd]_-{11r}&&RR\ar[dd]^-{r}\\
&RR\ar[rd]^-{r}&\\
RRR\ar@/_3mm/[rd]_-{1i^*\!1}\ar@/^3mm/[rd]^<<<{1r}\xtwocell[rd]{!<2mm,-1mm>}<>{^1\kappa\ }\ar[ru]^-{r1}\ar@{}[ruuu]|*[@]{\cong}\xtwocell[rr]{!<3mm,0mm>}<>{^\alpha}&&R\\
&RR\ar[ru]_-{r}&
}}}
\]
axiom~\eqref{ax:2SIM2} for $\mu_2$ is a consequence of the following calculation,
\[
\vcenter{\hbox{\xymatrix@!0@C=15mm{
&RRR\ar[rd]^-{r1}&\\
RRRR\ar[ru]^-{r11}\ar[rd]|-{1r1}\ar@/_3mm/[dd]_-{11i^*\!1}\ar@/^3mm/[dd]^-{11r}\xtwocell[dd]{}<>{^11\kappa}\xtwocell[rddd]{!<2mm,-1mm>}<>{^1\alpha\ \ }\xtwocell[rr]{}<>{^\alpha 1\ }&&RR\ar[dd]^-{r}\\
&RRR\ar[dd]^-{1r}\ar[ru]^-{r1}\xtwocell[rd]{}<>{^\alpha}&\\
RRR\ar[rd]_-{1r}&&R\\
&RR\ar[ru]_-{r}&
}}}
\stackrel{\eqref{ax:SKM1}}{=}
\vcenter{\hbox{\xymatrix@!0@C=15mm{
&RRR\ar[rd]^-{r1}\ar[dd]_-{1r}\xtwocell[rddd]{}<>{^\alpha}&\\
RRRR\ar[ru]^-{r11}\ar@/_3mm/[dd]_-{11i^*\!1}\ar@/^3mm/[dd]^-{11r}\xtwocell[dd]{}<>{^11\kappa}&&RR\ar[dd]^-{r}\\
&RR\ar[rd]^-{r}&\\
RRR\ar[rd]_-{1r}\ar[ru]_-{r1}\ar@{}[ruuu]|*[@]{\cong}\xtwocell[rr]{}<>{^\alpha}&&R\\
&RR\ar[ru]_-{r}&
}}}
\]\[
=
\vcenter{\hbox{\xymatrix@!0@C=15mm{
&RRR\ar[rd]^-{r1}\ar@/_3mm/[dd]_-{1i^*\!1}\ar@/^3mm/[dd]^-{1r}\xtwocell[dd]{}<>{^1\kappa}\xtwocell[rddd]{}<>{^\alpha}&\\
RRRR\ar[ru]^-{r11}\ar[dd]_-{11i^*\!1}&&RR\ar[dd]^-{r}\\
&RR\ar[rd]^-{r}&\\
RRR\ar[rd]_-{1r}\ar[ru]_-{r1}\ar@{}[ruuu]|*[@]{\cong}\xtwocell[rr]{}<>{^\alpha}&&R\\
&RR\ar[ru]_-{r}&
}}}
\]
and axiom~\eqref{ax:2SIM3} for $\mu_2$ is literally axiom~\eqref{ax:SKM4} for the skew monoidale. Now, axiom~\eqref{ax:2SIM1} for $\mu_0$ follows from the interchange law as we mentioned in Lemma~\ref{lem:BijectionLeftUnitVsMonadUnit} above, axiom~\eqref{ax:2SIM2} for $\mu_0$ is verified as follows,
\begin{align*}
\vcenter{\hbox{\xymatrix@!0@C=15mm{
&RRR\ar@/_3mm/[rd]_-{i^*\!11}\ar@/^3mm/[rd]^-{r1}\xtwocell[rd]{!<2mm,-1mm>}<>{^\kappa 1\ }&\\
RRRR\ar[ru]^-{i^*\!111}\ar[dd]_-{11i^*\!1}\ar[rd]_-{1i^*\!11}&&RR\ar@/_3mm/[dd]_-{i^*\!1}\ar@/^3mm/[dd]^-{r}\xtwocell[dd]{}<>{^\kappa}\\
&RRR\ar[dd]^-{1i^*\!1}\ar[ru]|-{i^*\!11}\ar@{}[uu]^*[@]{\cong}&\\
RRR\ar[rd]_-{1i^*\!1}\ar@{}[ru]|*[@]{\cong}&&R\\
&RR\ar[ru]_-{i^*\!1}\ar@{}[ruuu]|*=0[@]{\cong}&
}}}
&\stackrel{\phantom{\eqref{ax:SKM3'}}}{=}
\vcenter{\hbox{\xymatrix@!0@C=15mm{
&RRR\ar[dd]_-{1i^*\!1}\ar@/_3mm/[rd]_>>>{i^*\!11}\ar@/^3mm/[rd]^-{r1}\xtwocell[rd]{!<2mm,-1mm>}<>{^\kappa 1\ }&\\
RRRR\ar[ru]^-{i^*\!111}\ar[dd]_-{11i^*\!1}&&RR\ar@/_3mm/[dd]_-{i^*\!1}\ar@/^3mm/[dd]^-{r}\xtwocell[dd]{}<>{^\kappa}\\
&RR\ar[rd]|-{i^*\!1}\ar@{}[ru]|*=0[@]{\cong}&\\
RRR\ar[rd]_-{1i^*\!1}\ar[ru]_-{i^*\!11}\ar@{}[ruuu]|*[@]{\cong}&&R\\
&RR\ar[ru]_-{i^*\!1}\ar@{}[uu]|*[@]{\cong}&
}}}
\\
=
\vcenter{\hbox{\xymatrix@!0@C=15mm{
&RRR\ar@/_3mm/[dd]_-{1i^*\!1}\ar@/^3mm/[dd]^>>>{1r}\xtwocell[dd]{}<>{^1\kappa}\ar@/_3mm/@[]!<2mm,-2mm>;[rd]_>>{i^*\!11}\ar@/^3mm/[rd]^-{r1}\xtwocell[rd]{!<2mm,-1mm>}<>{^\kappa 1\ }&\\
RRRR\ar[ru]^-{i^*\!111}\ar[dd]_-{11i^*\!1}&&RR\ar[dd]^-{r}\\
&RR\ar[rd]|-{i^*\!1}\ar@{}[ru]_*[@]{\cong}&\\
RRR\ar[rd]_-{1i^*\!1}\ar[ru]_-{i^*\!11}\ar@{}[ruuu]|*[@]{\cong}&&R\\
&RR\ar[ru]_-{i^*\!1}\ar@{}[uu]|*[@]{\cong}&
}}}
&\stackrel{\eqref{ax:SKM3'}}{=}
\vcenter{\hbox{\xymatrix@!0@C=15mm{
&RRR\ar@/_3mm/[dd]_-{1i^*\!1}\ar@/^3mm/[dd]^-{1r}\xtwocell[dd]{}<>{^1\kappa}\ar[rd]^-{r1}\xtwocell[rddd]{}<>{^\alpha}&\\
RRRR\ar[ru]^-{i^*\!111}\ar[dd]_-{11i^*\!1}&&RR\ar[dd]^-{r}\\
&RR\ar@/_3mm/[rd]_-{i^*\!1}\ar@/^3mm/[rd]^-{r}\xtwocell[rd]{!<2mm,-1mm>}<>{^\kappa}&\\
RRR\ar[rd]_-{1i^*\!1}\ar[ru]_-{i^*\!11}\ar@{}[ruuu]|*[@]{\cong}&&R\\
&RR\ar[ru]_-{i^*\!1}\ar@{}[uu]|*[@]{\cong}&
}}}
\end{align*}
and axiom~\eqref{ax:2SIM3} for $\mu_0$ holds true as one can see below.
\[
\vcenter{\hbox{\xymatrix@!0@C=7.5mm@R=4.3mm{
&&R\ar@/^7mm/[rrdddddd]^-{1}&&\\
&&&&\\
RR\ar[rruu]^-{i^*\!1}\ar[dddd]_-{11i}\ar[rddd]|-{1i^*}\ar@/^7mm/[rrdddddd]^-{1}\xtwocell[rrdddddd]{}<>{^<-2.5>1\varepsilon\ }&&&&\\
&&&&\\
&&&&\\
&R\ar[rddd]|-{1i}&&&\\
RRR\ar[rrdd]_-{1i^*\!1}\ar@{}[ruu]|*=0[@]{\cong}&&&&R\\
&&&&\\
&&RR\ar[rruu]_-{i^*\!1}&&
}}}
=
\vcenter{\hbox{\xymatrix@!0@C=7.5mm@R=4.3mm{
&&R\ar[dddd]_-{1i}\ar@/^7mm/[rrdddddd]^-{1}\ar[rddd]|-{i^*}\xtwocell[rrdddddd]{}<>{^<-2.5>\varepsilon\ }&&\\
&&&&\\
RR\ar[rruu]^-{i^*\!1}\ar[dddd]_-{11i}&&&&\\
&&&R\ar[rddd]|-{i}&\\
&&RR\ar[rrdd]_-{i^*\!1}\ar@{}[ruu]|*=0[@]{\cong}&&\\
&&&&\\
RRR\ar[rrdd]_-{1i^*\!1}\ar[rruu]^-{i^*\!11}\ar@{}[rruuuuuu]|*=0[@]{\cong}&&&&R\\
&&&&\\
&&RR\ar[rruu]_-{i^*\!1}\ar@{}[uuuu]|*=0[@]{\cong}&&
}}}
\]\[
\stackrel{\eqref{ax:SKM5'}}{=}
\vcenter{\hbox{\xymatrix@!0@C=15mm{
&R\ar@/^7mm/[rddd]^-{1}\ar[dd]_-{1i}\xtwocell[rddd]{}<>{^\rho}&\\
RR\ar[ru]^-{i^*\!1}\ar[dd]_-{11i}&&\\
&RR\ar@/_3mm/[rd]_-{i^*\!1}\ar@/^3mm/[rd]^-{r}\xtwocell[rd]{}<>{^\kappa\ }&\\
RRR\ar[rd]_-{1i^*\!1}\ar[ru]^-{i^*\!11}\ar@{}[ruuu]|*=0[@]{\cong}&&R\\
&RR\ar[ru]_-{i^*\!1}&
}}}
\]
\begin{enumerate}[resume]
\item This item requires the existence of three 3-simplices, which amounts to verifying the three axioms~\eqref{ax:M1}, \eqref{ax:M2}, and \eqref{ax:M3} for the data defined previously. Axiom~\eqref{ax:M1} is nothing but the pentagon axiom~\eqref{ax:SKM1} for $\alpha$, axiom~\eqref{ax:M2} happens to be the same as \eqref{ax:OLA1} for $r$ as verified in Proposition~\ref{prop:SkewMonoidaleInducesOplaxAction}, and the following calculation proves axiom~\eqref{ax:M3}.
\end{enumerate}
\[
\vcenter{\hbox{\xymatrix@!0@C=15mm{
&RRR\ar@/_3mm/[rd]_-{i^*\!11}\ar@/^3mm/[rd]^-{r1}\xtwocell[rd]{!<2mm,-1mm>}<>{^\kappa 1\ }&\\
RRRR\ar[ru]^-{i^*\!111}\ar[dd]_-{11r}\ar[rd]_-{1i^*\!11}&&RR\ar[dd]^-{r}\\
&RRR\ar[dd]^-{1r}\ar[ru]|-{i^*\!11}\ar@{}[uu]^*[@]{\cong}&\\
RRR\ar[rd]_-{1i^*\!1}\ar@{}[ru]|*[@]{\cong}&&R\\
&RR\ar[ru]_-{i^*\!1}\ar@{}[ruuu]|*=0[@]{\cong}&
}}}
=
\vcenter{\hbox{\xymatrix@!0@C=15mm{
&RRR\ar[dd]_-{1r}\ar@/_3mm/[rd]_>>>{i^*\!11}\ar@/^3mm/[rd]^-{r1}\xtwocell[rd]{!<2mm,-1mm>}<>{^\kappa 1\ }&\\
RRRR\ar[ru]^-{i^*\!111}\ar[dd]_-{11r}&&RR\ar[dd]^-{r}\\
&RR\ar[rd]|-{i^*\!1}\ar@{}[ru]|*=0[@]{\cong}&\\
RRR\ar[rd]_-{1i^*\!1}\ar[ru]_-{i^*\!11}\ar@{}[ruuu]|*[@]{\cong}&&R\\
&RR\ar[ru]_-{i^*\!1}\ar@{}[uu]|*[@]{\cong}&
}}}
\]\[
\stackrel{\eqref{ax:SKM3'}}{=}
\vcenter{\hbox{\xymatrix@!0@C=15mm{
&RRR\ar[dd]_-{1r}\ar[rd]^-{r1}\xtwocell[rddd]{}<>{^\alpha}&\\
RRRR\ar[ru]^-{i^*\!111}\ar[dd]_-{11r}&&RR\ar[dd]^-{r}\\
&RR\ar@/_3mm/[rd]_-{i^*\!1}\ar@/^3mm/[rd]^-{r}\xtwocell[rd]{!<2mm,-1mm>}<>{^\kappa}&\\
RRR\ar[rd]_-{1i^*\!1}\ar[ru]_-{i^*\!11}\ar@{}[ruuu]|*[@]{\cong}&&R\\
&RR\ar[ru]_-{i^*\!1}\ar@{}[uu]|*[@]{\cong}&
}}}
\]
This completes the description of the monad of oplax actions induced by a right skew monoidale. The right skew monoidale induced by this monad of oplax actions as in Proposition~\ref{prop:MonadInducesSkewMonoidale} is the same as the original right skew monoidale; all the structure is literally the same, except maybe for the left unitor $\lambda$ which by the bijection in Lemma~\ref{lem:BijectionLeftUnitVsMonadUnit} is verified to be the same as the induced one. Conversely, the monad of oplax actions built from a right skew monoidale (as done in this proof) which is induced by a monad of oplax actions using Proposition~\ref{prop:MonadInducesSkewMonoidale} is also the same monad as the original one. Again, most of the structure is literally equal except maybe the associator of the 1-simplex oplax action $r^2$, which by Lemma~\ref{lem:RedundantOplaxActionAssociator}, is verified to be the same as the induced one; and the cell $\mu_0$, which by Proposition~\ref{prop:MonadInducesSkewMonoidale}, is also verified to be the same as the induced one.
\end{proof}
\section{Oplax Actions and Opmonoidal Arrows}\label{sec:OplaxActionsOpmonoidalArrows}

In this section, we will address the point (2) that we made at the beginning of Subection~\ref{subsec:OplaxActions}, which states that under mild conditions on $\mathcal{M}$, a simplicial object of oplax actions is weakly equivalent in $[\Delta^{\mathrm{op}},\Cat]$ to the lax-2-nerve of a bicategory of opmonoidal arrows. For a moment, allow us to be vague about the particular simplicial objects in question and look at Theorem~\ref{teo:Opmon_is_OplaxAct_Local}; this result provides an equivalence of categories of the following form:
\[
\Opmon(R\ot R,S\ot S)\simeq\OplaxAct(R;S).
\]
Observe that on the left hand side we have a hom-category and on the right hand side we have a category of 1-simplices of oplax actions. Our goal is to extend this equivalence to a weak equivalence of simplicial objects in $\Cat$ in such a way that the weak equivalence at the level of 1-simplices is the equivalence of the aforementioned theorem. Thus, the rest of this section will consist of proving similar equivalences for the remaining categories of simplices. To achieve this, we impose the following conditions on $\mathcal{M}$:
\begin{enumerate}[label={(\alph*)}]
\item For every object $R$ in $\mathcal{M}$ there is a chosen right bidual $R\dashv R\ot$.
\item For every object $R$ in $\mathcal{M}$ there is a chosen adjunction $i\dashv i^*$ as shown.
\[
\vcenter{\hbox{\xymatrix{
R\dtwocell_{i}^{i^*}{'\dashv}\\
I
}}}
\]
\item For every object $R$ the opposite of its chosen adjunction is opmonadic.
\[
\vcenter{\hbox{\xymatrix{
R\ot\dtwocell_{i\ob}^{i\ot}{'\dashv}\\
I
}}}
\]
\item Tensoring with objects in $\mathcal{M}$ preserves opmonadicity.
\item Composing with arrows in $\mathcal{M}$ preserves coequalisers of reflexive pairs in the hom categories of $\mathcal{M}$.
\end{enumerate}

Using conditions (a) and (b) we can be explicit about which simplicial objects in $\Cat$ we are considering; it is not $\OplaxAct(\mathcal{M})$ nor the lax-2-nerve of $\Opmon(\mathcal{M})$, but a small alteration of each of them. This adjustment controls each of the collections of 0-simplices making them isomorphic to the set of objects of $\mathcal{M}$. Conditions (c)-(e) are a global version of the hypotheses of \cite[Theorem~5.2]{Lack2012} which are the same that those of Theorem~\ref{teo:Opmon_is_OplaxAct_Local}, and they let us control the rest of the categories of simplices.

Condition (a) is the same as requiring that $\mathcal{M}$ is right autonomous, provided a suitable version of the axiom of choice holds, as it implies that taking chosen right biduals is a strong monoidal pseudofunctor which is a local equivalence.
\[
\vcenter{\hbox{\xymatrix{
\mathcal{M}^{\mathrm{rev op}}\ar[r]^-{(\ )\ot}&\mathcal{M}
}}}
\]
Using condition (a) consider the bicategory $\Opmon^{\mathrm{e}}(\mathcal{M})$ of the enveloping monoidales $R\ot R$ induced by the \emph{chosen} bidualities, opmonoidal arrows between them, and opmonoidal cells between them; it is a full subbicategory of $\Opmon(\mathcal{M})$. We denote by $\Opmon^{\mathrm{e}}_n$ the category of $n$-simplices of the lax-2-nerve of $\Opmon^{\mathrm{e}}(\mathcal{M})$. Furthermore, the collection of objects of $\Opmon^{\mathrm{e}}(\mathcal{M})$ is isomorphic to that of $\mathcal{M}$, in other words $\Opmon^{\mathrm{e}}_0\cong\Ob\mathcal{M}$.

Condition (b) allows us to consider the subsimplicial object $\OplaxAct^{\mathrm e}(\mathcal{M})$ of $\OplaxAct(\mathcal{M})$ whose simplices are all those simplices in $\OplaxAct(\mathcal{M})$ that have 0-faces $(R,i\dashv i^*)$, where $i\dashv i^*$ is the \emph{chosen} adjunction for $R$. We denote by $\OplaxAct^{\mathrm{e}}_n$ the category of $n$-simplices of $\OplaxAct^{\mathrm{e}}(\mathcal{M})$. Now both the collection of 0-simplices of the lax-2-nerve of $\Opmon^{\mathrm e}(\mathcal{M})$ and the collection of 0-simplices of $\OplaxAct^{\mathrm e}(\mathcal{M})$ are isomorphic to the collection of objects of $\mathcal{M}$.
\[
\Opmon^{\mathrm{e}}_0\cong\Ob\mathcal{M}\cong\OplaxAct^{\mathrm e}_0
\]

Condition (c) might seem slightly artificial, but let us recall an example to see that it is quite reasonable. In $\Mod_k$ we may pick as our chosen adjunctions for $k$-algebras $R$ the ones induced by the unit morphism $\xymatrix@1@C=5mm{i:k\ar[r]&R}$ which, as any other $k$-algebra morphism, defines an adjunction $i\dashv i^*$ in $\Mod_k$. In this case, this adjunction and its opposite adjunction are monadic and opmonadic. In particular, opmonadicity of $i\ob\dashv i\ot$ states that right $R\ot$-modules may be viewed as arrows with target $R\ot$ in $\Mod_k$ or as a $k$-algebra together with a left $R\ot$-action.

Conditions (d) and (e) ensure the compatibility between opmonadicity and the rest of the structure of $\mathcal{M}$, and are what we called opmonadic-friendly monoidal bicategory in Definition~\ref{def:OpmonadicFriendly}. Now, Theorem~\ref{teo:Opmon_is_OplaxAct_Local} implies that the categories of 1-simplices are equivalent, and furthermore, these equivalences commute with the face and degeneracy functors.
\begin{equation}\label{eq:FacesDegeneracies01}
\vcenter{\hbox{\xymatrix@!0@R=25mm@C=35mm{
\Opmon^{\mathrm{e}}_0\ar[d]|-{\cs_0}\ar[r]_-{\cong}&\OplaxAct^{\mathrm{e}}_0\ar[d]|-{\cs_0}\\
\Opmon^{\mathrm{e}}_1\ar@<3mm>[u]|-{\partial_0}\ar@<-3mm>[u]|-{\partial_1}\ar[r]^-{\simeq}&\OplaxAct^{\mathrm{e}}_1\ar@<3mm>[u]|-{\partial_0}\ar@<-3mm>[u]|-{\partial_1}
}}}
\end{equation}
\begin{rem}\label{rem:degeneracies}
It is easy to see that square commutes strictly with respect to the face functors. However, it commutes only up to isomorphism with respect to the degeneracy functors. If we are to construct a genuine simplicial morphism, the square must commute strictly with the degeneracy functors too. For this to happen one must modify the definition of the degenerate 1-simplices of $\OplaxAct(\mathcal{M})$ in the following way: instead of having $\cs_0(R,i\dashv i^*)=\xymatrix@1@C=5mm{i^*\!1:RR\ar[r]&R}$ one has to take the composite below.
\[
\vcenter{\hbox{\xymatrix{
\cs_0(R,i\dashv i^*)=RR\ar[r]^-{1i\ob 1}&RR\ot R\ar[r]^-{1}&RR\ot R\ar[r]^-{e1}&R,
}}}
\]
But this approach is inconvenient since it obscures our calculations. This behaviour continues to happen for the other dimensions and similar adjustments may be done for the higher degeneracy functors. Thus, we are going to show the existence of a pseudosimplicial morphism between the lax-2-nerve of $\Opmon^{\mathrm{e}}(\mathcal{M})$ and $\OplaxAct^{\mathrm{e}}(\mathcal{M})$ which may be easily strictified by changing how the degeneracies of $\OplaxAct^{\mathrm{e}}(\mathcal{M})$ are defined.
\end{rem}

To prove that the categories of $n$-simplices for $n\in\{2,3\}$ of the lax-2-nerve of $\Opmon^{\mathrm{e}}(\mathcal{M})$ are equivalent to those of $\OplaxAct^{\mathrm{e}}(\mathcal{M})$, no further assumptions are required on $\mathcal{M}$. We first simplify the task by showing that the categories of $n$-simplices \emph{with fixed 0-faces} are equivalent. Then, these equivalences break down into two steps that are similar to those used for 1-simplices in Theorem~\ref{teo:Opmon_is_OplaxAct_Local}; one using the transposition along bidualities, and another using the opmonadicity of certain adjunctions. In the case of 1-simplices, these two steps could be performed in any order, this fact manifests as the commutative square in \cite[Corollary~6.11]{Abud2018}. However, it seems that for $n$-simplices there is only one possible way to perform these steps, which is: starting from the opmonoidal simplices, transpose first and then use opmonadicity. We begin with the case of 2-simplices by introducing the categories that form part of this process.

\begin{notat}
Limited by the shortage of letters in the Latin and Greek alphabets, we introduce a naming convention for the opmonoidal arrows and oplax actions that are of the type that correspond to each other under \cite[Corollary~6.11]{Abud2018}, even when its hypotheses are not satisfied. We name them with the same letter and differentiate them with mathematical accents; for example,
\begin{center}
\begin{tabular}{c p{85mm} }
$
\vcenter{\hbox{\xymatrix@R=0mm{
\bar{s}:R\ot R\ar[r]&S\ot S\\
}}}
$
&Opmonoidal arrow between enveloping monoidales.\\
$
\vcenter{\hbox{\xymatrix@R=0mm{
\widehat{s}:SR\ot R\ar[r]&S\\
}}}
$
&Oplax Action with respect to an enveloping monoidale $R\ot R$.\\
$
\vcenter{\hbox{\xymatrix@R=0mm{
s:SR\ar[r]&S
}}}
$
&Oplax action with respect to a skew monoidale $R$ induced by an adjunction $i\dashv i^*$.
\end{tabular}
\end{center}
This will help the reader to figure out which things ought to be the same, as well as not to lose focus by going back to the definitions too often to figure out what is each thing with a new name.
\end{notat}

\begin{defi}
For three bidualities $R\dashv R\ot$, $S\dashv S\ot$, and $T\dashv T\ot$ the category denoted by $\Opmon(R\ot R,S\ot S,T\ot T)$ is the category of 2-simplices with fixed 0-faces $R\ot R$, $S\ot S$, and $T\ot T$ of the lax-2-nerve of $\Opmon(\mathcal{M})$. An object in this category consists of three opmonoidal arrows $(\bar{s},\bar{t},\bar{v})$, and an opmonoidal cell $\bar{\alpha}$ with a triangular shape,
\[
\vcenter{\hbox{\xymatrix@!0@C=12mm@R=5mm{
R\ot R\ar[dr]_{\bar{s}}\ar@/^4mm/[rr]^{\bar{v}}\xtwocell[rr]{}<>{^{\bar{\alpha}}}&&T\ot T\\
&S\ot S\ar[ur]_{\bar{t}}&
}}}
\]
while a morphism is a triple $(\bar{\sigma},\bar{\tau},\bar{\nu})$ of opmonoidal cells that satisfy the following equation.
\begin{equation}
\tag{OM6}\label{ax:OM6}
\vcenter{\hbox{\xymatrix@R=4mm@C=10mm{
R\ot R\ar@/^3mm/[rd]^-{\bar{s}'}\ar@/_3mm/[rd]_-{\bar{s}}\xtwocell[rd]{}<>{^\bar{\sigma}}\ar@/^6mm/[rr]^-{\bar{v}'}\xtwocell[rr]{}<>{^<-1>\bar{\alpha}'}&&T\ot T\\
&S\ot S\ar@/^3mm/[ru]^-{\bar{t}'}\ar@/_3mm/[ru]_-{\bar{t}}\xtwocell[ru]{}<>{^\bar{\tau}}&
}}}
\quad=\quad
\vcenter{\hbox{\xymatrix@R=4mm@C=10mm{
R\ot R\ar@/_3mm/[rd]_-{\bar{s}}\ar@/^6mm/[rr]^-{\bar{v}'}\ar@/_1mm/[rr]|{\bar{v}}\xtwocell[rr]{}<>{^<-1.5>\bar{\nu}}\xtwocell[rr]{}<>{^<2.5>\bar{\alpha}}&&T\ot T\\
&S\ot S\ar@/_3mm/[ru]_-{\bar{t}}&
}}}
\end{equation}
\end{defi}

\begin{defi}
For two enveloping monoidales $R\ot R$ and $S\ot S$ induced by bidualities $R\dashv R\ot$ and $S\dashv S\ot$, and an object $T$ in $\mathcal{M}$, define a category $\mathcal{A}(R;S;T)$. An object consists of three oplax actions,
\[
\vcenter{\hbox{\xymatrix{
\widehat{s}:SR\ot R\ar[r]&S
}}}
\qquad
\vcenter{\hbox{\xymatrix{
\widehat{t}:TS\ot S\ar[r]&T
}}}
\qquad
\vcenter{\hbox{\xymatrix{
\widehat{v}:TR\ot R\ar[r]&T
}}}
\]
and a quadrangular cell $\widehat{\alpha}$,
\[
\vcenter{\hbox{\xymatrix@!0@=15mm{
**[l]TS\ot SR\ot R\ar[r]^-{\widehat{t}11}\ar[d]_-{11\widehat{s}}\xtwocell[rd]{}<>{^\widehat{\alpha}}&TR\ot R\ar[d]^-{\widehat{v}}\\
TS\ot S\ar[r]_-{\widehat{t}}&T
}}}
\]
satisfying three axioms.
\begin{align}
\tag{A1}\label{ax:A1}
\vcenter{\hbox{\xymatrix@!0@C=15mm{
&TS\ot SR\ot R\ar[rd]^>>>{\widehat{t}11}&\\
**[l]TS\ot SS\ot SR\ot R\ar[ru]^-{\widehat{t}1111}\ar[dd]_-{1111\widehat{s}}\ar[rd]|-{11e111}\xtwocell[rr]{}<>{^\widehat{t}^211\quad}&&TR\ot R\ar[dd]^-{\widehat{v}}\\
&TS\ot SR\ot R\ar[dd]^-{11\widehat{s}}\ar[ru]^-{\widehat{t}11}\xtwocell[rd]{}<>{^\widehat{\alpha}}&\\
**[l]TS\ot SS\ot S\ar[rd]_-{11e1}\ar@{}[ru]|*=0[@]{\cong}&&T\\
&TS\ot S\ar[ru]_-{\widehat{t}}&
}}}
\quad&=\hspace{-3mm}
\vcenter{\hbox{\xymatrix@!0@C=15mm{
&TS\ot SR\ot R\ar[rd]^>>>{\widehat{t}11}\ar[dd]_-{11\widehat{s}}\xtwocell[rddd]{}<>{^\widehat{\alpha}}&\\
**[l]TS\ot SS\ot SR\ot R\ar[ru]^-{\widehat{t}1111}\ar[dd]_-{1111\widehat{s}}&&TR\ot R\ar[dd]^-{\widehat{v}}\\
&TS\ot S\ar[rd]^-{\widehat{t}}&\\
**[l]TS\ot SS\ot S\ar[rd]_-{11e1}\ar[ru]|-{\widehat{t}11}\ar@{}[ruuu]|*=0[@]{\cong}\xtwocell[rr]{}<>{^\widehat{t}^2\ }&&T\\
&TS\ot S\ar[ru]_-{\widehat{t}}&
}}}
\\
\tag{A2}\label{ax:A2}
\vcenter{\hbox{\xymatrix@!0@C=15mm{
&TR\ot RR\ot R\ar[rd]^>>>{\widehat{v}11}&\\
**[l]TS\ot SR\ot RR\ot R\ar[ru]^-{\widehat{t}1111}\ar[dd]_-{1111e1}\ar[rd]|-{11\widehat{s}11}\xtwocell[rddd]{}<>{^11\widehat{s}^2\quad}\xtwocell[rr]{}<>{^\widehat{\alpha} 11\quad}&&TR\ot R\ar[dd]^-{\widehat{v}}\\
&TS\ot SR\ot R\ar[dd]^-{11\widehat{s}}\ar[ru]^-{\widehat{t}11}\xtwocell[rd]{}<>{^\widehat{\alpha}}&\\
**[l]TS\ot SR\ot R\ar[rd]_-{11\widehat{s}}&&T\\
&TS\ot S\ar[ru]_-{\widehat{t}}&
}}}
\quad&=\hspace{-3mm}
\vcenter{\hbox{\xymatrix@!0@C=15mm{
&TR\ot RR\ot R\ar[rd]^>>>{\widehat{v}11}\ar[dd]_-{11e1}\xtwocell[rddd]{}<>{^\widehat{v}^2}&\\
**[l]TS\ot SR\ot RR\ot R\ar[ru]^-{\widehat{t}1111}\ar[dd]_-{1111e1}&&TR\ot R\ar[dd]^-{\widehat{v}}\\
&TR\ot R\ar[rd]^-{\widehat{v}}&\\
**[l]TS\ot SR\ot R\ar[rd]_-{11\widehat{s}}\ar[ru]_-{\widehat{t}11}\ar@{}[ruuu]|*=0[@]{\cong}\xtwocell[rr]{}<>{^\widehat{\alpha}}&&T\\
&TS\ot S\ar[ru]_-{\widehat{t}}&
}}}
\\
\tag{A3}\label{ax:A3}
\vcenter{\hbox{\xymatrix@!0@C=15mm{
&T\ar@/^7mm/[dddr]^-{1}&\\
TS\ot S\ar[ru]^-{\widehat{t}}\ar[dd]_-{111n}\ar@/^7mm/[dddr]^-{1}\xtwocell[rddd]{}<>{^11\widehat{t}^0\quad}&&\\
&&\\
**[l]TS\ot SR\ot R\ar[rd]_-{11\widehat{s}}&&T\phantom{R\ot R}\\
&TS\ot S\ar[ru]_-{\widehat{t}}&
}}}
\quad&=\hspace{-3mm}
\vcenter{\hbox{\xymatrix@!0@C=15mm{
&T\ar@/^7mm/[dddr]^-{1}\ar[dd]_-{1n}\xtwocell[rddd]{}<>{^\widehat{v}^0\ }&\\
TS\ot S\ar[ru]^-{\widehat{t}}\ar[dd]_-{111n}&&\\
&TR\ot R\ar[rd]^-{\widehat{v}}&\\
**[l]\phantom{R\ot R}TS\ot SR\ot R\ar[rd]_-{11\widehat{s}}\ar[ru]_-{\widehat{t}11}\ar@{}[ruuu]|*=0[@]{\cong}\xtwocell[rr]{}<>{^\widehat{\alpha}}&&T\\
&TS\ot S\ar[ru]_-{\widehat{t}}&
}}}
\end{align}
The morphisms between them consist of three cells of oplax actions
\[
\vcenter{\hbox{\xymatrix{
**[l]R\ot R\ar@/^4mm/[r]|-@{|}^-{\widehat{s}'}\ar@/_4mm/[r]|-@{|}_-{\widehat{s}}\xtwocell[r]{}<>{^\widehat{\sigma}}\ar@{}@<-.5mm>[r]|{-}&S
}}}
\qquad
\vcenter{\hbox{\xymatrix{
**[l]S\ot S\ar@/^4mm/[r]|-@{|}^-{\widehat{t}'}\ar@/_4mm/[r]|-@{|}_-{\widehat{t}}\xtwocell[r]{}<>{^\widehat{\tau}}\ar@{}@<-.5mm>[r]|{-}&T
}}}
\qquad
\vcenter{\hbox{\xymatrix{
**[l]R\ot R\ar@/^4mm/[r]|-@{|}^-{\widehat{v}'}\ar@/_4mm/[r]|-@{|}_-{\widehat{v}}\xtwocell[r]{}<>{^\widehat{\nu}}\ar@{}@<-.5mm>[r]|{-}&T
}}}
\]
satisfying the following equation.
\begin{equation}
\tag{A4}\label{ax:A4}
\vcenter{\hbox{\xymatrix@!0@=19mm{
**[l]TS\ot SR\ot R\ar@/^3mm/[r]^-{\widehat{t}'11}\ar@/^3mm/[d]^-{11\widehat{s}'}\ar@/_3mm/[d]_-{11\widehat{s}}\xtwocell[d]{}<>{^11\widehat{\sigma}}\xtwocell[rd]{}<>{^<-2>\widehat{\alpha}'}&TR\ot R\ar@/^3mm/[d]^-{\widehat{v}'}\\
TS\ot S\ar@/^3mm/[r]^-{\widehat{t}'}\ar@/_3mm/[r]_-{\widehat{t}}\xtwocell[r]{}<>{^\widehat{\tau}}&T
}}}
\quad\ \ \ =\quad\!\!\!
\vcenter{\hbox{\xymatrix@!0@=19mm{
**[l]TS\ot SR\ot R\ar@/^3mm/[r]^-{\widehat{t}'11}\ar@/_3mm/[r]_-{\widehat{t}11}\xtwocell[r]{}<>{^\widehat{\tau} 11\quad}\ar@/_3mm/[d]_-{1\widehat{s}}\xtwocell[rd]{}<>{^<2>\widehat{\alpha}}&TR\ot R\ar@/^3mm/[d]^-{\widehat{v}'}\ar@/_3mm/[d]_-{\widehat{v}}\xtwocell[d]{}<>{^\widehat{\nu}}\\
TS\ot S\ar@/_3mm/[r]_-{\widehat{t}}&T
}}}
\end{equation}
Composition and identities are given as in the category
\[
\OplaxAct(R\ot R;S)\times\OplaxAct(S\ot S;T)\times\OplaxAct(R\ot R;T).
\]
\end{defi}

\begin{rem}\label{rem:SkewMonBicaty}
At this point it is worth mentioning that oplax actions with respect to enveloping monoidales $\xymatrix@1@C=5mm{\hat{s}:SR\ot R\ar[r]&S}$ may be rewritten using the product
\[
X\underset{R}{\circ}Y:=XR\ot Y
\]
defined for a biduality $R\dashv R\ot$ in $\mathcal{M}$. So, the data for an oplax action is now an arrow $\xymatrix@1@C=5mm{\widehat{s}:S\underset{R}{\circ}R\ar[r]&S}$ and cells as shown,
\[
\vcenter{\hbox{\xymatrix@!0@=15mm{
**[l]S\underset{R}{\circ}R\underset{R}{\circ}R\ar[r]^-{\widehat{s}\underset{R}{\circ}1}\ar[d]_-{1\underset{R}{\circ}\ell}\xtwocell[rd]{}<>{^\widehat{s}^2\ }&S\underset{R}{\circ}R\ar[d]^-{\widehat{s}}\\
S\underset{R}{\circ}R\ar[r]_-{\widehat{s}}&S
}}}
\qquad
\vcenter{\hbox{\xymatrix@!0@=15mm{
S\underset{R}{\circ}R\ar[d]_-{\widehat{s}}&S\ar[l]_-{\mathfrak{r}}\ar[dl]^-{1}\xtwocell[ld]{}<>{^<2>\ \ \widehat{s}^0}\\
S
}}}
\]
where $\xymatrix@1@C=5mm{\ell=e1:RR\ot R\ar[r]&R}$ and $\xymatrix@1@C=5mm{\mathfrak{r}=1n:S\ar[r]&SR\ot R}$. In the same way, we may also rewrite the objects $(\widehat{s},\widehat{t},\widehat{v},\widehat{\alpha})$ of the categories $\mathcal{A}(R;S;T)$ defined above. Hence, the cell $\widehat{\alpha}$ becomes
\[
\vcenter{\hbox{\xymatrix@!0@=15mm{
**[l]T\underset{S}{\circ}S\underset{R}{\circ}R\ar[r]^-{\widehat{t}\underset{R}{\circ}1}\ar[d]_-{1\underset{S}{\circ}\widehat{s}}\xtwocell[rd]{}<>{^\widehat{\alpha}}&T\underset{R}{\circ}R\ar[d]^-{\widehat{v}}\\
T\underset{S}{\circ}S\ar[r]_-{\widehat{t}}&T
}}}
\]
and the axioms change accordingly. For example, axiom~\eqref{ax:A1} becomes the following equation.
\[
\vcenter{\hbox{\xymatrix@!0@C=19mm@R=10.5mm{
&T\underset{S}{\circ} S\underset{R}{\circ} R\ar[rd]^-{\widehat{t}\underset{R}{\circ}1}&\\
**{!<7mm>}T\underset{S}{\circ} S\underset{S}{\circ} S\underset{R}{\circ} R\ar[ru]^-{\widehat{t}\underset{S}{\circ}1\underset{R}{\circ}1}\ar[dd]_-{1\underset{S}{\circ}1\underset{S}{\circ}\widehat{s}}\ar[rd]|-{1\underset{S}{\circ}\ell\underset{R}{\circ}1}\xtwocell[rr]{}<>{^\widehat{t}^2\underset{R}{\circ}1\quad}&&T\underset{R}{\circ} R\ar[dd]^-{\widehat{v}}\\
&T\underset{S}{\circ} S\underset{R}{\circ} R\ar[dd]^-{1\underset{S}{\circ}\widehat{s}}\ar[ru]|-{\widehat{t}\underset{R}{\circ}1}\xtwocell[rd]{}<>{^\widehat{\alpha}}&\\
**[l]T\underset{S}{\circ} S\underset{S}{\circ} S\ar[rd]_-{1\underset{S}{\circ}\ell}\ar@{}[ru]|*=0[@]{\cong}&&T\\
&T\underset{S}{\circ} S\ar[ru]_-{\widehat{t}}&
}}}
\hspace{2mm}=\hspace{-5mm}
\vcenter{\hbox{\xymatrix@!0@C=19mm@R=10.5mm{
&T\underset{S}{\circ} S\underset{R}{\circ} R\ar[rd]^-{\widehat{t}\underset{R}{\circ}1}\ar[dd]_-{1\underset{S}{\circ}\widehat{s}}\xtwocell[rddd]{}<>{^\widehat{\alpha}}&\\
**{!<5mm>}T\underset{S}{\circ} S\underset{S}{\circ} S\underset{R}{\circ} R\ar[ru]^-{\widehat{t}\underset{S}{\circ}1\underset{R}{\circ}1}\ar[dd]_-{1\underset{S}{\circ}1\underset{S}{\circ}\widehat{s}}&&T\underset{R}{\circ} R\ar[dd]^-{\widehat{v}}\\
&T\underset{S}{\circ} S\ar[rd]^-{\widehat{t}}&\\
**[l]T\underset{S}{\circ} S\underset{S}{\circ} S\ar[rd]_-{1\underset{S}{\circ}\ell}\ar[ru]|-{\widehat{t}\underset{S}{\circ}1}\ar@{}[ruuu]|*=0[@]{\cong}\xtwocell[rr]{}<>{^\widehat{t}^2\ }&&T\\
&T\underset{S}{\circ} S\ar[ru]_-{\widehat{t}}&
}}}
\]
This approach is taken by Lack and Street in \cite[Section 5]{Lack2012} where they use the fact that the product $\circ_R$ turns the bicategory $\mathcal{M}$ into a \emph{skew monoidal bicategory} which they call $\mathcal{M}_R$, where the skew unit is $R$. It is possible to define oplax actions in skew monoidal bicategories, and under this perspective, oplax actions with respect to an enveloping monoidale $R\ot R$ in $\mathcal{M}$ are oplax actions with respect to the unit monoidale in $\mathcal{M}_R$. We refer the reader to Lack and Street's paper for more about skew monoidal bicategories. A full definition may be found in \cite[Definition 5.2]{Buckley2016}.
\end{rem}

We are now ready to address the equivalence between the categories of 2-simplices of opmonoidal arrows and oplax actions, and we begin with the transposition step mentioned earlier which in fact does not require any of the extra assumptions on the monoidal bicategory $\mathcal{M}$. This transposition step consists of an equivalence of categories, and the map that defines it is not trivial. The idea behind it comes from examining the equivalence in \cite[Theorem 5.1]{Lack2012} between opmonoidal monads on an enveloping monoidale $R\ot R$ in $\mathcal{M}$ and right skew monoidales with a given skew unit on the object $R$ in the skew monoidal bicategory $\mathcal{M}_R$. One has to look at how the multiplication of an opmonoidal monad is mapped to the associator of the corresponding right skew monoidale and generalise appropriately.

\begin{teo}\label{teo:Opmon2simplexTransposition}
For every three bidualities $R\dashv R\ot$, $S\dashv S\ot$ and $T\dashv T\ot$ in a monoidal bicategory $\mathcal{M}$
there is an equivalence of categories,
\[
\Opmon(R\ot R,S\ot S,T\ot T)\simeq\mathcal{A}(R;S;T)
\]
given on objects by
\[
\vcenter{\hbox{\xymatrix@!0@C=12mm@R=5mm{
R\ot R\ar[dr]_{\bar{s}}\ar@/^4mm/[rr]^{\bar{v}}\xtwocell[rr]{}<>{^{\bar{\alpha}}}&&T\ot T\\
&S\ot S\ar[ur]_{\bar{t}}&
}}}
\vcenter{\hbox{\xymatrix{
\ar@{|->}[r]^-{P}&
}}}
\vcenter{\hbox{\xymatrix@!0@=23mm{
TS\ot SR\ot R\ar[r]^-{1\bar{t}11}\ar[d]|-{111\bar{s}}\ar@/^9mm/[rr]^-{\widehat{t}11}\ar@/_9mm/[dd]_-{11\widehat{s}}\xtwocell[dr]{}<>{^1\bar{t}\bar{\alpha}\quad}&TT\ot TR\ot R\ar[d]|-{111\bar{v}}\ar[r]^-{e111}&TR\ot R\ar[d]^-{1\bar{v}}\ar@/^9mm/[dd]^-{\widehat{v}}\\
TS\ot SS\ot S\ar[r]^-{1\bar{t}\bar{t}}\ar[d]|-{11e1}\xtwocell[dr]{}<>{^\bar{t}^2}&TT\ot TT\ot T\ar[d]|-{11e1}\ar[r]^-{e111}\ar@{}[ru]|*[@]{\cong}&TT\ot T\ar[d]^-{e1}\\
TS\ot S\ar[r]_-{1\bar{t}}\ar@/_9mm/[rr]_-{\widehat{t}}&TT\ot T\ar[r]_-{e1}\ar@{}[ru]|*[@]{\cong}&T
}}}
\]
where $\widehat{s}$ is obtained by transposition along $S\dashv S\ot$, and $\widehat{t}$ and $\widehat{v}$ are obtained by transposition along $T\dashv T\ot$.
\end{teo}
\begin{proof}

By \cite[Theorem~6.7]{Abud2018} one has the following equivalence of categories,
\begin{multline}
\label{eq:P1}
\Opmon(R\ot R,S\ot S)\times\Opmon(S\ot S,T\ot T)\times\Opmon(R\ot R,T\ot T)\\
\simeq\OplaxAct(R\ot R;S)\times\OplaxAct(S\ot S;T)\times\OplaxAct(R\ot R;T)
\end{multline}
which means that triples $(\bar{s},\bar{t},\bar{v})$ of opmonoidal arrows as in the statement are already in equivalence with their transpose oplax actions $(\widehat{s},\widehat{t},\widehat{v})$, so now one may focus on their fillings.

Hence, fix a triple $(\bar{s},\bar{t},\bar{v})$ of opmonoidal arrows on the left hand side of $\eqref{eq:P1}$, on the right hand side, fix the corresponding triple $(\widehat{s},\widehat{t},\widehat{v})$ of oplax actions, and fix isomorphisms as shown below that witness the correspondence.
\begin{equation}\label{iso:transposition}
\vcenter{\hbox{\xymatrix@!0@C=15mm{
R\ot R\ar[rd]_-{n11}\ar[rr]^-{\bar{s}}&&S\ot S\\
&S\ot SR\ot R\ar[ru]_-{1\widehat{s}}\ar@{}[u]|>>>*=0[@]{\cong}&
}}}
\vcenter{\hbox{\xymatrix@!0@C=15mm{
&SS\ot S\ar[rd]^-{e1}\ar@{}[d]|>>>*=0[@]{\cong}&\\
SR\ot R\ar[ru]^-{1\bar{s}}\ar[rr]_-{\widehat{s}}&&S
}}}
\end{equation}
\begin{claim}
The assignation $P$ defines a bijection between the set of opmonoidal cells $\bar{\alpha}$ and the set of cells $\widehat{\alpha}$ that satisfy \eqref{ax:A1}, \eqref{ax:A2}, and \eqref{ax:A3}.
\begin{multline*}
\left\lbrace
\vcenter{\hbox{\xymatrix@!0@C=12mm@R=5mm{
R\ot R\ar[dr]_{\bar{s}}\ar@/^4mm/[rr]^{\bar{v}}\xtwocell[rr]{}<>{^{\bar{\alpha}}}&&T\ot T\\
&S\ot S\ar[ur]_{\bar{t}}&
}}}
\vrule\ \eqref{ax:OM4},\eqref{ax:OM5}
\right\rbrace\\
\cong
\left\lbrace
\vcenter{\hbox{\xymatrix@!0@=14mm{
**[l]TS\ot SR\ot R\ar[r]^-{\widehat{t}11}\ar[d]_-{11\widehat{s}}\xtwocell[rd]{}<>{^\widehat{\alpha}}&TR\ot R\ar[d]^-{\widehat{v}}\\
TS\ot S\ar[r]_-{\widehat{t}}&T
}}}
\vrule\ \eqref{ax:A1},\eqref{ax:A2},\eqref{ax:A3}
\right\rbrace
\end{multline*}
\end{claim}
\begin{proof}[Sketch]



Let $\bar{\alpha}$ be a cell as in the statement of the claim, then axiom~\eqref{ax:OM4} for $\bar{\alpha}$ implies that $P(\bar{\alpha})$ satisfies
axiom~\eqref{ax:A1}, axiom~\eqref{ax:OM4} for $\bar{\alpha}$ implies axiom~\eqref{ax:A2} for $P(\bar{\alpha})$, and axiom~\eqref{ax:OM5} for $\bar{\alpha}$ implies axiom~\eqref{ax:A3} for $P(\bar{\alpha})$, thus $P$ is well defined. Let $P'$ be the function given as follows.
\[
\vcenter{\hbox{\xymatrix@!0@=15mm{
&T\ot TR\ot R\ar[rd]|-{11n11}\ar@/^9mm/[rr]^-{1}\xtwocell[rr]{!<4mm,0mm>}<>{^1\widehat{t}^011\quad\ }&&T\ot TR\ot R\ar[rd]^-{1\widehat{v}}&\\
R\ot R\ar[ru]^-{n11}\ar[rd]_-{n11}&&T\ot TS\ot SR\ot R\ar[ru]|-{1\widehat{t}11}\ar[rd]|-{111\widehat{s}}\xtwocell[rr]{!<3mm,0mm>}<>{^1\widehat{\alpha}\ }&&T\ot T\\
&S\ot SR\ot R\ar[ru]|-{n1111}\ar[rd]_-{1\widehat{s}}\ar@{}[uu]|*[@]{\cong}&&T\ot TS\ot S\ar[ru]_-{1\widehat{t}}&\\
&&S\ot S\ar[ru]_-{n11}\ar@{}[uu]|*[@]{\cong}&&
}}}
\vcenter{\hbox{\xymatrix{
&\ar@{|->}[l]_-{P'}
}}}
\vcenter{\hbox{\xymatrix@!0@=15mm{
**[l]TS\ot SR\ot R\ar[r]^-{\widehat{t}11}\ar[d]_-{11\widehat{s}}\xtwocell[rd]{}<>{^\widehat{\alpha}}&TR\ot R\ar[d]^-{\widehat{v}}\\
TS\ot S\ar[r]_-{\widehat{t}}&T
}}}
\]
To show that $P'$ is well defined let $\widehat{\alpha}$ be a cell in the domain of $P'$, then $P'(\widehat{\alpha})$ satisfies axiom~\eqref{ax:OM4} as a consequence of \eqref{ax:A2} for $\widehat{\alpha}$, and axiom~\eqref{ax:OM5} as a consequence of \eqref{ax:A3} for $\widehat{\alpha}$. Finally, the assignations $P$ and $P'$ are inverse to each other; let $\bar{\alpha}$ in the domain of $P$ and consider the calculation below.
\[
\vcenter{\hbox{\xymatrix@!0@=8mm{
&&&&&&&&T\ot TR\ot R\ar[rrdd]^-{11\bar{v}}&&&&\\
&&&&&&&&&&&&\\
&&T\ot TR\ot R\ar[rrdd]|-{11n11}\ar@/^5mm/[rrrr]^-{11n11}\ar@/^9mm/[rrrrrruu]^-{1}\xtwocell[rrrr]{!<5mm,0mm>}<>{^<1>11\bar{t}^011\qquad}&&&&T\ot TT\ot TR\ot R\ar[rruu]|-{1e111}\ar[rrdd]|-{1111\bar{v}}\ar@{}[uu]|>>>*=0[@]{\cong}&&&&T\ot TT\ot T\ar[rrdd]^-{1e1}&&\\
&&&&&&&&&&&&\\
R\ot R\ar[rruu]^-{n11}\ar[rrdd]|-{n11}\ar[rddd]_-{\bar{s}}&&&&T\ot TS\ot SR\ot R\ar[rruu]|-{11\bar{t}11}\ar[rrdd]|-{1111\bar{s}}\xtwocell[rrrr]{!<5mm,0mm>}<>{^11\bar{t}\bar{\alpha}\quad}&&&&TT\ot T\ot TT\ot T\ar[rruu]|-{1e111}\ar[rrdd]|-{111e1}\ar@{}[uuuu]|*=0[@]{\cong}&&&&T\ot T\\
&&&&&&&&&&&&\\
&&S\ot SR\ot R\ar[rruu]|-{n1111}\ar[rrdd]|-{11\bar{s}}\ar@{}[uuuu]|*[@]{\cong}&&&&T\ot TS\ot SS\ot S\ar[rrdd]|-{111e1}\ar[rruu]|-{11\bar{t}\bar{t}}\xtwocell[rrrr]{!<5mm,0mm>}<>{^11\bar{t}^2\quad}&&&&T\ot TT\ot T\ar[rruu]_-{1e1}\ar@{}[uuuu]|*=0[@]{\cong}&&\\
&S\ot S\ar[rrrd]_-{n11}\ar@/_5mm/[rrrrrddd]_-{1}\ar@{}[ru]|*=0[@]{\cong}&&&&&&&&&&\ar@{}[lu]|*=0[@]{\cong}&\\
&&&&R\ot RR\ot R\ar[rruu]|-{n1111}\ar[rrdd]_-{1e1}\ar@{}[uuuu]|*=0[@]{\cong}&&&&T\ot TS\ot S\ar[rruu]|-{11\bar{t}}&&&&\\
&&&\ar@{}[ru]|*=0[@]{\cong}&&&&&&T\ot T\ar[ruuu]_-{n11}\ar@/_5mm/[rrruuuuu]_-{1}\ar@{}[lu]|*=0[@]{\cong}&&&\\
&&&&&&S\ot S\ar[rruu]|-{n11}\ar@{}[uuuu]|*[@]{\cong}\ar[rrru]_-{\bar{t}}&&&
}}}
\]
\[
=
\vcenter{\hbox{\xymatrix@!0@=8mm{
&&&&&&&&T\ot TR\ot R\ar[rrdd]^-{11\bar{v}}&&&&\\
&&&&&&&&&&&&\\
&&T\ot TR\ot R\ar[rrdd]|-{11n11}\ar@/^5mm/[rrrr]^-{11n11}\ar[rddd]_-{11\bar{s}}\ar@/^9mm/[rrrrrruu]^-{1}\xtwocell[rrrr]{!<5mm,0mm>}<>{^<1>11\bar{t}^011\qquad}&&&&T\ot TT\ot TR\ot R\ar[rruu]|-{1e111}\ar[rrdd]|-{1111\bar{v}}\ar@{}[uu]|>>>*=0[@]{\cong}&&&&T\ot TT\ot T\ar[rrdd]^-{1e1}&&\\
&&&&&&&&&&&&\\
R\ot R\ar[rruu]^-{n11}\ar[rddd]_-{\bar{s}}&&&&T\ot TS\ot SR\ot R\ar[rruu]|-{11\bar{t}11}\ar[rrdd]|-{1111\bar{s}}\xtwocell[rrrr]{!<5mm,0mm>}<>{^11\bar{t}\bar{\alpha}\quad}&&&&TT\ot T\ot TT\ot T\ar[rruu]|-{1e111}\ar[rrdd]|-{111e1}\ar@{}[uuuu]|*=0[@]{\cong}&&&&T\ot T\\
&&&T\ot TS\ot S\ar[rrrd]_-{11n11}\ar@/_5mm/[rrrrrddd]_-{1}\ar@{}[ru]|*=0[@]{\cong}&&&&&&&&&\\
&&&&&&T\ot TS\ot SS\ot S\ar[rrdd]|-{111e1}\ar[rruu]|-{11\bar{t}\bar{t}}\xtwocell[rrrr]{!<5mm,0mm>}<>{^11\bar{t}^2\quad}&&&&T\ot TT\ot T\ar[rruu]_-{1e1}\ar@{}[uuuu]|*=0[@]{\cong}&&\\
&S\ot S\ar@/_5mm/[rrrrrddd]_-{1}\ar[rruu]|-{n11}\ar@{}[ruuuuu]|*=0[@]{\cong}&&&&\ar@{}[ru]|*=0[@]{\cong}&&&&&&\ar@{}[lu]|*=0[@]{\cong}&\\
&&&&&&&&T\ot TS\ot S\ar[rruu]|-{11\bar{t}}&&&&\\
&&&&&&&&&T\ot T\ar[ruuu]_-{n11}\ar@/_5mm/[rrruuuuu]_-{1}\ar@{}[lu]|*=0[@]{\cong}&&&\\
&&&&&&S\ot S\ar[rruu]|-{n11}\ar[rrru]_-{\bar{t}}&&&
}}}
\]
\[
=
\vcenter{\hbox{\xymatrix@!0@=8mm{
&&&&&&&&T\ot TR\ot R\ar[rrdd]^-{11\bar{v}}&&&&\\
&&&&&&&&&&&&\\
&&T\ot TR\ot R\ar@/^5mm/[rrrr]^-{11n11}\ar[rddd]_-{11\bar{s}}\ar[rrrd]^-{11\bar{v}}\ar@/^9mm/[rrrrrruu]^-{1}\xtwocell[rrrd]{!<4mm,-3mm>}<>{^<3>11\bar{\alpha}\ \ }&&&&T\ot TT\ot TR\ot R\ar[rruu]|-{1e111}\ar[rrdd]|-{1111\bar{v}}\ar@{}[uu]|>>>*=0[@]{\cong}&&&&T\ot TT\ot T\ar[rrdd]^-{1e1}&&\\
&&&&&T\ot TT\ot T\ar[rrrd]_-{11n11}\ar@{}[ru]|*=0[@]{\cong}&&&&&&&\\
R\ot R\ar[rruu]^-{n11}\ar[rddd]_-{\bar{s}}&&&&&&&&TT\ot T\ot TT\ot T\ar[rruu]|-{1e111}\ar[rrdd]|-{111e1}\ar@{}[uuuu]|*=0[@]{\cong}&&&&T\ot T\\
&&&T\ot TS\ot S\ar[rrrd]|-{11n11}\ar[rruu]|-{11\bar{t}}\ar@/_5mm/[rrrrrddd]_-{1}\xtwocell[rrrrru]{!<5mm,0mm>}<>{^11\bar{t}^011\qquad}&&&&&&&&&\\
&&&&&&T\ot TS\ot SS\ot S\ar[rrdd]|-{111e1}\ar[rruu]|-{11\bar{t}\bar{t}}\xtwocell[rrrr]{!<5mm,0mm>}<>{^11\bar{t}^2\quad}&&&&T\ot TT\ot T\ar[rruu]_-{1e1}\ar@{}[uuuu]|*=0[@]{\cong}&&\\
&S\ot S\ar@/_5mm/[rrrrrddd]_-{1}\ar[rruu]|-{n11}\ar@{}[ruuuuu]|*=0[@]{\cong}&&&&\ar@{}[ru]|*=0[@]{\cong}&&&&&&\ar@{}[lu]|*=0[@]{\cong}&\\
&&&&&&&&T\ot TS\ot S\ar[rruu]|-{11\bar{t}}&&&&\\
&&&&&&&&&T\ot T\ar[ruuu]_-{n11}\ar@/_5mm/[rrruuuuu]_-{1}\ar@{}[lu]|*=0[@]{\cong}&&&\\
&&&&&&S\ot S\ar[rruu]|-{n11}\ar[rrru]_-{\bar{t}}&&&
}}}
\]
\[
=
\vcenter{\hbox{\xymatrix@!0@=8mm{
&&&&&&&&T\ot TR\ot R\ar[rrdd]^-{11\bar{v}}&&&&\\
&&&&&&&&&&&&\\
&&T\ot TR\ot R\ar[rddd]_-{11\bar{s}}\ar[rrrd]^-{11\bar{v}}\ar@/^9mm/[rrrrrruu]^-{1}\xtwocell[rrrd]{!<4mm,-3mm>}<>{^<3>11\bar{\alpha}\ \ }&&&&&&&&T\ot TT\ot T\ar[rrdd]^-{1e1}&&\\
&&&&&T\ot TT\ot T\ar[rrrd]_-{11n11}\ar@/_5mm/[rrrrrddd]_-{1}\ar@/^5mm/[rrrrru]^-{1}&&&&&&&\\
R\ot R\ar[rruu]^-{n11}\ar[rddd]_-{\bar{s}}&&&&&&&&TT\ot T\ot TT\ot T\ar[rruu]|-{1e111}\ar[rrdd]|-{111e1}\ar@{}[luu]|*=0[@]{\cong}&&&&T\ot T\\
&&&T\ot TS\ot S\ar[rruu]|-{11\bar{t}}\ar@/_5mm/[rrrrrddd]_-{1}&&&&\ar@{}[ru]|*=0[@]{\cong}&&&&&\\
&&&&&&&&&&T\ot TT\ot T\ar[rruu]_-{1e1}\ar@{}[uuuu]|*=0[@]{\cong}&&\\
&S\ot S\ar@/_5mm/[rrrrrddd]_-{1}\ar[rruu]|-{n11}\ar@{}[ruuuuu]|*=0[@]{\cong}&&&&&&&&&&\ar@{}[lu]|*=0[@]{\cong}&\\
&&&&&&&&T\ot TS\ot S\ar[rruu]|-{11\bar{t}}&&&&\\
&&&&&&&&&T\ot T\ar[ruuu]_-{n11}\ar@/_5mm/[rrruuuuu]_-{1}\ar@{}[lu]|*=0[@]{\cong}&&&\\
&&&&&&S\ot S\ar[rruu]|-{n11}\ar[rrru]_-{\bar{t}}&&&
}}}
\]
\[
=
\vcenter{\hbox{\xymatrix@!0@=8mm{
&&&&&&&T\ot T\\
&&T\ot TR\ot R\ar[rddd]_-{11\bar{s}}\ar[rrrd]^-{11\bar{v}}\xtwocell[rrrd]{!<4mm,-3mm>}<>{^<3>11\bar{\alpha}\ \ }&&&&&\\
&&&&&T\ot TT\ot T\ar[rruu]^-{1e1}&&\\
R\ot R\ar[rruu]^-{n11}\ar[rddd]_-{\bar{s}}&&&&&&\ar@{}[lu]|*=0[@]{\cong}&\\
&&&T\ot TS\ot S\ar[rruu]|-{11\bar{t}}&&&&\\
&&&&T\ot T\ar[ruuu]|-{n11}\ar@/_5mm/[rrruuuuu]_-{1}\ar@{}[lu]|*=0[@]{\cong}&&&\\
&S\ot S\ar[rruu]|-{n11}\ar[rrru]_-{\bar{t}}\ar@{}[ruuuuu]|*=0[@]{\cong}&&&&&&
}}}
\!\!\!\!\!\!\!\!\!\!=\!\!\!\!
\vcenter{\hbox{\xymatrix@!0@C=13mm@R=9mm{
&T\ot TR\ot R\ar@/^4mm/[rr]^-{11\bar{v}}&&T\ot TT\ot T\ar[rdd]^-{1e1}&\\
&&&&\\
R\ot R\ar[rd]_{\bar{s}}\ar@/^4mm/[rr]^{\bar{v}}\ar[ruu]^-{n11}\xtwocell[rr]{}<>{^<1>{\bar{\alpha}}}&&T\ot T\ar[ruu]|-{n11}\ar[rr]_-{1}\ar@{}[luu]!<5mm,0mm>|>>>>>*=0[@]{\cong}&\ar@{}[u]|>*=0[@]{\cong}&T\ot T\\
&S\ot S\ar[ur]_{\bar{t}}&&&
}}}
\]
This proves for $\bar{s}$, $\bar{t}$, and $\bar{v}$ that the triple of \emph{fixed} isomorphisms of the following kind,
\[
\vcenter{\hbox{\xymatrix@!0@C=12mm@R=12mm{
&S\ot SR\ot R\ar[rr]^-{11\bar{s}}&&S\ot SS\ot S\ar[rd]^-{1e1}&\\
R\ot R\ar[rr]_{\bar{s}}\ar[ru]^-{n11}&&S\ot S\ar[ru]|-{n11}\ar[rr]_-{1}\ar@{}[lu]|*=0[@]{\cong}&\ar@{}[u]|*=0[@]{\cong}&S\ot S
}}}
\]
satisfies \eqref{ax:OM6}, creating a isomorphism of opmonoidal 2-simplices between $P'P\bar{\alpha}$ and $\bar{\alpha}$ which does not depend on the cell $\bar{\alpha}$ but only on its source and target. Now, let $\widehat{\alpha}$ in the domain of $P'$, then $PP'(\widehat{\alpha})=\widehat{\alpha}$ follows by axiom~\eqref{ax:A1}.
\end{proof}

To continue with the proof of the theorem, fix two 2-simplices $(\bar{s},\bar{t},\bar{v},\bar{\alpha})$ and $(\bar{s}',\bar{t}',\bar{v}',\bar{\alpha}')$ in $\Opmon(R,S,T)$, and let $(\widehat{s},\widehat{t},\widehat{v},\widehat{\alpha})$ and $(\widehat{s}',\widehat{t}',\widehat{v}',\widehat{\alpha}')$ be their corresponding objects in $\mathcal{A}(R;S;T)$ under the equivalence \eqref{eq:P1} above and the isomorphism $P$ of the claim.

\begin{claim}
There is an isomorphism between the following hom sets.
\begin{multline*}
\Opmon(R\ot R,S\ot S,T\ot T)((\bar{s},\bar{t},\bar{v},\bar{\alpha}),(\bar{s}',\bar{t}',\bar{v}',\bar{\alpha}'))\\
\cong\mathcal{A}(R;S;T)((\widehat{s},\widehat{t},\widehat{v},\widehat{\alpha}),(\widehat{s}',\widehat{t}',\widehat{v}',\widehat{\alpha}'))
\end{multline*}
\end{claim}
\begin{proof}

The only thing to verify is that the isomorphism on hom sets induced by \eqref{eq:P1} restricts to the isomorphism in this claim. Take an arrow of opmonoidal 2-simplices $\xymatrix@1@C=5mm{(\bar{\sigma},\bar{\tau},\bar{\nu}):(\bar{s},\bar{t},\bar{v},\bar{\alpha})\ar[r]&(\bar{s}',\bar{t}',\bar{v}',\bar{\alpha}')}$, then its transpose $(P\bar{\sigma},P\bar{\tau},P\bar{\nu})$ satisfies axiom~\eqref{ax:A4} as a consequence of axiom~\eqref{ax:OM4} for $\bar{\tau}$ and axiom~\eqref{ax:OM6} for $(\bar{\sigma},\bar{\tau},\bar{\nu})$. Now, let $\xymatrix@1@C=5mm{(\widehat{\sigma},\widehat{\tau},\widehat{\nu}):(\widehat{s},\widehat{t},\widehat{v},\widehat{\alpha})\ar[r]&(\widehat{s}',\widehat{t}',\widehat{v}',\widehat{\alpha}')}$ be an arrow in $\mathcal{A}(R;S;T)$ then its transpose satisfies axiom~\eqref{ax:OM6} as a consequence of axiom~\eqref{ax:OLA5} for $\widehat{\tau}$ axiom~\eqref{ax:A4} for $(\widehat{\sigma},\widehat{\tau},\widehat{\nu})$.
\end{proof}

Thus with both claims we conclude the proof of the theorem.
\end{proof}

Even though we call this the ``transposition step'', it is no longer mere transposition via the universal property of bidualities between the cells $\bar{\alpha}$ and cells $\widehat{\alpha}$ as in the case of 1-simplices. Perhaps one shall call this process ``2-simplex transposition along $S\dashv S\ot\ $ and $T\dashv T\ot\ $'' since this proof does not depend on the enveloping monoidal structure of $R\ot R$. For a fully general version of Theorem~\ref{teo:Opmon2simplexTransposition}, one might replace this enveloping monoidale for an arbitrary right skew monoidale $M$.

Now, we perform the opmonadicity step.

\begin{teo}\label{teo:2SimplexOpmonadicOplaxAct}
Let $\mathcal{M}$ be an opmonadic-friendly monoidal bicategory. For every three bidualities $R\dashv R\ot$, $S\dashv S\ot$ and $T\dashv T\ot$, and every two opmonadic adjunctions,
\[
\vcenter{\hbox{\xymatrix{
R\ot\dtwocell_{i\ob}^{i\ot}{'\dashv}\\
I
}}}
\qquad
\vcenter{\hbox{\xymatrix{
S\ot\dtwocell_{j\ob}^{j\ot}{'\dashv}\\
I
}}}
\]
there is an equivalence of categories,
\[
\mathcal{A}(R;S;T)\simeq\OplaxAct(R;S;T)
\]
given by precomposition with the opmonadic left adjoint $1j\ob 1i\ob 1$.
\[
\vcenter{\hbox{\xymatrix@!0@=15mm{
**[l]TS\ot SR\ot R\ar[r]^-{\widehat{t}11}\ar[d]_-{11\widehat{s}}\xtwocell[rd]{}<>{^\widehat{\alpha}}&TR\ot R\ar[d]^-{\widehat{v}}\\
TS\ot S\ar[r]_-{\widehat{t}}&T
}}}
\vcenter{\hbox{\xymatrix{
\ar@{|->}[r]^-{Q}&
}}}
\vcenter{\hbox{\xymatrix@!0@=18mm{
TSR\ar[rd]|-{1j\ob 1i\ob 1}\ar[r]^-{1j\ob 11}\ar[d]_-{11i\ob 1}&TS\ot SR\ar[r]^-{\widehat{t}1}\ar[d]|-{111i\ob 1}&TR\ar[d]^-{1i\ob1}\\
TSR\ot R\ar[d]_-{1\widehat{s}}\ar@/_/[]!<0mm,-1mm>;[r]!<0mm,-1mm>_-{1j\ob 111}\ar@{}[ru]|>>>>>>*=0[@ru]{\cong}&TS\ot SR\ot R\ar[r]^-{\widehat{t}11}\ar[d]_-{11\widehat{s}}\ar@{}[ru]|*=0[@]{\cong}\xtwocell[rd]{}<>{^\widehat{\alpha}}&TR\ot R\ar[d]^-{\widehat{v}}\\
TS\ar[r]_-{1j\ob 1}\ar@{}[ru]|*=0[@]{\cong}&TS\ot S\ar[r]_-{\widehat{t}}&T
}}}
\]
\end{teo}
\begin{proof}

The equivalence between the edges of the squares in the statement is determined by the equivalence below which obtained by three instances of the equivalence in \cite[Theorem~6.10]{Abud2018}.
\begin{multline}\label{eq:Q1}
\OplaxAct(R\ot R;S)\times\OplaxAct(S\ot S;T)\times\OplaxAct(R\ot R;T)\\
\simeq\OplaxAct(R;S)\times\OplaxAct(S;T)\times\OplaxAct(R;T)
\end{multline}
Now, $Q$ is well defined on objects because for a square $\widehat{\alpha}$, composing the appropriate left adjoint with each of the sides of axioms~\eqref{ax:A1}, \eqref{ax:A2}, and \eqref{ax:A3} turns them into the axioms~\eqref{ax:2SIM1}, \eqref{ax:2SIM2} and \eqref{ax:2SIM3}; for example, precomposing the arrow
\[
\vcenter{\hbox{\xymatrix{
1j\ob 1j\ob 1i\ob 1:TS\ot SS\ot SR\ot R\ar[r]&TSSR
}}}
\]
with both sides of axiom~\eqref{ax:A1} gives each of the sides of axiom~\eqref{ax:2SIM1}. Likewise, $Q$ is well defined on the arrows of $\mathcal{A}(R;S;T)$, because precomposing both sides of axiom~\eqref{ax:A4} with the arrow $\xymatrix@1@C=5mm{1j\ob 1i\ob 1:TS\ot SR\ot R\ar[r]&TSR}$ gives each of the sides of axiom~\eqref{ax:2SIM4}. Hence, $Q$ is a well defined functor, and because of the equivalence~\eqref{eq:Q1} above it is automatically faithful. To prove that $Q$ is essentially surjective on objects and full, first remember the core technicalities of the equivalence~\eqref{eq:Q1} regarding opmonadicity. An oplax action in any of the categories on the right hand side of \eqref{eq:Q1}, comes equipped with the structure of a module for a monad on induced by an opmonadic adjunction \cite[Theorem~6.10]{Abud2018}. This module structure is expressed in terms of the oplax action constraints, for example, take an oplax right action $\xymatrix@1@C=5mm{s:SR\ar[r]&S}$ with respect to the skew monoidale induced by an adjunction $i\dashv i^*$, then the cell $\psi$ below
\begin{equation}\label{eq:redundant_action2}
\vcenter{\hbox{\xymatrix@!0@C=16mm{
SR\ar[r]_-{1i\ob 1}\ar@/^9mm/[rrr]^-{s}&SR\ot R\ar[r]_-{1i\ot 1}\xtwocell[r]{}<>{^<-3>\psi}&SR\ar[r]_-{s}&S
}}}
=
\vcenter{\hbox{\xymatrix@!0@R=10mm@C=8mm{
&&&&SR\ar[rrdd]^-{s}&&\\
&&SRR\ar[rrd]|-{1i^*1}\ar[rru]^-{s1}&{\xtwocell[rrd]{}<>{^<-2>s^2\ }}&&\\
SR\ar[rr]_-{1i\ob 1}\ar@/^12mm/[rrrruu]^-{1}\ar[rru]^-{1i1}\xtwocell[rrrruu]{}<>{^<-4>s^01\quad}&&SR\ot R\ar[rr]_-{1i\ot 1}\ar@{}[u]|*[@]{\cong}&&SR\ar[rr]_-{s}&&S
}}}
\end{equation}
is a module structure for $s$, with respect to the monad induced by the adjunction $1i\ob 1\dashv 1i\ot 1$. And because this adjunction is opmonadic by hypothesis, there exists an oplax right action $\xymatrix@1@C=5mm{\widehat{s}:SR\ot R\ar[r]&S}$ with respect to the enveloping monoidale $R\ot R$, such that the precomposition with $\xymatrix@1@C=5mm{1i\ob 1:SR\ar[r]&SR\ot R}$ is isomorphic to $s$.

To prove that $Q$ is essentially surjective on objects let $\alpha$ be a 2-simplex in the category $\OplaxAct(R;S;T)$.
\[
\vcenter{\hbox{\xymatrix@!0@=15mm{
TSR\ar[r]^-{t1}\ar[d]_-{1s}\xtwocell[rd]{}<>{^\alpha}&TR\ar[d]^-{v}\\
TS\ar[r]_-{t}&T
}}}
\]
Now, the arrows below are opmonadic
\[
\xymatrix@1@C=5mm{1i1:SR\ar[r]&SR\ot R}
\qquad
\xymatrix@1@C=5mm{1j1:TS\ar[r]&TS\ot S}
\qquad
\xymatrix@1@C=5mm{1i1:TR\ar[r]&TR\ot R}
\]
and $s$, $t$, and $v$ come equipped with the module structures given by the formula~\eqref{eq:redundant_action2} (all of which, by a slight abuse of notation, will be called $\psi$). Hence, there are three induced oplax actions $\widehat{s}$, $\widehat{t}$, and $\widehat{v}$ with isomorphisms as below,
\begin{equation}
\label{iso:opmonadicity2}
\vcenter{\hbox{\xymatrix@!0@C=14mm{
SR\ar[rd]_-{1i\ob 1}\ar[rr]^-{s}&&S\\
&SR\ot R\ar[ru]_-{\widehat{s}}\ar@{}[u]|<<<*=0[@]{\cong}&
}}}
\quad\!
\vcenter{\hbox{\xymatrix@!0@C=14mm{
TS\ar[rd]_-{1j\ob 1}\ar[rr]^-{t}&&T\\
&TS\ot S\ar[ru]_-{\widehat{t}}\ar@{}[u]|<<<*=0[@]{\cong}&
}}}
\quad\!
\vcenter{\hbox{\xymatrix@!0@C=14mm{
TR\ar[rd]_-{1i\ob 1}\ar[rr]^-{v}&&T\\
&TR\ot R\ar[ru]_-{\widehat{v}}\ar@{}[u]|<<<*=0[@]{\cong}&
}}}
\end{equation}
All this structure turns the cell $\alpha$ into a morphism of modules for the monad induced by the opmonadic adjunction below.
\[
\vcenter{\hbox{\xymatrix{
TS\ot SR\ot R\dtwocell_{1j\ob 1i\ob 1\quad\quad}^{\quad\quad 1j\ot 1i\ot 1}{'\dashv}\\
TSR
}}}
\]
but instead of proving that directly, it is simpler and equivalent to describe module morphism structures for the monads induced by the two adjunctions below, which are also opmonadic by hypothesis.
\[
\vcenter{\hbox{\xymatrix{
TS\ot SR\dtwocell_{1j\ob 11\qquad}^{\qquad 1j\ot 11}{'\dashv}\\
TSR
}}}
\qquad
\vcenter{\hbox{\xymatrix{
TSR\ot R\dtwocell_{11i\ob 1\qquad}^{\qquad 11i\ot 1}{'\dashv}\\
TSR
}}}
\]
The structure of modules for the monads induced by these adjunctions on source and target of $\alpha$ is induced by the actions $\psi$ on $s$, $t$, and $v$. And the fact that $\alpha$ is a module morphism means precisely that the following two equations must be satisfied.
\begin{align}
\label{ax:IR}
\vcenter{\hbox{\xymatrix@!0@R=6mm@C=6mm{
&&&&&TR\ar[rrddddddd]^-{v}&&\\
&&TSR\ar[lddd]_-{11i\ob 1}\ar[rrru]^-{t1}\ar[rrddddddd]_-{1s}&&&&&\\
&&&&&&&\\
&&&&&&&\\
&TSR\ot R\ar[lddd]_-{11i\ot 1}&&{\xtwocell[rrrd]{}<>{^<-1>\alpha}}&&&&\\
&&&&&&&\\
&&&&&&&\\
TSR\ar[rrrrd]_-{1s}\xtwocell[rrrr]{}<>{^<-6>1\psi\ }&&&&&&&T\\
&&&&TS\ar[rrru]_-{t}&&&
}}}
&\!\!=\!\!
\vcenter{\hbox{\xymatrix@!0@R=6mm@C=6mm{
&&&&&TR\ar[rrddddddd]^-{v}\ar[lddd]_-{1i\ob 1}&&\\
&&TSR\ar[lddd]_-{11i\ob 1}\ar[rrru]^-{t1}&&&&&\\
&&&&&&&\\
&&&&TR\ot R\ar[lddd]_-{1i\ot 1}&&&\\
&TSR\ot R\ar[lddd]_-{11i\ot 1}\ar[rrru]^-{t11}\ar@{}[rrrruuuu]|*=0[@]{\cong}&&&&&&\\
&&&&&&&\\
&&&TR\ar[rrrrd]^-{v}\xtwocell[rrrr]{}<>{^<-6>\psi}&&&&\\
TSR\ar[rrrrd]_-{1s}\ar[rrru]^-{t1}\ar@{}[rrrruuuu]|*=0[@]{\cong}\xtwocell[rrrrrrr]{}<>{^\alpha}&&&&&&&T\\
&&&&TS\ar[rrru]_-{t}&&&\\
}}}
\\
\label{ax:JS}
\vcenter{\hbox{\xymatrix@!0@R=6mm@C=6mm{
&&&&TR\ar[rrrd]^-{v}&&&\\
TSR\ar[rrrru]^-{t1}\ar[rddd]_-{1j\ob 11}\xtwocell[rrrr]{}<>{^<7>\psi 1\ }&&&&&&&T\\
&&&&&&&\\
&&&&&&&\\
&TS\ot SR\ar[rddd]_-{1j\ot 11}&&{\xtwocell[rrru]{}<>{^<1>\alpha}}&&&&\\
&&&&&&&\\
&&&&&&&\\
&&TSR\ar[rrrd]_-{1s}\ar[rruuuuuuu]^-{t1}&&&&&\\
&&&&&TS\ar[rruuuuuuu]_-{t}&&
}}}
&\!\!=\!\!
\vcenter{\hbox{\xymatrix@!0@R=6mm@C=6mm{
&&&&TR\ar[rrrd]^-{v}&&&\\
TSR\ar[rrrru]^-{t1}\ar[rrrd]_-{1s}\ar[rddd]_-{1j\ob 11}\xtwocell[rrrrrrr]{}<>{^\alpha}&&&&&&&T\\
&&&TS\ar[rrrru]_-{t}\ar[rddd]_-{1j\ob 1}\xtwocell[rrrr]{}<>{^<7>\psi}&&&&\\
&&&&&&&\\
&TS\ot SR\ar[rrrd]_-{11s}\ar[rddd]_-{1j\ot 11}\ar@{}[rruu]|*[@]{\cong}&&&&&&\\
&&&&TS\ot S\ar[rddd]_-{1j\ot 1}&&&\\
&&&&&&&\\
&&TSR\ar[rrrd]_-{1s}\ar@{}[rruu]|*[@]{\cong}&&&&&\\
&&&&&TS\ar[rruuuuuuu]_-{t}&&
}}}
\end{align}
Condition~\eqref{ax:IR} is a consequence of the following calculation,
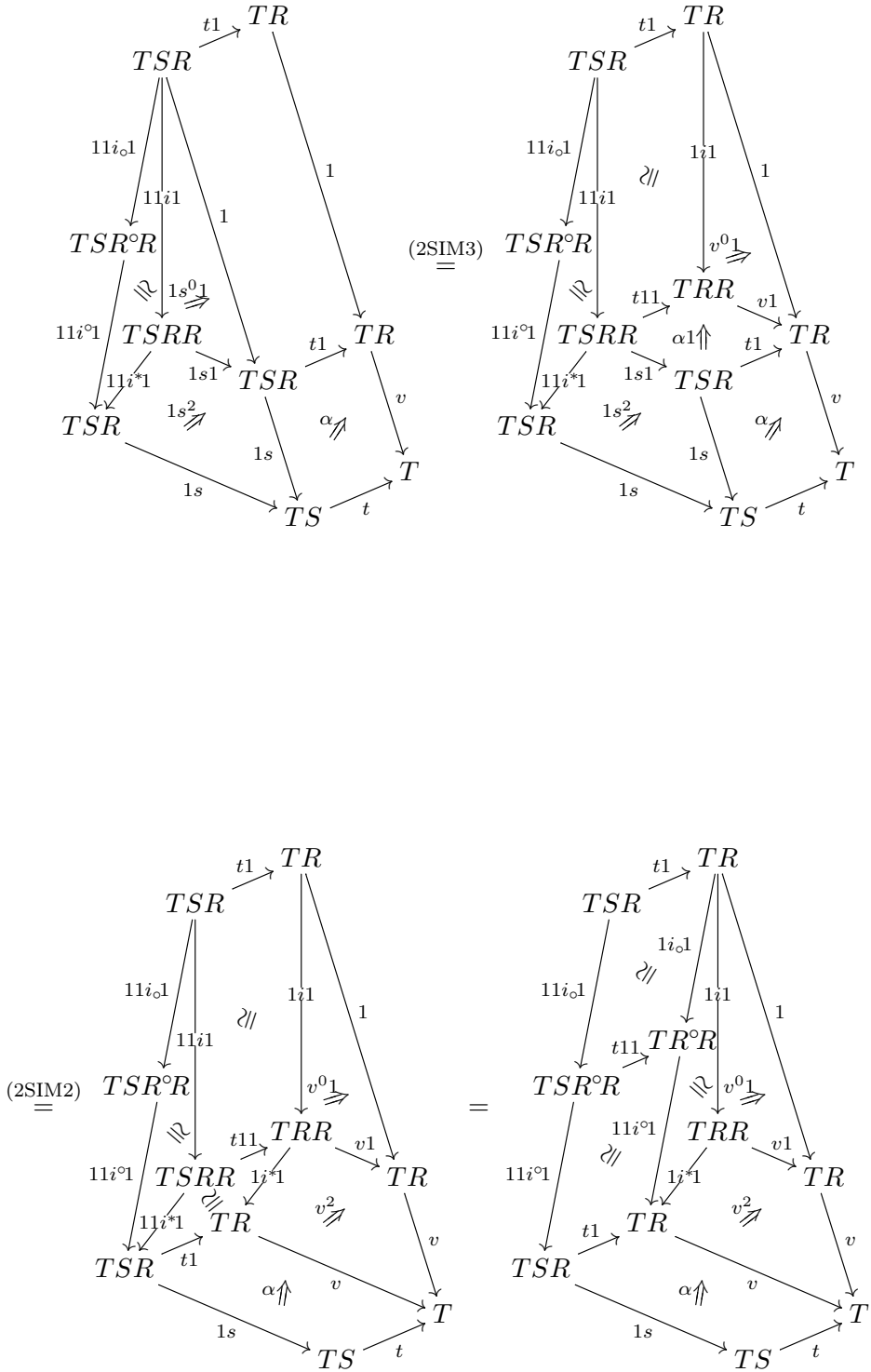
\[
\vcenter{\hbox{\xymatrix@!0@R=6.5mm@C=5mm{
&&&&&TR\ar[rrrddddddd]!<-.8mm>^-{1}&&&&\\
&&TSR\ar[ldddd]_-{11i\ob 1}\ar[dddddd]|-{11i1}\ar[rrru]^-{t1}\ar[rrrddddddd]!<-.8mm>^-{1}\xtwocell[rrrrdddddddddd]{}<>{^<3>1s^01\ }&&&&&&&\\
&&&&&&&&&\\
&&&&&&&&&\\
&&&&&&&&&\\
\ar@{}[rrrdd]|*=0[@]{\cong}&TSR\ot R\quad\ar[ldddd]_-{11i\ot 1}&&&&&&&&\\
&&&&&&&&&\\
&&TSRR\ar[lldd]|-{11i^*\!1}\ar[rrrd]_-{1s1}\xtwocell[rrrddd]{}<>{^<2>1s^2\ }&&&&&&TR\ar[]!<-.8mm>;[rddd]^-{v}&\\
&&&&&TSR\ar[]!<-.8mm>;[rddd]_-{1s}\ar[rrru]^-{t1}\xtwocell[rrrrdd]{}<>{^\alpha}&&&&\\
TSR\ar[rrrrrrdd]_-{1s}&&&&&&&&&\\
&&&&&&&&&T\\
&&&&&&TS\ar[rrru]_-{t}&&&
}}}
\!\!\!\!\!\!\stackrel{\eqref{ax:2SIM3}}{=}\!\!
\vcenter{\hbox{\xymatrix@!0@R=6.5mm@C=5mm{
&&&&&TR\ar[rrrddddddd]!<-.8mm>^-{1}\ar[dddddd]|-{1i1}\xtwocell[rrrrdddddddddd]{}<>{^<3>v^01\ \ }&&&&\\
&&TSR\ar[ldddd]_-{11i\ob 1}\ar[dddddd]|-{11i1}\ar[rrru]^-{t1}&&&&&&&\\
&&&&&&&&&\\
&&&&&&&&&\\
&&&&&&&&&\\
\ar@{}[rrrdd]|*=0[@]{\cong}&TSR\ot R\quad\ar[ldddd]_-{11i\ot 1}&&&&&&&&\\
&&&&&TRR\ar[rrrd]^-{v1}&&&&\\
&&TSRR\ar[lldd]|-{11i^*\!1}\ar[rrrd]_-{1s1}\ar[rrru]^-{t11}\ar@{}[rrruuuuuuu]|*=0[@]{\cong}\xtwocell[rrrrrr]{}<>{^\alpha 1\ }\xtwocell[rrrddd]{}<>{^<2>1s^2\ }&&&&&&TR\ar[]!<-.8mm>;[rddd]^-{v}&\\
&&&&&TSR\ar[]!<-.8mm>;[rddd]_-{1s}\ar[rrru]^-{t1}\xtwocell[rrrrdd]{}<>{^\alpha}&&&&\\
TSR\ar[rrrrrrdd]_-{1s}&&&&&&&&&\\
&&&&&&&&&T\\
&&&&&&TS\ar[rrru]_-{t}&&&
}}}
\]
\[
\!\!\!\!\!\!\stackrel{\eqref{ax:2SIM2}}{=}\!\!
\vcenter{\hbox{\xymatrix@!0@R=6.5mm@C=5mm{
&&&&&TR\ar[rrrddddddd]!<-.8mm>^-{1}\ar[dddddd]|-{1i1}\xtwocell[rrrrdddddddddd]{}<>{^<3>v^01\ \ }&&&&\\
&&TSR\ar[ldddd]_-{11i\ob 1}\ar[dddddd]|-{11i1}\ar[rrru]^-{t1}&&&&&&&\\
&&&&&&&&&\\
&&&&&&&&&\\
&&&&&&&&&\\
\ar@{}[rrrdd]|*=0[@]{\cong}&TSR\ot R\quad\ar[ldddd]_-{11i\ot 1}&&&&&&&&\\
&&&&&TRR\ar[rrrd]^-{v1}\ar[lldd]|-{1i^*\!1}\xtwocell[rrrddd]{}<>{^<2>v^2}&&&&\\
&&TSRR\ar[lldd]|-{11i^*\!1}\ar[rrru]^-{t11}\ar@{}[rrruuuuuuu]|*=0[@]{\cong}&&&&&&TR\ar[]!<-.8mm>;[rddd]^-{v}&\\
&&&TR\ar[rrrrrrdd]_-{v}&&&&&&\\
TSR\ar[rrrrrrdd]_-{1s}\ar[rrru]_-{t1}\ar@{}[rrrrruuu]|*=0[@]{\cong}\xtwocell[rrrrrrrrrd]{}<>{^\alpha}&&&&&&&&&\\
&&&&&&&&&T\\
&&&&&&TS\ar[rrru]_-{t}&&&
}}}
=
\vcenter{\hbox{\xymatrix@!0@R=6.5mm@C=5mm{
&&&&&TR\ar[rrrddddddd]!<-.8mm>^-{1}\ar[dddddd]|-{1i1}\ar[ldddd]_-{1i\ob 1}\xtwocell[rrrrdddddddddd]{}<>{^<3>v^01\ \ }&&&&\\
&&TSR\ar[ldddd]_-{11i\ob 1}\ar[rrru]^-{t1}&&&&&&&\\
&&&&&&&&&\\
&&&&&&&&&\\
&&&\ar@{}[rrrdd]|*=0[@]{\cong}&TR\ot R\ar[ldddd]_-{11i\ot 1}&&&&&\\
&TSR\ot R\ar[rrru]^-{t11}\ar[ldddd]_-{11i\ot 1}\ar@{}[rrrruuuuu]|*=0[@]{\cong}&&&&&&&&\\
&&&&&TRR\ar[rrrd]^-{v1}\ar[lldd]|-{1i^*\!1}\xtwocell[rrrddd]{}<>{^<2>v^2}&&&&\\
&&&&&&&&TR\ar[]!<-.8mm>;[rddd]^-{v}&\\
&&&TR\ar[rrrrrrdd]_-{v}&&&&&&\\
TSR\ar[rrrrrrdd]_-{1s}\ar[rrru]^-{t1}\ar@{}[rrrruuuuu]|*=0[@]{\cong}\xtwocell[rrrrrrrrrd]{}<>{^\alpha}&&&&&&&&&\\
&&&&&&&&&T\\
&&&&&&TS\ar[rrru]_-{t}&&&
}}}
\]
and equation~\eqref{ax:JS} may be proved in the following way.
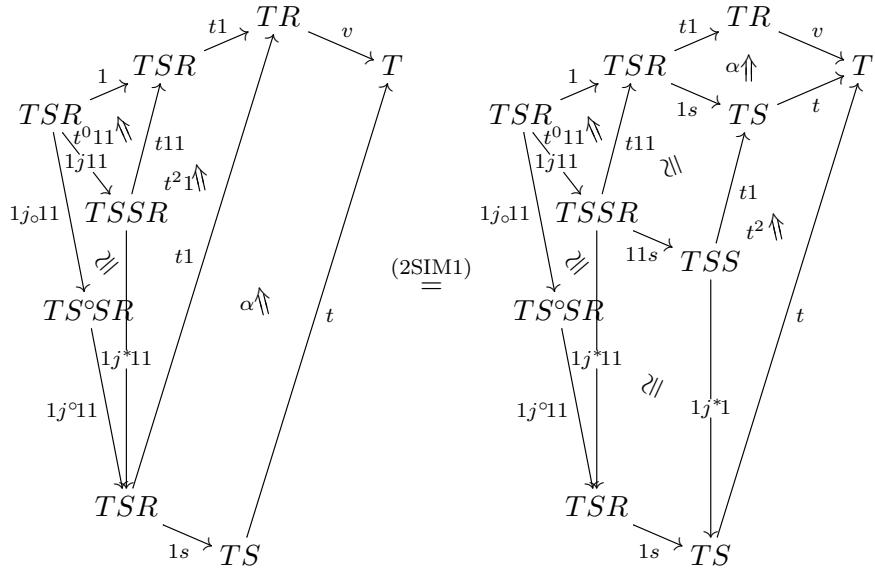
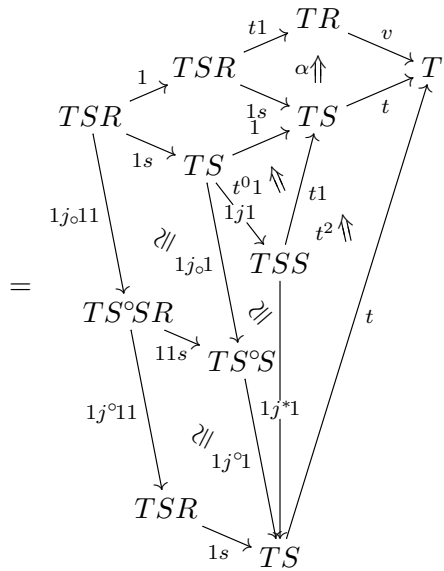
\[
\vcenter{\hbox{\xymatrix@!0@R=6.5mm@C=5mm{
&&&&&&TR\ar[rrrd]^-{v}&&&\\
&&&TSR\ar[rrru]^-{t1}&&&&&&T\\
TSR\ar[rrru]^-{1}\ar[rdddd]_-{1j\ob 11}\ar[rrdd]|-{1j11}\xtwocell[rrru]{}<>{^<3>t^011\quad}&&&&&&&&&\\
&&&&&&&&&\\
&{\xtwocell[rrrrrru]{}<>{^<-1>t^21\ }}&TSSR\ar[dddddd]|-{1j^*\!11}\ar[ruuu]_-{t11}&&&&&&&\\
&&&&&&&&&\\
\ar@{}[rrruu]|*=0[@]{\cong}&TS\ot SR\ar[rdddd]_-{1j\ot 11}&&{\xtwocell[rrrrru]{}<>{^<1>\alpha}}&&&&&&\\
&&&&&&&&&\\
&&&&&&&&&\\
&&&&&&&&&\\
&&TSR\ar[rrrd]_-{1s}\ar[rrrruuuuuuuuuu]^-{t1}&&&&&&&\\
&&&&&TS\ar[rrrruuuuuuuuuu]_-{t}&&&&
}}}
\!\!\!\!\!\!\stackrel{\eqref{ax:2SIM1}}{=}\!\!
\vcenter{\hbox{\xymatrix@!0@R=6.5mm@C=5mm{
&&&&&&TR\ar[rrrd]^-{v}&&&\\
&&&TSR\ar[rrru]^-{t1}\ar[rrrd]_-{1s}\xtwocell[rrrrrr]{}<>{^\alpha}&&&&&&T\\
TSR\ar[rrru]^-{1}\ar[rdddd]_-{1j\ob 11}\ar[rrdd]|-{1j11}\xtwocell[rrru]{}<>{^<3>t^011\quad}&&&&&&TS\ar[rrru]_-{t}&&&\\
&&&&&&&&&\\
&&TSSR\ar[dddddd]|-{1j^*\!11}\ar[ruuu]_-{t11}\ar[rrrd]_-{11s}&&&&&&&\\
&&&&{\xtwocell[rrrrru]{!<3mm,0mm>}<>{^<-1>t^2\ }}&TSS\ar[ruuu]_-{t1}\ar[dddddd]|-{1j^*\!1}\ar@{}[lluuuu]|*=0[@]{\cong}&&&&\\
\ar@{}[rrruu]|*=0[@]{\cong}&TS\ot SR\ar[rdddd]_-{1j\ot 11}&&&&&&&&\\
&&&&&&&&&\\
&&&&&&&&&\\
&&&&&&&&&\\
&&TSR\ar[rrrd]_-{1s}\ar@{}[rrruuuuu]|*=0[@]{\cong}&&&&&&&\\
&&&&&TS\ar[rrrruuuuuuuuuu]_-{t}&&&&
}}}
\]
\[
=
\vcenter{\hbox{\xymatrix@!0@R=6.5mm@C=5mm{
&&&&&&TR\ar[rrrd]^-{v}&&&\\
&&&TSR\ar[rrru]^-{t1}\ar[rrrd]_-{1s}\xtwocell[rrrrrr]{}<>{^\alpha}&&&&&&T\\
TSR\ar[rrru]^-{1}\ar[rdddd]_-{1j\ob 11}\ar[rrrd]_-{1s}&&&&&&TS\ar[rrru]_-{t}&&&\\
&&&TS\ar[rrdd]|-{1j1}\ar[rrru]^-{1}\ar[rdddd]_-{1j\ob 1}\xtwocell[rrru]{}<>{^<3>t^01\quad}&&&&&&\\
&&&&&&&&&\\
&&&&{\xtwocell[rrrrru]{!<3mm,0mm>}<>{^<-1>t^2\ }}&TSS\ar[ruuu]_-{t1}\ar[dddddd]|-{1j^*\!1}&&&&\\
&TS\ot SR\ar[rdddd]_-{1j\ot 11}\ar[rrrd]_-{11s}\ar@{}[rruuu]|*=0[@]{\cong}&&&&&&&&\\
&&&\ar@{}[rrruu]|*=0[@]{\cong}&TS\ot S\ar[rdddd]_-{1j\ot 1}&&&&&\\
&&&&&&&&&\\
&&&&&&&&&\\
&&TSR\ar[rrrd]_-{1s}\ar@{}[rruuu]|*=0[@]{\cong}&&&&&&&\\
&&&&&TS\ar[rrrruuuuuuuuuu]_-{t}&&&&
}}}
\]
Therefore, by the opmonadicity of $1j\ob 1i\ob 1$, there exists a cell $\widehat{\alpha}$ such that the following equation is satisfied.
\begin{equation}
\label{eq:isos_are_morphisms}
\vcenter{\hbox{\xymatrix@!0@=8.5mm{
&TSR\ar[rrrr]^-{t1}\ar[dddd]^-{1s}\ar[ldd]_-{11i\ob 1}\xtwocell[rrrrdddd]{}<>{^\alpha}\ar@{{}{ }{}}@/^8mm/[rrrr]^-{\phantom{t1}}&&&&TR\ar[dddd]^-{v}\\&&&&&&\\
TSR\ot R\qquad\ar[rdd]_-{1\widehat{s}}\ar@{}[r]|>>>*=0[@]{\cong}&&&&&\\&&&&&&\\
&TS\ar[rrrr]^-{t}\ar[rrd]_-{1j\ob 1}\ar@{{}{ }{}}@/_8mm/[rrrr]_-{\phantom{t1}}&&&&T\\
&&&TS\ot S\ar[rru]_-{\widehat{t}}\ar@{}[u]|>>>*=0[@]{\cong}&&\\
}}}
\!\!\!\!\!\!\!\!=\!\!
\vcenter{\hbox{\xymatrix@!0@=17mm{
TSR\ar[rd]|-{1j\ob 1i\ob 1}\ar[r]^-{1j\ob 11}\ar[d]_-{11i\ob 1}\ar@/^8mm/[rr]^-{t1}&TS\ot SR\ar[r]^-{\widehat{t}1}\ar[d]|-{\quad111i\ob 1}\ar@{}[]!<0mm,10mm>|*=0[@]{\cong}&TR\ar[d]^-{1i\ob1}\ar@/^10mm/[dd]^-{v}\\
TSR\ot R\ar[d]_-{1\widehat{s}}\ar@/_/[]!<0mm,-1mm>;[r]!<0mm,-1mm>_-{1j\ob 111}\ar@{}[ru]|>>>>>>*=0[@ru]{\cong}&TS\ot SR\ot R\ar[r]^-{\widehat{t}11}\ar[d]_-{11\widehat{s}}\ar@{}[ru]|*=0[@]{\cong}\xtwocell[rd]{}<>{^\widehat{\alpha}}&TR\ot R\ar[d]^-{\widehat{v}}\ar@{}[]!<14mm,0mm>|*=0[@]{\cong}\\
TS\ar[r]_-{1j\ob 1}\ar@{}[ru]|*=0[@]{\cong}\ar@{{}{ }{}}@/_8mm/[rr]_-{\phantom{t1}}&TS\ot S\ar[r]_-{\widehat{t}}&T
}}}
\end{equation}
The cell $\widehat{\alpha}$ makes the quadruple $(\widehat{s},\widehat{t},\widehat{v},\widehat{\alpha})$ into an object of $\mathcal{A}(R;S;T)$, because precomposing both sides of axioms~\eqref{ax:A1}, \eqref{ax:A2}, and \eqref{ax:A3} for $\widehat{\alpha}$ with the appropriate opmonadic left adjoint gives each of the sides of axioms~\eqref{ax:2SIM1}, \eqref{ax:2SIM2} and \eqref{ax:2SIM3} for $\alpha$ which are pairwise equal. The three isomorphisms \eqref{iso:opmonadicity2} become an morphism in $\mathcal{A}(R;S;T)$ as condition~\eqref{eq:isos_are_morphisms} is the appropriate instance of axiom~\eqref{ax:A4}. Hence $Q(\widehat{s},\widehat{t},\widehat{v},\widehat{\alpha})\cong(s,t,v,\alpha)$ and so $Q$ is essentially surjective on objects. To verify that $Q$ is full take a morphism $\xymatrix@1@C=5mm{(\sigma,\tau,\nu):(s,t,v,\alpha)\ar[r]&(s',t',v',\alpha')}$ in $\OplaxAct(R;S;T)$, by the equivalence~\eqref{eq:Q1} there exist a cell $(\widehat{\sigma},\widehat{\tau},\widehat{\nu})$ which is induced by opmonadicity. Precomposing both sides of axiom~\eqref{ax:A4} for $(\widehat{\sigma},\widehat{\tau},\widehat{\nu})$ with the opmonadic left adjoint $1j\ob 1i\ob 1$ gives each of the sides of axiom~\eqref{ax:2SIM4} for $(\sigma,\tau,\nu)$ which are equal, therefore $Q$ is full and so an equivalence.
\end{proof}

Now we bring the equivalences in Theorems~\ref{teo:Opmon2simplexTransposition} and \ref{teo:2SimplexOpmonadicOplaxAct} together for any opmonadic-friendly monoidal bicategory, see Definition~\ref{def:OpmonadicFriendly}.

\begin{cor}
Let $\mathcal{M}$ be an opmonadic-friendly monoidal bicategory. For every three bidualities $R\dashv R\ot$, $S\dashv S\ot$, and $T\dashv T\ot$, and opmonadic adjunctions
\[
\vcenter{\hbox{\xymatrix{
R\ot\dtwocell_{i\ob}^{i\ot}{'\dashv}\\
I
}}}
\qquad
\vcenter{\hbox{\xymatrix{
S\ot\dtwocell_{j\ob}^{j\ot}{'\dashv}\\
I
}}}
\]
there is an equivalence of categories as shown.
\[
\Opmon(R\ot R,S\ot S,T\ot T)\simeq\OplaxAct(R;S;T)
\]
\end{cor}
\begin{proof}
\[
\Opmon(R\ot R,S\ot S,T\ot T)\simeq\mathcal{A}(R;S;T)\simeq\OplaxAct(R;S;T)
\]
\end{proof}

\begin{cor}\label{cor:2simpEquiv}
Let $\mathcal{M}$ be an opmonadic-friendly monoidal bicategory such that every object $R$ has a chosen right bidual, and a chosen adjunction $i\dashv i^*$ whose opposite adjunction is opmonadic. There is an equivalence of categories of 2-simplices,
\[
\Opmon^{\mathrm{e}}_2\simeq\OplaxAct^{\mathrm{e}}_2
\]
which commutes with faces and degeneracies.
\[
\vcenter{\hbox{\xymatrix@!0@R=25mm@C=35mm{
\Opmon^{\mathrm{e}}_1\ar@<-3mm>[d]|-{\cs_0}\ar@<3mm>[d]|-{\cs_1}\ar[r]_-{\simeq}&\OplaxAct^{\mathrm{e}}_1\ar@<-3mm>[d]|-{\cs_0}\ar@<3mm>[d]|-{\cs_1}\\
\Opmon^{\mathrm{e}}_2\ar@<6mm>[u]|-{\partial_0}\ar[u]|-{\partial_1}\ar@<-6mm>[u]|-{\partial_2}\ar[r]^-{\simeq}&\OplaxAct^{\mathrm{e}}_2\ar@<6mm>[u]|-{\partial_0}\ar[u]|-{\partial_1}\ar@<-6mm>[u]|-{\partial_2}
}}}
\]
\end{cor}
\begin{proof}

The face maps commute strictly with the equivalences, since the equivalence of 2-simplices was defined precisely as such. For the degeneracies, the square commutes up to isomorphism. These isomorphisms may be strictified in a similar way as is already discussed in Remark~\ref{rem:degeneracies} by redefining the degeneracy functors.
\end{proof}

We now continue with the case of 3-simplices by introducing the categories that will form part of this process.

\begin{defi}
Given four bidualities $R\dashv R\ot$, $S\dashv S\ot$, $T\dashv T\ot$, and $U\dashv U\ot$ denote by $\Opmon(R\ot R,S\ot S,T\ot T,U\ot U)$ the category of 3-simplices of the lax-2-nerve of $\Opmon(\mathcal{M})$ with fixed 0-faces $R\ot R$, $S\ot S$, $T\ot T$ and $U\ot U$. It is a full subcategory of
\begin{multline*}
\Opmon(R\ot R,S\ot S,T\ot T)\times\Opmon(R\ot R,S\ot S,U\ot U)\\
\times\Opmon(R\ot R,T\ot T,U\ot U)\times\Opmon(S\ot S,T\ot T,U\ot U)
\end{multline*}
on the objects $(\bar{\alpha},\bar{\beta},\bar{\gamma},\bar{\zeta})$ satisfying the equation
\begin{equation}
\tag{OM7}\label{ax:OM7}
\vcenter{\hbox{\xymatrix@!0@C=8mm@R=10mm{
R\ot R\ar[dr]_-{\bar{s}}\ar@/^4mm/[rrrd]^-{\bar{v}}\ar@/^5mm/[rrrr]^-{\bar{w}}\xtwocell[rrrd]{}<>{^\bar{\alpha}}&&{\xtwocell[rr]{}<>{^\bar{\gamma}}}&&U\ot U\\
&S\ot S\ar[rr]_-{\bar{t}}&&T\ot T\ar[ru]_-{\bar{u}}&
}}}
=
\vcenter{\hbox{\xymatrix@!0@C=8mm@R=10mm{
R\ot R\ar[dr]_-{\bar{s}}\ar@/^5mm/[rrrr]^-{\bar{w}}\xtwocell[rrr]{}<>{^\bar{\beta}}&&&&U\ot U\\
&S\ot S\ar[rr]_-{\bar{t}}\ar@/^5mm/[rrru]^-{\bar{x}}\xtwocell[rrru]{}<>{^\bar{\zeta}}&&T\ot T\ar[ru]_-{\bar{u}}&
}}}
\end{equation}
\end{defi}
\begin{rem}
For the proofs below we conveniently rewrite condition~\eqref{ax:OM7}, since it makes our pasting diagram calculations easier to follow.
\begin{equation}
\tag{OM7'}\label{ax:OM7'}
\vcenter{\hbox{\xymatrix@!0@C=15mm{
&R\ot R\ar[rd]^-{1}&\\
R\ot R\ar[ru]^-{1}\ar[dd]_-{\bar{s}}\ar[rd]_-{1}\xtwocell[rddd]{}<>{^\bar{\alpha}}&&R\ot R\ar[dd]^-{\bar{w}}\\
&R\ot R\ar[dd]^-{\bar{v}}\ar[ru]^-{1}\xtwocell[rd]{}<>{^\bar{\gamma}}&\\
S\ot S\ar[rd]_-{\bar{t}}&&U\ot U\\
&T\ot T\ar[ru]_-{\bar{u}}&
}}}
=
\vcenter{\hbox{\xymatrix@!0@C=15mm{
&R\ot R\ar[rd]^-{1}\ar[dd]_-{\bar{s}}\xtwocell[rddd]{}<>{^\bar{\beta}}&\\
R\ot R\ar[ru]^-{1}\ar[dd]_-{\bar{s}}&&R\ot R\ar[dd]^-{\bar{w}}\\
&S\ot S\ar[rd]^-{\bar{x}}&\\
S\ot S\ar[rd]_-{\bar{t}}\ar[ru]_-{1}\xtwocell[rr]{}<>{^\bar{\zeta}}&&U\ot U\\
&T\ot T\ar[ru]_-{\bar{u}}&
}}}
\end{equation}
\end{rem}
\begin{defi}
For three enveloping monoidales $R\ot R$, $S\ot S$, and $T\ot T$ induced by bidualities $R\dashv R\ot$, $S\dashv S\ot$, and $T\dashv T\ot$, and an object $U$ in $\mathcal{M}$, define the category $\mathcal{A}(R;S;T;U)$ as the full subcategory of quadruples $(\widehat{\alpha},\widehat{\beta},\widehat{\gamma},\widehat{\zeta})$ in
\[
\mathcal{A}(R;S;T)\times\mathcal{A}(R;S;U)\times\mathcal{A}(R;T;U)\times\mathcal{A}(S;T;U)
\]
that satisfy the following condition,
\begin{equation}
\tag{A5}\label{ax:A5}
\vcenter{\hbox{\xymatrix@!0@C=19mm@R=10.5mm{
&US\ot SR\ot R\ar[rd]^-{\widehat{x}11}&\\
UT\ot TS\ot SR\ot R\qquad\ar[ru]^-{\widehat{u}1111}\ar[dd]_-{1111\widehat{s}}\ar[rd]_-{11\widehat{t}11}\xtwocell[rddd]{}<>{^11\widehat{\alpha}\quad}\xtwocell[rr]{}<>{^\widehat{\zeta}11\quad}&&UR\ot R\ar[dd]^-{\widehat{w}}\\
&UT\ot TR\ot R\ar[dd]^-{11\widehat{v}}\ar[ru]^-{\widehat{u}11}\xtwocell[rd]{}<>{^\widehat{\gamma}}&\\
UT\ot TS\ot S\ar[rd]_-{11\widehat{t}}&&U\\
&UT\ot T\ar[ru]_-{\widehat{u}}&
}}}
\quad=\hspace{-5mm}
\vcenter{\hbox{\xymatrix@!0@C=19mm@R=10.5mm{
&US\ot SR\ot R\ar[rd]^-{\widehat{x}11}\ar[dd]_-{11\widehat{s}}\xtwocell[rddd]{}<>{^\widehat{\beta}}&\\
UT\ot TS\ot SR\ot R\qquad\ar[ru]^-{\widehat{u}1111}\ar[dd]_-{1111\widehat{s}}\ar@{}[rd]|*=0[@]{\cong}&&UR\ot R\ar[dd]^-{\widehat{w}}\\
&US\ot S\ar[rd]^-{\widehat{x}}&\\
UT\ot TS\ot S\ar[rd]_-{11\widehat{t}}\ar[ru]_-{\widehat{u}11}\xtwocell[rr]{}<>{^\widehat{\zeta}}&&U\\
&UT\ot T\ar[ru]_-{\widehat{u}}&
}}}
\end{equation}
\end{defi}

We now proceed with the transposition step.

\begin{prop}\label{prop:Opmon3simplexTransposition}
For every four bidualities $R\ot R$, $S\ot S$, $T\ot T$ and $U\ot U$, there is an equivalence of categories
\[
\Opmon(R\ot R,S\ot S,T\ot T,U\ot U)\simeq\mathcal{A}(R;S;T;U)
\]
induced by the functor $P$ defined in Theorem~\ref{teo:Opmon2simplexTransposition} and given on a 3-simplex $(\bar{\alpha},\bar{\beta},\bar{\gamma},\bar{\zeta})$ by $(P\bar{\alpha},P\bar{\beta},P\bar{\gamma},P\bar{\zeta})$
\end{prop}
\begin{proof}

The categories in question are defined as full subcategories of 4-fold products of categories with the form $\Opmon(\_,\_,\_)$ and $\mathcal{A}(\_;\_;\_)$ respectively, and thus by Theorem~\ref{teo:Opmon2simplexTransposition} we have the equivalence of categories,
\begin{multline}\label{eq:P3}
\Opmon(R\ot R,S\ot S,T\ot T)\times\Opmon(R\ot R,S\ot S,U\ot U)\\
\times\Opmon(R\ot R,T\ot T,U\ot U)\times\Opmon(S\ot S,T\ot T,U\ot U)\\
\simeq\mathcal{A}(R;S;T)\times\mathcal{A}(R;S;U)\times\mathcal{A}(R;T;U)\times\mathcal{A}(S;T;U)
\end{multline}
We verify that for a quadruple $(\bar{\alpha},\bar{\beta},\bar{\gamma},\bar{\zeta})$ in $\Opmon(R\ot R,S\ot S,T\ot T,U\ot U)$ the quadruple $(P\bar{\alpha},P\bar{\beta},P\bar{\gamma},P\bar{\zeta})$ satisfies axiom~\eqref{ax:A5}.
{\footnotesize
\[
\vcenter{\hbox{\xymatrix@!0@C=15mm{
&&US\ot SR\ot R\ar[rd]^-{1\bar{x}11}&&\\
&UU\ot US\ot SR\ot R\ar[ru]^-{e11111}\ar[rd]|-{111\bar{x}11}&&UU\ot UR\ot R\ar[rd]^-{e111}&\\
UT\ot TS\ot SR\ot R\qquad\quad\ar[ru]^-{1\bar{u}1111}\ar[dd]_-{11111\bar{s}}\ar[rd]|-{111\bar{t}11}\xtwocell[rddd]{}<>{^111\bar{t}\bar{\alpha}\quad\ }\xtwocell[rr]{}<>{^1\bar{u}\bar{\zeta}11\quad\ \ }&&UU\ot UU\ot UR\ot R\ar[ru]|-{e11111}\ar[rd]|-{11e111}\ar@{}[uu]|*=0[@]{\cong}&&UR\ot R\ar[dd]^-{1\bar{w}}\\
&UT\ot TT\ot TR\ot R\ar[ru]|-{1\bar{u}\bar{u}11}\ar[rd]|-{11e111}\ar[dd]|-{11111\bar{v}}\xtwocell[rr]{}<>{^\qquad\qquad\ 1\bar{u}^211}&&UU\ot UR\ot R\ar[ru]|-{e111}\ar[dd]|-{111\bar{w}}\ar@{}[uu]|*=0[@]{\cong}&\\
UT\ot TS\ot SS\ot S\qquad\ar[rd]|-{111\bar{t}\bar{t}}\ar[dd]_-{1111e1}\xtwocell[rddd]{}<>{^111\bar{t}^2\quad}&&UT\ot TR\ot R\ar[dd]|-{111\bar{v}}\ar[ru]|-{1\bar{u}11}\xtwocell[rd]{}<>{^1\bar{u}\bar{\gamma}\quad}&&UU\ot U\ar[dd]^-{e1}\\
&UT\ot TT\ot TT\ot T\ar[rd]|-{11e111}\ar[dd]|-{1111e1}\ar@{}[ru]|*=0[@]{\cong}&&UU\ot UU\ot U\ar[ru]|-{e111}\ar[dd]|-{11e1}\ar@{}[ruuu]|*=0[@]{\cong}&\\
UT\ot TS\ot S\ar[rd]_-{111\bar{t}}&&UT\ot TT\ot T\ar[ru]|-{1\bar{u}\bar{u}}\ar[dd]|-{11e1}\xtwocell[rd]{}<>{^1\bar{u}^2\ }&&U\\
&UT\ot TT\ot T\ar[rd]_-{11e1}\ar@{}[ru]|*=0[@]{\cong}&&UU\ot U\ar[ru]_-{e1}\ar@{}[ruuu]|*=0[@]{\cong}&\\
&&UT\ot T\ar[ru]_-{1\bar{u}}&&
}}}
\]
\[
=
\vcenter{\hbox{\xymatrix@!0@C=15mm{
&&US\ot SR\ot R\ar[rd]^-{1\bar{x}11}&&\\
&UU\ot US\ot SR\ot R\ar[ru]^-{e11111}\ar[rd]|-{111\bar{x}11}&&UU\ot UR\ot R\ar[rd]^-{e111}&\\
UT\ot TS\ot SR\ot R\qquad\quad\ar[ru]^-{1\bar{u}1111}\ar[dd]_-{11111\bar{s}}\ar[rd]|-{111\bar{t}11}\xtwocell[rddd]{}<>{^111\bar{t}\bar{\alpha}\quad\ }\xtwocell[rr]{}<>{^1\bar{u}\bar{\zeta}11\quad\ \ }&&UU\ot UU\ot UR\ot R\ar[ru]|-{e11111}\ar[rd]|-{11e111}\ar[dd]|-{11111\bar{w}}\ar@{}[uu]|*=0[@]{\cong}&&UR\ot R\ar[dd]^-{1\bar{w}}\\
&UT\ot TT\ot TR\ot R\ar[ru]|-{1\bar{u}\bar{u}11}\ar[dd]|-{11111\bar{v}}\xtwocell[rd]{}<>{^1\bar{u}\bar{u}\bar{\gamma}\quad\ }&&UU\ot UR\ot R\ar[ru]|-{e111}\ar[dd]|-{111\bar{w}}\ar@{}[uu]|*=0[@]{\cong}&\\
UT\ot TS\ot SS\ot S\qquad\ar[rd]|-{111\bar{t}\bar{t}}\ar[dd]_-{1111e1}\xtwocell[rddd]{}<>{^111\bar{t}^2\quad}&&UU\ot UU\ot UU\ot U\ar[rd]|-{11e111}\ar@{}[ru]|*=0[@]{\cong}&&UU\ot U\ar[dd]^-{e1}\\
&UT\ot TT\ot TT\ot T\ar[rd]|-{11e111}\ar[dd]|-{1111e1}\ar[ru]|-{1\bar{u}\bar{u}\bar{u}}\xtwocell[rr]{}<>{^\qquad\qquad1\bar{u}^2\bar{u}}&&UU\ot UU\ot U\ar[ru]|-{e111}\ar[dd]|-{11e1}\ar@{}[ruuu]|*=0[@]{\cong}&\\
UT\ot TS\ot S\ar[rd]_-{111\bar{t}}&&UT\ot TT\ot T\ar[ru]|-{1\bar{u}\bar{u}}\ar[dd]|-{11e1}\xtwocell[rd]{}<>{^1\bar{u}^2\ }&&U\\
&UT\ot TT\ot T\ar[rd]_-{11e1}\ar@{}[ru]|*=0[@]{\cong}&&UU\ot U\ar[ru]_-{e1}\ar@{}[ruuu]|*=0[@]{\cong}&\\
&&UT\ot T\ar[ru]_-{1\bar{u}}&&
}}}
\]
\[
\stackrel{\eqref{ax:OM1}}{\stackrel{\eqref{ax:OM7'}}{=}}
\vcenter{\hbox{\xymatrix@!0@C=15mm{
&&US\ot SR\ot R\ar[rd]^-{1\bar{x}11}&&\\
&UU\ot US\ot SR\ot R\ar[ru]^-{e11111}\ar[rd]|-{111\bar{x}11}\ar[dd]|-{11111\bar{s}}\xtwocell[rddd]{}<>{^111\bar{x}\bar{\beta}\quad\ }&&UU\ot UR\ot R\ar[rd]^-{e111}\ar[dd]|-{111\bar{w}}&\\
UT\ot TS\ot SR\ot R\qquad\quad\ar[ru]^-{1\bar{u}1111}\ar[dd]_-{11111\bar{s}}&&UU\ot UU\ot UR\ot R\ar[ru]|-{e11111}\ar[dd]|-{11111\bar{w}}\ar@{}[uu]|*=0[@]{\cong}&&UR\ot R\ar[dd]^-{1\bar{w}}\\
&UU\ot US\ot SS\ot S\ar[rd]|-{111\bar{x}\bar{x}}&&UU\ot UU\ot U\ar[rd]|-{e111}\ar@{}[ru]|*=0[@]{\cong}&\\
UT\ot TS\ot SS\ot S\qquad\ar[ru]|-{1\bar{u}1111}\ar[rd]|-{111\bar{t}\bar{t}}\ar[dd]_-{1111e1}\ar@{}[ruuu]|*=0[@]{\cong}\xtwocell[rddd]{}<>{^111\bar{t}^2\quad}\xtwocell[rr]{}<>{^1\bar{u}\bar{\zeta}\bar{\zeta}\quad\ }&&UU\ot UU\ot UU\ot U\ar[ru]|-{e11111}\ar[rd]|-{11e111}\ar[dd]|-{1111e1}\ar@{}[ruuu]|*=0[@]{\cong}&&UU\ot U\ar[dd]^-{e1}\\
&UT\ot TT\ot TT\ot T\ar[dd]|-{1111e1}\ar[ru]|-{1\bar{u}\bar{u}\bar{u}}\xtwocell[rd]{}<>{^1\bar{u}\bar{u}^2\quad}&&UU\ot UU\ot U\ar[ru]|-{e111}\ar[dd]|-{11e1}\ar@{}[uu]|*=0[@]{\cong}&\\
UT\ot TS\ot S\ar[rd]_-{111\bar{t}}&&UU\ot UU\ot U\ar[rd]|-{11e1}\ar@{}[ru]|*=0[@]{\cong}&&U\\
&UT\ot TT\ot T\ar[rd]_-{11e1}\ar[ru]|-{1\bar{u}\bar{u}}\xtwocell[rr]{}<>{^1\bar{u}^2\quad}&&UU\ot U\ar[ru]_-{e1}\ar@{}[ruuu]|*=0[@]{\cong}&\\
&&UT\ot T\ar[ru]_-{1\bar{u}}&&
}}}
\]
\[
\stackrel{\eqref{ax:OM4}}{=}
\vcenter{\hbox{\xymatrix@!0@C=15mm{
&&US\ot SR\ot R\ar[rd]^-{1\bar{x}11}\ar[dd]|-{111\bar{s}}\xtwocell[rddd]{}<>{^1\bar{x}\bar{\beta}\quad}&&\\
&UU\ot US\ot SR\ot R\ar[ru]^-{e11111}\ar[dd]|-{11111\bar{s}}&&UU\ot UR\ot R\ar[rd]^-{e111}\ar[dd]|-{111\bar{w}}&\\
UT\ot TS\ot SR\ot R\qquad\quad\ar[ru]^-{1\bar{u}1111}\ar[dd]_-{11111\bar{s}}&&US\ot SS\ot S\ar[rd]|-{1\bar{x}\bar{x}}&&UR\ot R\ar[dd]^-{1\bar{w}}\\
&UU\ot US\ot SS\ot S\ar[ru]|-{e11111}\ar[rd]|-{111\bar{x}\bar{x}}\ar[dd]|-{1111e1}\ar@{}[ruuu]|*=0[@]{\cong}\xtwocell[rddd]{}<>{^111\bar{x}^2\quad}&&UU\ot UU\ot U\ar[rd]|-{e111}\ar[dd]|-{11e1}\ar@{}[ru]|*=0[@]{\cong}&\\
UT\ot TS\ot SS\ot S\qquad\ar[ru]|-{1\bar{u}1111}\ar[dd]_-{1111e1}\ar@{}[ruuu]|*=0[@]{\cong}&&UU\ot UU\ot UU\ot U\ar[ru]|-{e11111}\ar[dd]|-{1111e1}\ar@{}[uu]|*=0[@]{\cong}&&UU\ot U\ar[dd]^-{e1}\\
&UU\ot US\ot S\ar[rd]|-{111\bar{x}}&&UU\ot U\ar[rd]|-{e1}\ar@{}[ru]|*=0[@]{\cong}&\\
UT\ot TS\ot S\ar[rd]_-{111\bar{t}}\ar[ru]|-{1\bar{u}11}\ar@{}[ruuu]|*=0[@]{\cong}\xtwocell[rr]{}<>{^1\bar{u}\bar{\zeta}\quad}&&UU\ot UU\ot U\ar[ru]|-{e111}\ar[rd]|-{11e1}\ar@{}[ruuu]|*=0[@]{\cong}&&U\\
&UT\ot TT\ot T\ar[rd]_-{11e1}\ar[ru]|-{1\bar{u}\bar{u}}\xtwocell[rr]{}<>{^1\bar{u}^2\quad}&&UU\ot U\ar[ru]_-{e1}\ar@{}[uu]|*=0[@]{\cong}&\\
&&UT\ot T\ar[ru]_-{1\bar{u}}&&
}}}
\]
\[
=
\vcenter{\hbox{\xymatrix@!0@C=15mm{
&&US\ot SR\ot R\ar[rd]^-{1\bar{x}11}\ar[dd]|-{111\bar{s}}\xtwocell[rddd]{}<>{^1\bar{x}\bar{\beta}\quad}&&\\
&UU\ot US\ot SR\ot R\ar[ru]^-{e11111}\ar[dd]|-{11111\bar{s}}&&UU\ot UR\ot R\ar[rd]^-{e111}\ar[dd]|-{111\bar{w}}&\\
UT\ot TS\ot SR\ot R\qquad\quad\ar[ru]^-{1\bar{u}1111}\ar[dd]_-{11111\bar{s}}&&US\ot SS\ot S\ar[rd]|-{1\bar{x}\bar{x}}\ar[dd]|-{11e1}\xtwocell[rddd]{}<>{^1\bar{x}^2\quad}&&UR\ot R\ar[dd]^-{1\bar{w}}\\
&UU\ot US\ot SS\ot S\ar[ru]|-{e11111}\ar[dd]|-{1111e1}\ar@{}[ruuu]|*=0[@]{\cong}&&UU\ot UU\ot U\ar[rd]|-{e111}\ar[dd]|-{11e1}\ar@{}[ru]|*=0[@]{\cong}&\\
UT\ot TS\ot SS\ot S\qquad\ar[ru]|-{1\bar{u}1111}\ar[dd]_-{1111e1}\ar@{}[ruuu]|*=0[@]{\cong}&&US\ot S\ar[rd]|-{1\bar{x}}&&UU\ot U\ar[dd]^-{e1}\\
&UU\ot US\ot S\ar[ru]|-{e111}\ar[rd]|-{111\bar{x}}\ar@{}[ruuu]|*=0[@]{\cong}&&UU\ot U\ar[rd]|-{e1}\ar@{}[ru]|*=0[@]{\cong}&\\
UT\ot TS\ot S\ar[rd]_-{111\bar{t}}\ar[ru]|-{1\bar{u}11}\ar@{}[ruuu]|*=0[@]{\cong}\xtwocell[rr]{}<>{^1\bar{u}\bar{\zeta}\quad}&&UU\ot UU\ot U\ar[ru]|-{e111}\ar[rd]|-{11e1}\ar@{}[uu]|*=0[@]{\cong}&&U\\
&UT\ot TT\ot T\ar[rd]_-{11e1}\ar[ru]|-{1\bar{u}\bar{u}}\xtwocell[rr]{}<>{^1\bar{u}^2\quad}&&UU\ot U\ar[ru]_-{e1}\ar@{}[uu]|*=0[@]{\cong}&\\
&&UT\ot T\ar[ru]_-{1\bar{u}}&&
}}}
\]
}
And for a quadruple $(\widehat{\alpha},\widehat{\beta},\widehat{\gamma},\widehat{\zeta})$ in the category $\mathcal{A}(R;S;T;U)$ the quadruple $(P^{-1}\widehat{\alpha},P^{-1}\widehat{\beta},P^{-1}\widehat{\gamma},P^{-1}\widehat{\zeta})$ satisfies axiom~\eqref{ax:OM7}.
{\footnotesize
\[
\vcenter{\hbox{\xymatrix@!0@C=15mm{
&&&&\\
&U\ot UR\ot R\ar[rd]|-{11n11}\ar@/^6mm/[rr]^-{1}\xtwocell[rr]{}<>{^1\widehat{u}^011\quad\ }&&U\ot UR\ot R\ar@/^6mm/[rddd]^-{1}&\\
R\ot R\ar[ru]^-{n11}\ar[dd]_-{n11}\ar[rd]|-{n11}&&U\ot UT\ot TR\ot R\ar[ru]|-{1\widehat{u}11}\ar@/^6mm/[rddd]^-{1}&&\\
&T\ot TR\ot R\ar[ru]|-{n1111}\ar[dd]|-{11n11}\ar@/^6mm/[rddd]^-{1}\ar@{}[uu]|*=0[@]{\cong}\xtwocell[rddd]{}<>{^1\widehat{t}^011\quad\ }&&&\\
S\ot SR\ot R\ar[dd]_-{1\widehat{s}}\ar[rd]|-{n1111}\ar@{}[ru]|*=0[@]{\cong}&&&&U\ot UR\ot R\ar[dd]^-{1\widehat{w}}\\
&T\ot TS\ot SR\ot R\ar[dd]|-{111\widehat{s}}\ar[rd]|-{1\widehat{t}11}\xtwocell[rddd]{}<>{^1\widehat{\alpha}\ }&&U\ot UT\ot TR\ot R\ar[ru]|-{1\widehat{u}11}\ar[dd]|-{111\widehat{v}}\xtwocell[rd]{}<>{^1\widehat{\gamma}\ }&\\
S\ot S\ar[rd]_-{n11}\ar@{}[ru]|*=0[@]{\cong}&&T\ot TR\ot R\ar[dd]^-{1\widehat{v}}\ar[ru]|-{n1111}&&U\ot U\\
&T\ot TS\ot S\ar[rd]_-{1\widehat{t}}&&U\ot UT\ot T\ar[ru]_-{1\widehat{u}}&\\
&&T\ot T\ar[ru]_-{n11}\ar@{}[ruuu]|*=0[@]{\cong}&&
}}}
\]
\[
=
\vcenter{\hbox{\xymatrix@!0@C=15mm{
&&&&\\
&U\ot UR\ot R\ar[rd]|-{11n11}\ar@/^6mm/[rr]^-{1}\xtwocell[rr]{}<>{^1\widehat{u}^011\quad\ }&&U\ot UR\ot R\ar@/^6mm/[rddd]^-{1}&\\
R\ot R\ar[ru]^-{n11}\ar[dd]_-{n11}\ar[rd]|-{n11}&&U\ot UT\ot TR\ot R\ar[ru]|-{1\widehat{u}11}\ar@/^6mm/[rddd]^-{1}\ar[dd]|-{1111n11}\xtwocell[rddd]{}<>{^111\widehat{t}^011}&&\\
&T\ot TR\ot R\ar[ru]|-{n1111}\ar[dd]|-{11n11}\ar@{}[uu]|*=0[@]{\cong}&&&\\
S\ot SR\ot R\ar[dd]_-{1\widehat{s}}\ar[rd]|-{n1111}\ar@{}[ru]|*=0[@]{\cong}&&U\ot UT\ot TS\ot SR\ot R\ar[rd]|-{111\widehat{t}11}\ar[dd]|-{11111\widehat{s}}\xtwocell[rddd]{}<>{^1\widehat{\alpha}\ }&&U\ot UR\ot R\ar[dd]^-{1\widehat{w}}\\
&T\ot TS\ot SR\ot R\ar[ru]|-{n111111}\ar[dd]|-{111\widehat{s}}\ar@{}[ruuu]|*=0[@]{\cong}&&U\ot UT\ot TR\ot R\ar[ru]|-{1\widehat{u}11}\ar[dd]|-{111\widehat{v}}\xtwocell[rd]{}<>{^1\widehat{\gamma}\ }&\\
S\ot S\ar[rd]_-{n11}\ar@{}[ru]|*=0[@]{\cong}&&U\ot UT\ot TS\ot S\ar[rd]|-{111\widehat{t}}&&U\ot U\\
&T\ot TS\ot S\ar[rd]_-{1\widehat{t}}\ar[ru]|-{n1111}\ar@{}[ruuu]|*=0[@]{\cong}&&U\ot UT\ot T\ar[ru]_-{1\widehat{u}}&\\
&&T\ot T\ar[ru]_-{n11}\ar@{}[uu]|*=0[@]{\cong}&&
}}}
\]
\[
\stackrel{\eqref{ax:A3}}{=}
\vcenter{\hbox{\xymatrix@!0@C=15mm{
&&&&\\
&U\ot UR\ot R\ar[rd]|-{11n11}\ar[dd]|-{11n11}\ar@/^6mm/[rr]^-{1}\xtwocell[rr]{}<>{^1\widehat{u}^011\quad\ }&&U\ot UR\ot R\ar[dd]|-{11n11}\ar@/^6mm/[rddd]^-{1}\xtwocell[rddd]{}<>{^1\widehat{x}^011}&\\
R\ot R\ar[ru]^-{n11}\ar[dd]_-{n11}&&U\ot UT\ot TR\ot R\ar[ru]|-{1\widehat{u}11}\ar[dd]|-{1111n11}&&\\
&U\ot US\ot SR\ot R\ar[rd]|-{11n1111}\ar@{}[ru]|*=0[@]{\cong}&&U\ot US\ot SR\ot R\ar[rd]|-{1\widehat{x}11}&\\
S\ot SR\ot R\ar[dd]_-{1\widehat{s}}\ar[rd]|-{n1111}\ar[ru]|-{n1111}\ar@{}[ruuu]|*=0[@]{\cong}&&U\ot UT\ot TS\ot SR\ot R\ar[ru]|-{1\widehat{u}1111}\ar[rd]|-{111\widehat{t}11}\ar[dd]|-{11111\widehat{s}}\ar@{}[ruuu]|*=0[@]{\cong}\xtwocell[rddd]{}<>{^1\widehat{\alpha}\ }\xtwocell[rr]{}<>{^\qquad\qquad 1\widehat{\zeta}11}&&U\ot UR\ot R\ar[dd]^-{1\widehat{w}}\\
&T\ot TS\ot SR\ot R\ar[ru]|-{n111111}\ar[dd]|-{111\widehat{s}}\ar@{}[uu]|*=0[@]{\cong}&&U\ot UT\ot TR\ot R\ar[ru]|-{1\widehat{u}11}\ar[dd]|-{111\widehat{v}}\xtwocell[rd]{}<>{^1\widehat{\gamma}\ }&\\
S\ot S\ar[rd]_-{n11}\ar@{}[ru]|*=0[@]{\cong}&&U\ot UT\ot TS\ot S\ar[rd]|-{111\widehat{t}}&&U\ot U\\
&T\ot TS\ot S\ar[rd]_-{1\widehat{t}}\ar[ru]|-{n1111}\ar@{}[ruuu]|*=0[@]{\cong}&&U\ot UT\ot T\ar[ru]_-{1\widehat{u}}&\\
&&T\ot T\ar[ru]_-{n11}\ar@{}[uu]|*=0[@]{\cong}&&
}}}
\]
\[
\stackrel{\eqref{ax:A5}}{=}
\vcenter{\hbox{\xymatrix@!0@C=15mm{
&&&&\\
&U\ot UR\ot R\ar[dd]|-{11n11}\ar@/^6mm/[rr]^-{1}&&U\ot UR\ot R\ar[dd]|-{11n11}\ar@/^6mm/[rddd]^-{1}\xtwocell[rddd]{}<>{^1\widehat{x}^011}&\\
R\ot R\ar[ru]^-{n11}\ar[dd]_-{n11}&&&&\\
&U\ot US\ot SR\ot R\qquad\quad\ar[rd]|-{11n1111}\ar[dd]|-{111\widehat{s}}\ar@/^6mm/[rr]^-{1}\xtwocell[rr]{}<>{^1\widehat{u}^01111\qquad\ }&&U\ot US\ot SR\ot R\ar[dd]|-{111\widehat{s}}\xtwocell[rddd]{}<>{^1\widehat{\beta}}\ar[rd]|-{1\widehat{x}11}&\\
S\ot SR\ot R\ar[dd]_-{1\widehat{s}}\ar[ru]|-{n1111}\ar@{}[ruuu]|*=0[@]{\cong}&&U\ot UT\ot TS\ot SR\ot R\ar[ru]|-{1\widehat{u}1111}\ar[dd]|-{11111\widehat{s}}&&U\ot UR\ot R\ar[dd]^-{1\widehat{w}}\\
&U\ot US\ot S\ar[rd]|-{11n11}\ar@{}[ru]|*=0[@]{\cong}&&U\ot US\ot S\ar[rd]|-{1\widehat{x}}&\\
S\ot S\ar[rd]_-{n11}\ar[ru]|-{n11}\ar@{}[ruuu]|*=0[@]{\cong}&&U\ot UT\ot TS\ot S\ar[rd]|-{111\widehat{t}}\ar[ru]|-{1\widehat{u}11}\ar@{}[ruuu]|*=0[@]{\cong}\xtwocell[rr]{}<>{^1\widehat{\zeta}}&&U\ot U\\
&T\ot TS\ot S\ar[rd]_-{1\widehat{t}}\ar[ru]|-{n1111}\ar@{}[uu]|*=0[@]{\cong}&&U\ot UT\ot T\ar[ru]_-{1\widehat{u}}&\\
&&T\ot T\ar[ru]_-{n11}\ar@{}[uu]|*=0[@]{\cong}&&
}}}
\]
\[
=
\vcenter{\hbox{\xymatrix@!0@C=15mm{
&&&&\\
&U\ot UR\ot R\ar[dd]|-{11n11}\ar@/^6mm/[rr]^-{1}&&U\ot UR\ot R\ar[dd]|-{11n11}\ar@/^6mm/[rddd]^-{1}\xtwocell[rddd]{}<>{^1\widehat{x}^011}&\\
R\ot R\ar[ru]^-{n11}\ar[dd]_-{n11}&&&&\\
&U\ot US\ot SR\ot R\ar[dd]|-{111\widehat{s}}\ar@/^6mm/[rr]^-{1}&&U\ot US\ot SR\ot R\ar[dd]|-{111\widehat{s}}\xtwocell[rddd]{}<>{^1\widehat{\beta}}\ar[rd]|-{1\widehat{x}11}&\\
S\ot SR\ot R\ar[dd]_-{1\widehat{s}}\ar[ru]|-{n1111}\ar@{}[ruuu]|*=0[@]{\cong}&&&&U\ot UR\ot R\ar[dd]^-{1\widehat{w}}\\
&U\ot US\ot S\ar[rd]|-{11n11}\ar@/^6mm/[rr]^-{1}\xtwocell[rr]{}<>{^1\widehat{u}^011\quad\ }&&U\ot US\ot S\ar[rd]|-{1\widehat{x}}&\\
S\ot S\ar[rd]_-{n11}\ar[ru]|-{n11}\ar@{}[ruuu]|*=0[@]{\cong}&&U\ot UT\ot TS\ot S\ar[rd]|-{111\widehat{t}}\ar[ru]|-{1\widehat{u}11}\xtwocell[rr]{}<>{^1\widehat{\zeta}}&&U\ot U\\
&T\ot TS\ot S\ar[rd]_-{1\widehat{t}}\ar[ru]|-{n1111}\ar@{}[uu]|*=0[@]{\cong}&&U\ot UT\ot T\ar[ru]_-{1\widehat{u}}&\\
&&T\ot T\ar[ru]_-{n11}\ar@{}[uu]|*=0[@]{\cong}&&
}}}
\]
}
Therefore the equivalence~\ref{eq:P3} restricts its domain and codomain to the desired equivalence of categories.
\end{proof}

Now we continue with the opmonadicity step.

\begin{teo}\label{teo:3SimplexOpmonadicOplaxAct}
Let $\mathcal{M}$ be an opmonadic-friendly monoidal bicategory. For every four bidualities $R\dashv R\ot$, $S\dashv S\ot$, $T\dashv T\ot$ and $U\dashv U\ot$, and every three opmonadic adjunctions
\[
\vcenter{\hbox{\xymatrix{
R\ot\dtwocell_{i\ob}^{i\ot}{'\dashv}\\
I
}}}
\qquad
\vcenter{\hbox{\xymatrix{
S\ot\dtwocell_{j\ob}^{j\ot}{'\dashv}\\
I
}}}
\qquad
\vcenter{\hbox{\xymatrix{
T\ot\dtwocell_{k\ob}^{k\ot}{'\dashv}\\
I
}}}
\]
there is an equivalence of categories as shown.
\[
\mathcal{A}(R;S;T;U)\simeq\OplaxAct(R;S;T;U)
\]
\end{teo}
\begin{proof}

By Theorem~\ref{teo:2SimplexOpmonadicOplaxAct} there is an equivalence of categories as follows.
\begin{multline}\label{eq:Q3}
\mathcal{A}(R;S;T)\times\mathcal{A}(R;S;U)\times\mathcal{A}(R;T;U)\times\mathcal{A}(S;T;U)\\
\simeq\OplaxAct(R;S;T)\times\OplaxAct(R;S;U)\\
\times\OplaxAct(R;T;U)\times\OplaxAct(S;T;U)
\end{multline}
This equivalence restricts to $\mathcal{A}(R;S;T;U)$ because precomposition with the arrow
\[
\xymatrix{
1k\ob 1j\ob 1i\ob 1:UT\ot TS\ot SR\ot R\ar[r]&UTSR
}
\]
takes axiom~\eqref{ax:A5} to axiom~\eqref{ax:3SIM}. Moreover, for any quadruple $(\alpha,\beta,\gamma,\zeta)$ in the category $\OplaxAct(R;S;T;U)$, the quadruple $(\widehat{\alpha},\widehat{\beta},\widehat{\gamma},\widehat{\zeta})$ in the category  $\mathcal{A}(R;S;T)\times\mathcal{A}(R;S;U)\times\mathcal{A}(R;T;U)\times\mathcal{A}(S;T;U)$ induced by opmonadicity satisfies axiom~\eqref{ax:A5}. Indeed, if both sides of axiom~\eqref{ax:A5} for $(\widehat{\alpha},\widehat{\beta},\widehat{\gamma},\widehat{\zeta})$ are precomposed with the opmonadic arrow $1k\ob 1j\ob 1i\ob 1$, one ends up with each of the sides of axiom~\eqref{ax:3SIM} for $(\alpha,\beta,\gamma,\zeta)$, which are equal.

\end{proof}

With Proposition~\ref{prop:Opmon3simplexTransposition} and Theorem~\ref{teo:3SimplexOpmonadicOplaxAct} we deduce the following result.

\begin{cor}
Let $\mathcal{M}$ be an opmonadic-friendly monoidal bicategory. For every four bidualities $R\dashv R\ot$, $S\dashv S\ot$, $T\dashv T\ot$ and $U\dashv U\ot$, and every three opmonadic adjunctions
\[
\vcenter{\hbox{\xymatrix{
R\ot\dtwocell_{i\ob}^{i\ot}{'\dashv}\\
I
}}}
\qquad
\vcenter{\hbox{\xymatrix{
S\ot\dtwocell_{j\ob}^{j\ot}{'\dashv}\\
I
}}}
\qquad
\vcenter{\hbox{\xymatrix{
T\ot\dtwocell_{k\ob}^{k\ot}{'\dashv}\\
I
}}}
\]
there is an equivalence of categories as shown.
\[
\OplaxAct(R;S;T;U)\simeq\Opmon(R\ot R,S\ot S,T\ot T,U\ot U)
\]
\end{cor}
\begin{proof}
\[
\OplaxAct(R;S;T;U)\simeq\mathcal{A}(R;S;T;U)\simeq\Opmon(R\ot R,S\ot S,T\ot T,U\ot U)
\]
\end{proof}

\begin{cor}\label{cor:3simpEquiv}
Let $\mathcal{M}$ be an opmonadic-friendly monoidal bicategory such that every object $R$ has a chosen right bidual, and a chosen adjunction $i\dashv i^*$ whose opposite adjunction is opmonadic. There is an equivalence of categories of 3-simplices,
\[
\OplaxAct^{\mathrm{e}}_3\simeq \Opmon^{\mathrm{e}}_3
\]
which commutes with faces and degeneracies.
\[
\vcenter{\hbox{\xymatrix@!0@R=25mm@C=35mm{
\Opmon^{\mathrm{e}}_2\ar@<-6mm>[d]|-{\cs_0}\ar[d]|-{\cs_1}\ar@<6mm>[d]|-{\cs_2}\ar[r]_-{\simeq}&\OplaxAct^{\mathrm{e}}_2\ar@<-6mm>[d]|-{\cs_0}\ar[d]|-{\cs_1}\ar@<6mm>[d]|-{\cs_2}\\
\Opmon^{\mathrm{e}}_3\ar@<9mm>[u]|-{\partial_0}\ar@<3mm>[u]|-{\partial_1}\ar@<-3mm>[u]|-{\partial_2}\ar@<-9mm>[u]|-{\partial_3}\ar[r]^-{\simeq}&\OplaxAct^{\mathrm{e}}_3\ar@<9mm>[u]|-{\partial_0}\ar@<3mm>[u]|-{\partial_1}\ar@<-3mm>[u]|-{d_2}\ar@<-9mm>[u]|-{d_3}
}}}
\]\qed
\end{cor}

And finally we put together \eqref{eq:FacesDegeneracies01}, Corollary~\ref{cor:2simpEquiv}, and Corollary~\ref{cor:3simpEquiv} which form a pseudo-simplicial morphism of simplicial objects in $\Cat$, which as we mentioned in Remark~\ref{rem:degeneracies} may be strictified by making the right choices in the definition of the degeneracy functors of $\OplaxAct(\mathcal{M})$ to get an actual simplicial map.

\begin{teo}
Let $\mathcal{M}$ be an opmonadic-friendly monoidal bicategory such that every object $R$ has a chosen right bidual, and a chosen adjunction $i\dashv i^*$ whose opposite adjunction is opmonadic. There is a weak equivalence of simplicial objects in $\Cat$ as shown.
\[
\OplaxAct^{\mathrm{e}}(\mathcal{M})\simeq N(\Opmon^{\mathrm{e}}(\mathcal{M}))
\]\qed
\end{teo}
\bibliographystyle{alpha}
\bibliography{CoalgebroidsAndOplaxActions}
\end{document}